\documentclass[11pt]{amsart}
\usepackage{amssymb}
\usepackage{graphicx,tikz}
\usetikzlibrary{arrows}

\graphicspath{{.}{./eps/}}

\newtheorem{theorem}{Theorem}[section]

\newtheorem{lemma}[theorem]{Lemma}

\newtheorem{corollary}[theorem]{Corollary}
\newtheorem{definition}[theorem]{Definition}
\newtheorem{notation}[theorem]{Notation}
\newtheorem{observation}[theorem]{Observation}

\newcommand{\ra}{\rightarrow}

\newcommand\Z{\mathbb Z}
\newcommand\R{\mathbb R}
\newcommand\N{\mathbb N}
\newcommand\Q{\mathbb Q}
\newcommand\p{p}
\newcommand\h{h}
\newcommand\mx{\mathrm{Max}}
\newcommand\mn{\mathrm{Min}}

\newcommand{\seanjen}{MR2052598}
\newcommand{\seanjennotac}{MR2016653}

\newcommand{\hmtcfcn}{MR1804409}
\newcommand{\cannon}{MR877210}
\newcommand{\cfp}{MR1426438}
\newcommand{\BG}{MR752825}
\newcommand{\murraysusan}{MR2126733}
\newcommand{\thiel}{MR1237823}
\newcommand{\ilyamac}{MR1879521}
\newcommand{\poenaruac}{MR1259521}
\newcommand{\poenarugeom}{MR1152227}
\newcommand{\echlpt}{MR1161694}
\newcommand{\belkbux}{MR2139683}
\newcommand{\millershapiro}{MR1644415}
\newcommand{\MTs}{MR1390045}
\newcommand{\riley}{MR1931373}
\newcommand{\horaksteintaback}{MR2596882}
\newcommand{\fordham}{MR1998934}
\newcommand{\guba}{MR2207020}

\title{Tame combing and almost convexity conditions}

\author[Sean Cleary]{Sean Cleary}
\address{Department of Mathematics,
The City College of New York, City University of New York, New
York, NY 10031}
\email{cleary@sci.ccny.cuny.edu}

\author[Susan Hermiller]{Susan Hermiller}
\address{Department of Mathematics, University of Nebraska, Lincoln, NE}
\email{smh@math.unl.edu}

\author[Melanie Stein] {Melanie Stein}
\address{Department of Mathematics, Trinity College, Hartford, CT 06106}
\email{melanie.stein@trincoll.edu}

\author[Jennifer Taback] {Jennifer Taback}
\address{Department of Mathematics, Bowdoin College, Brunswick, ME 04011}
\email{jtaback@bowdoin.edu}

\thanks{The first author acknowledges support from National Science
Foundation grant DMS-0811002.}
\thanks{The fourth author acknowledges support from
National Science Foundation grant DMS-0604645. The third and fourth
authors acknowledge support from a Bowdoin College Faculty Research Award.}

\date{\today}

\begin{document}

\begin{abstract}

We give the first examples of groups which
admit a tame combing with linear radial
tameness function with respect to any choice of finite presentation, but
which are not minimally
almost convex on a standard generating set.
 Namely, we explicitly construct such combings for Thompson's group $F$
and the Baumslag-Solitar groups $BS(1,p)$ with $p \ge 3$.
In order to make this construction for Thompson's group $F$, we
significantly expand the understanding of the Cayley complex of
this group with respect to the standard finite presentation.  In
particular we describe a quasigeodesic set of normal forms and
combinatorially classify the arrangements of 2-cells
adjacent to edges that do not lie on normal form paths.

\end{abstract}

\maketitle

\section{Introduction}
\label{sec:intro}

This paper has two goals: to study the relationships between
the hierarchies of convexity conditions and tame combing
conditions on a Cayley complex corresponding to a given group,
 and to significantly expand the understanding of the Cayley
complex of Thompson's group $F$ with respect to the standard
finite presentation with two generators and two relators.

Several notions of almost convexity for groups
have been developed in geometric group theory,
from the most restrictive property defined by
Cannon \cite{\cannon} to the weakest
notion of minimal almost convexity introduced by Kapovich \cite{\ilyamac}.
For a group $G$ with finite generating set $A$,
almost convexity conditions for different classes of functions measure,
in terms of the given function, how close balls in the Cayley graph
for $(G,A)$ are to being convex sets (see Section~\ref{sub:ac}
for the formal definition).  Results
of Thiel \cite{\thiel} and Elder and Hermiller \cite{\murraysusan},
respectively, show that Cannon's almost convexity and
minimal almost convexity, respectively, are not quasi-isometry
invariants.

Mihalik and Tschantz \cite{\MTs} introduced the notion of
a tame 1-combing of a 2-complex, and in particular of
the Cayley complex of a group presentation, in the context of
studying properties of 3-manifolds.
Hermiller and Meier \cite{\hmtcfcn}
refined the definition
of tame combing to differentiate between types of tameness
functions, analogously to almost convexity conditions. For a group
$G$ with finite presentation ${\mathcal P}$, intuitively the
radial tameness function measures the relationship,
for any loop, between the size of the ball in the
Cayley complex containing the loop and the
size of the ball needed to contain a disk filling in that loop
(see Section~\ref{sub:tc} for the formal definitions).
Hermiller and Meier  \cite{\hmtcfcn} showed that the
advantage of studying balls in a Cayley complex from the viewpoint of
tame combings and radial tameness functions is that the classes of
tame combable groups
are, up to Lipschitz equivalence
of radial tameness functions (for example, linear functions or
exponential functions), invariant under quasi-isometry, and hence under change of presentation.
In the same paper they also showed that
groups which are almost convex with respect to
several classes of functions are contained in the quasi-isometry invariant
class of groups admitting a 1-combing with a linear
radial tameness function.  In Section~\ref{sub:hier}, we
give a more complete discussion of the hierarchies of
almost convexity conditions and
tame combing functions, and their interconnections.

In seeking to further understand
the correspondence
between these two hierarchies,
we use geometric information from the Cayley
complex to construct tame 1-combings with linear
tameness functions for two groups: Thompson's group $F$
and the solvable Baumslag-Solitar group $BS(1,p)$ for $p \geq 3$.
Cleary and Taback \cite{\seanjennotac} showed that Thompson's group $F$ 
is not almost convex with respect to the standard generating set, and 
Belk and Bux \cite{\belkbux} showed that Thompson's
group $F$ is not minimally almost convex with respect to the standard
finite generating set.  Elder and Hermiller \cite{\murraysusan}
showed that the groups $BS(1,p)$ for $p \geq 7$
are also not minimally almost convex with respect to their usual
generating set; moreover, Miller and Shapiro \cite{\millershapiro} showed
shown that the group $BS(1,p)$ does not satisfy
Cannon's almost convexity condition for any generating set.
Combining these
then provides the
first examples of groups
which admit a combing with a linear radial tameness function
(with respect to any choice of finite presentation),
but  which are not minimally almost convex on a particular
finite generating set.
These also provide the first examples of groups
which admit a combing with a linear radial tameness function
but which do not satisfy Cannon's almost convexity condition with
respect to every finite presentation.
In the case of Thompson's group $F$,
our proof also gives significant new insight into the Cayley
complex of this group.

Despite the prevalence of $F$ in geometric group theory, a
detailed understanding of the Cayley complex for the standard
finite presentation
$$
\langle x_0,x_1 |
[x_0x_1^{-1},x_0^{-1}x_1x_0],[x_0x_1^{-1},x_0^{-2}x_1x_0^2] \rangle
$$
has been elusive.
In an intricate analysis, Guba \cite{\guba} 
showed that Thompson's group $F$ has a quadratic isoperimetric
function, but
it is as yet unknown if Thompson's group is automatic,
nor even if it is asynchronously combable. The tame 1-combing we construct for $F$ utilizes the
nested traversal paths defined by Cleary and Taback in \cite{\seanjen}.
We show that these paths yield a set of quasigeodesic normal forms for the group.
Extending these paths to a tame 1-combing of the Cayley
2-complex $X$ for the
presentation above then requires a careful, detailed classification
of the edges and 2-cells of $X$.
In particular
we analyze which edges do not lie on
these normal form paths and for each
such edge,
we characterize which 2-cells of $X$ adjacent to that edge
have the property that their other
boundary edges lie ``closer'' to the identity vertex $\epsilon$.  Our analysis and measure of
closeness to the identity use
combinatorial properties computed from the group elements
labeling the vertices adjacent to that edge.

The paper is organized as follows. In Section 2, we provide an
overview of the development of the notions of convexity and
combings for groups, and the relations between them. In Section 3,
we provide a brief introduction to Thompson's group $F$, and
define the set of normal forms which will be used in the construction
of the tame 1-combing.  We also show that this set of normal
forms satisfies a quasi-geodesic property.
In Section 4, we construct the 1-combing of $F$, and in Section 5 we
show that this combing satisfies a linear radial tameness function, as stated
in Theorem~\ref{thm:flintame}. In Section 6 we show that
$G=BS(1,p)$ with $p \ge 3$ has a 1-combing which satisfies a linear radial tameness
function, proving Theorem~\ref{thm:bsp}. Finally, Section 7 is
devoted to the proof of Theorem~\ref{thm:bscoeff1}, verifying that
although $G=BS(1,p)$ with $p \ge 8$ and the standard presentation admits a 
1-combing with a linear radial tameness
function, this tameness function must have a multiplicative constant greater than 1.

\section{Convexity and combings for groups}
\label{sec:convexityandcombings}

\subsection{Almost convexity conditions on Cayley graphs}\label{sub:ac}

For a group $G$ with a finite inverse-closed generating set $A$, we
let $\Gamma(G,A)$ denote the Cayley graph of $G$ with respect
to $A$, and let $d_A$ denote the word metric with respect to this
generating set.  The
pair $(G,A)$  satisfies the almost convexity condition $AC_{f}$
for a function $f:\N \rightarrow \R_+$ if there is an
$r_0 \in \N$  such that for
every two points $a,b$ in the sphere $S(r)$ (centered at the identity)
with $d_A(a,b) \leq 2$ and every
natural number $r>r_0$, there is a path inside the ball $B(r)$ from
$a$ to $b$ of length no more than $f(r)$.

Every group satisfies the almost convexity condition $AC_f$
for the function $f(r)=2r$, as two
points in the ball of radius $r$ can always be connected by a path
of length $2r$ which remains inside $B(r)$, simply by going to the identity and returning outward.
Thus the weakest nontrivial almost convexity
condition for a pair $(G,A)$  is $AC_f$ for the
function $f(r) = 2r-1$.  Kapovich \cite{\ilyamac} and Riley \cite{\riley}
have shown that this minimal almost convexity condition (MAC)
implies finite presentation of the group and the existence
of an algorithm for constructing the Cayley graph.

At the other end of the spectrum, $(G,A)$ is almost convex (AC) in the sense of
Cannon \cite{\cannon} if it satisfies $AC_f$ for a constant function $f$.
Between the constant
function and $f(r) = 2r-1$, there are a number of other possible functions
which give rise to a range of almost convexity conditions.
For example, Po\'enaru  \cite{\poenarugeom,\poenaruac}
 studied the property $AC_f$
for  sublinear functions $f$.

\subsection{Tame combings of Cayley complexes}\label{sub:tc}

Let $G=\langle A \mid
R \rangle$ be a finitely presented group, with $A$ an inverse-closed generating set.
Let $X$ denote the
Cayley complex corresponding to this presentation, with 0- and
1-skeletons $X^0=G$ and $X^1 =\Gamma(G,A)$; that is, $X^1$ is the
Cayley graph with respect to this presentation.

In order to have a notion of a ball centered at the identity
$\epsilon$ in the 2-complex $X$,
the notion of distance
between the vertices of a Cayley graph is extended
to a notion of level on the entire complex.
The following
definition is equivalent to that in \cite{\hmtcfcn}.

\begin{definition}\label{def:level}
\begin{enumerate}
\item If $g$ is a vertex in $X^0$, the {\em level} $\mathrm{lev}(g)$ is
defined to be the word length $l_A(g)$ with respect to the generating set $A$.
\item If $x\in X^1-X^0$, then $x$ is in the interior of some
edge with vertices $g,h \in X^0$. Then let
$$\mathrm{lev}(x):=\frac{\mathrm{lev}(g)+ \mathrm{lev}(h)}{2} + \frac{1}{4}$$
\item If $x\in X-X^1$,
then $x$ is in the interior of some 2-cell with vertices $g_1,
g_2,  \ldots, g_n$ along the boundary, and
$$\mathrm{lev}(x):=\frac{\mathrm{lev}(g_1) + \mathrm{lev}(g_2)+ \cdots +
\mathrm{lev}(g_n)}{n} + \frac{1}{4} +
\frac{1}{c}$$ where if $R=\{r_1,r_2, \ldots, r_k\}$ is the set of relators, and for each
$1 \leq i \leq k$, $n_i$ is the number of letters in the relator
$r_i$, then $c:=4n_1n_2 \cdots n_k+1$.
\end{enumerate}
\end{definition}

Intuitively, a $0$-combing of a group $G$ with generating set $A$
is a choice of path in the Cayley graph $\Gamma(G,A)$ from the
identity to each group element.  To obtain a $1$-combing for
$(G,\langle A \mid R \rangle)$, a $0$-combing is extended
continuously through the $1$-skeleton of the Cayley complex.
Viewing the ball of radius $q$ in $X$ as the set of  points of
level at most $q$, a radial tameness function $\rho:{\mathbb Q}
\ra {\mathbb R}_+$ for a $1$-combing ensures that once a combing
path leaves the ball of radius $\rho(q)$, it never returns to the
ball of radius $q$.

\begin{definition} The pair $(G,\langle A \mid R \rangle)$
satisfies the tame combing condition $TC_\rho$ for a function
$\rho:{\mathbb Q} \ra {\mathbb R}_+$ if there is a continuous
function $\Psi:X^1 \times [0,1] \ra X$ satisfying:
\begin{enumerate}
\item For all $x \in X^1$, $\Psi(x,0)=\epsilon$ and $\Psi(x,1)=x$,
\item $\Psi(X^0 \times [0,1]) \subseteq X^1$, and \item For all $x
\in X^1$, $0 \le s<t \le 1$, and $q \in {\mathbb Q}$, if
$\mathrm{lev}(\Psi(x,s))>\rho(q)$, then $\mathrm{lev}(\Psi(x,t))>q$.
\end{enumerate}
The function $\Psi$ is a {\em 1-combing} of $X$, and $\rho$ is
a {\em radial tameness function} for $\Psi$.
\end{definition}

A continuous function $\Psi:X^0 \times [0,1] \ra X^1$
with  $\Psi(x,0)=\epsilon$ and $\Psi(x,1)=x$ for all $x \in X^0$
is called a {\it 0-combing} for the pair $(G,A)$.  The restriction
of a 1-combing to the vertices of $X$ is a 0-combing.

In \cite{\hmtcfcn}, Hermiller and Meier show that the condition
$TC_\rho$
is quasi-isometry invariant, and thus is independent of the choice of presentation for the group, up to
a Lipschitz equivalence on the radial tameness functions. Hence it
makes sense to define the class of groups admitting a tame
1-combing with a linear radial
tameness function, and also classes with polynomial and exponential
radial tameness functions.

\subsection{Hierarchies of convexity and combing functions}\label{sub:hier}
A pair $(G,\langle A \mid R \rangle)$ may satisfy a variety of
almost convexity and tame combing
conditions.  In Figure \ref{fig:bigpicture} below, we illustrate
what is known about the relevant relationships between the different
classes of almost convexity functions, and the different classes of
possible radial tameness functions which arise from tame
1-combings. In particular, all descending vertical arrows in
Figure~\ref{fig:bigpicture} follow immediately from
the definitions.

\begin{figure}[ht!]
\begin{center}
\begin{tikzpicture}
\path
node at (0,0) [shape=rectangle,draw] (a2) {Minimally almost convex with $f(n) = 2n-1$}
node at (0,4) [shape=rectangle,draw] (a4) {Poenaru almost convex with $f(n)$ sublinear}
node at (0,8) [shape=rectangle,draw] (a5) {Almost convex with $f(n) =C$}
node at (0,10) [shape=rectangle] (a6) {\textbf{Hierarchy of almost}}
node at (0,9.5) [shape=rectangle] (a6) {\textbf{convexity conditions}}
node at (8,-2) [shape=rectangle,draw,double] (b6) {$\rho(q)$ recursive}
node at (8,0) [shape=rectangle,draw,double] (b5) {$\rho(q)$ exponential}
node at (8,2) [shape=rectangle,draw,double] (b3) {$\rho(q)$ linear}
node at (8,5) [shape=rectangle,draw] (b4) {$\rho(q) = q+C$}
node at (8,8) [shape=rectangle,draw] (b1) {$\rho(q) = q$}
node at (8,10) [shape=rectangle] (b8) {\textbf{Hierarchy of radial }}
node at (8,9.5) [shape=rectangle] (b9) {\textbf{tameness functions}}

node at (5,8.5) [shape=rectangle] (x1) {\cite{\hmtcfcn}}
node at (.9,2) [shape=rectangle] (qq) {\cite{\murraysusan}}
node at (5.25,3.6) [shape=rectangle] (qqq) {\cite{\hmtcfcn},\cite{\riley}}
node at (8.75,3.5) [shape=rectangle] (xxx) {*}
node at (4.2,1.4) [shape=rectangle] (ww) {\cite{\belkbux}, \cite{\murraysusan}, *};

\draw[->,thick,double] (8.25,2.45)--(8.25,4.55);

\draw[->,thick,double] (8,7.55)--(8,5.45);
\draw[->,thick,double] (7.75,4.55)--(7.75,2.45);
\draw[->,thick,double] (8,1.55)--(8,0.45);
\draw[->,thick,double] (8,-0.55)--(8,-1.45);

\draw[<->,thick,double] (2.8,8)--(7,8);

\draw[<-,thick,double] (-.25,.45)--(-.25,3.55);
\draw[<-,thick,double] (.25,3.55)--(.25,.45);
\draw[->,thick,double] (0,7.55)--(0,4.45);
\draw[very thick] (0,1.75)--(.5,2.25);
\draw[very thick] (8,3.25)--(8.5,3.75);

\draw[->,thick,double] (4,3.5)--(6.75,2.5);
\draw[->,thick,double] (6.75,1.75)--(3.75,0.5);
\draw[very thick] (5,1.25)--(5.5,0.95);

\end{tikzpicture}
\caption{The relationships between the hierarchies of convexity
conditions and degrees of radial tameness functions
for a pair $(G,\langle A \mid R \rangle)$.  A slash across an
arrow indicates that it is known that there exists a
counterexample to the implication in that direction.
Numbers in brackets are bibliographic references;
the two instances marked by a $*$ are established in this paper.
The tameness properties contained in double boxes
are independent of the choice of the finite presentation for $G$.}
\label{fig:bigpicture}
\end{center}
\end{figure}
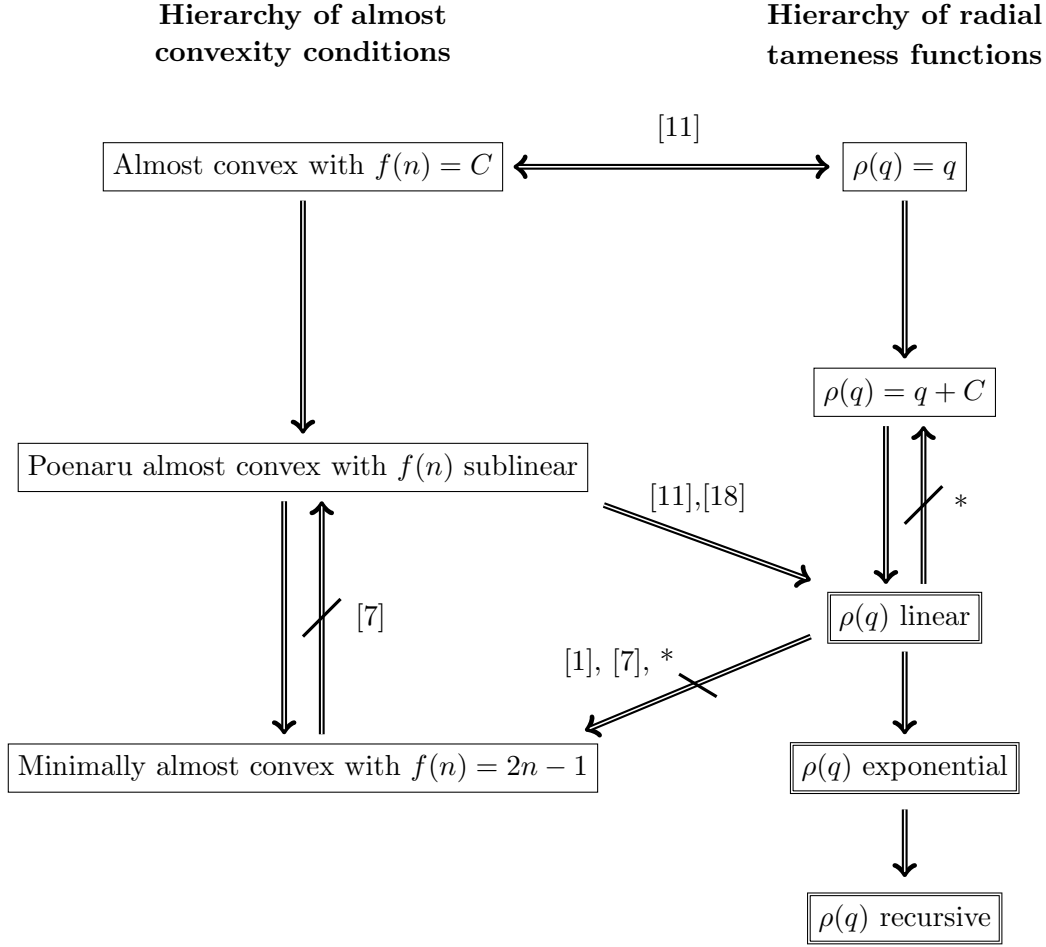

These two chains of conditions are tied together at the base by
the results of Hermiller and Meier \cite{\hmtcfcn} which show
that a pair $(G,A)$ is almost convex if and only if there is a set
of defining relations $R$  such that the pair $(G,\langle A \mid R
\rangle)$ admits a 1-combing satisfying the radial tameness
function $\rho(q)=q$.

In Theorem D of \cite{\hmtcfcn}, Hermiller
and Meier showed that the property $AC_f$ with $f$ sublinear,
 together with a linear isodiametric
function (a combination of properties motivated by
work of Po\'enaru in \cite{\poenarugeom}),
imply the existence of a 1-combing with a linear
radial tameness function.  For a pair $(G,A)$ satisfying $AC_f$ for any function $f:\N \rightarrow \R_+$ such that $f(n) \le n-2$ for all $n$, it follows from Riley \cite[Equation~3.2]{\riley} and induction that for all $n \ge 2r_0+2$, we have $Diam(n) \le n +D$, where $D=Diam(2r_0+2)$ is a constant. Hence the property $AC_f$ with $f$ sublinear implies
a linear isodiametric function, and so this extra assumption
was redundant.  As a consequence, it follows that $AC_f$
with $f$ sublinear
implies the condition $TC_\rho$ with $\rho$ linear.

Tantalizing questions to consider, given these results,
involve the potential connections between weaker
notions of almost convexity and radial tameness functions.
As yet, there are few examples known, other than for groups
satisfying the condition $AC_f$ with $f$ sublinear,
of groups with 1-combings admitting restricted tameness functions.

In this paper the results of Theorems~\ref{thm:flintame} and \ref{thm:bsp}
show that the quasi-isometry independent
class $TC_{\mathrm{linear}}$ of groups with a 1-combing satisfying a linear
radial tameness function contains groups which are not even minimally
almost convex for some particular generating set, giving the
diagonal non-implication in Figure~\ref{fig:bigpicture}.
However, this leaves open the question
of whether every group in $TC_{\mathrm{linear}}$,
and in particular whether $F$ and $BS(1,p)$ with $p \ge 3$, might have
some generating set with respect to which it is minimally
almost convex.

Other intriguing questions involve the possibility
of upward implications in either of the two hierarchies.
In Theorem~\ref{thm:bscoeff1} we show the vertical non-implication
for tameness functions drawn in Figure~\ref{fig:bigpicture}.
For almost convexity, Elder and Hermiller \cite{\murraysusan}
have exhibited a pair $(G,A)$ which is minimally almost convex
but  does not satisfy the condition $AC_f$ with $f$ sublinear.
It is still an open question whether there can be a pair
$(G,A)$ satisfying the Po\'enaru $AC_f$ condition with $f$ sublinear that does not also satisfy Cannon's AC property.


\section{An introduction to Thompson's group $F$}
We present a brief introduction to Thompson's group $F$ and refer
the reader to \cite{\cfp} for a more detailed discussion, with historical background.
In addition, in Section \ref{NTNF's} we define a
the set of normal forms for $F$ which will be used in
our construction of a 1-combing.

Thompson's
group $F$ has a standard infinite presentation:
$$\langle x_k, \ k \geq 0 | x_i^{-1}x_jx_i = x_{j+1} \ \text{ if
}\ i<j \rangle.$$
The elements $x_0$ and $x_1$ are sufficient to
generate the entire group,
 since powers of $x_0$ conjugate $x_1$ to $x_i$ for $i \geq 2$.
Only two relators are required for a presentation with
 the generating set $A:=\{x_0, x_1\}$, resulting in the
finite presentation for $F$:
$$\langle x_0,x_1 |
[x_0x_1^{-1},x_0^{-1}x_1x_0],[x_0x_1^{-1},x_0^{-2}x_1x_0^2]
\rangle.$$
This is the most commonly used finite generating set and presentation
for Thompson's group $F$, and
in this paper we will build the 1-combing for $F$ using the
Cayley complex for this presentation.

With respect to the infinite presentation given above, each element $w \in F$
can be written in normal form as $$w =x_{i_1}^{r_1}
x_{i_2}^{r_2}\ldots x_{i_k}^{r_k} x_{j_l}^{-s_l} \ldots
x_{j_2}^{-s_2} x_{j_1}^{-s_1} $$ with $r_i, s_i >0$, $0 \leq i_1<i_2
\ldots < i_k$ and $0 \leq j_1<j_2 \ldots < j_l$. Furthermore, we require that when both $x_i$ and $x_i^{-1}$
occur, so does $x_{i+1}$ or $x_{i+1}^{-1}$, as discussed by Brown
and Geoghegan \cite{\BG}. We will use the term \emph{infinite
normal form} to mean this normal form, and write
$w=w_pw_n$ where $w_p$ is the maximal subword of this normal form with
positive exponents, and $w_n$ is the maximal subword with negative exponents.

Elements of $F$ can be viewed combinatorially as pairs of finite
binary rooted trees, each with the same number of edges and vertices,
called {\em tree pair diagrams}.
Let $T$ be a finite rooted binary tree.
We define a {\em caret} of $T$ to be a vertex of the
tree together with two downward oriented edges, which we refer to as
the left and right edges of the caret. The {\em right (respectively
left) child} of a caret $c$ is defined to be a caret which is
attached to the right (resp. left) edge of $c$.  If a caret $c$ does
not have a right (resp. left) child, we call the right (resp. left)
leaf of $c$ {\em exposed}.  The caret itself is {\em exposed} if
both of its leaves are also leaves of
the tree; that is, the caret has no children.

For a given tree $T$, let $N(T)$ denote the number of carets in $T$.
We number the carets from $1$ through $N(T)$ in infix order.
The infix ordering is carried out by numbering the left descendants (the left child and all descendants of the left child)
of a caret $c$ before numbering $c$, and the right descendants of $c$
afterward.
We use the
infix numbers as names for the carets, and the statement $p<q$ for
two carets $p$ and $q$ simply expresses the relationship between their
infix numbers. In a tree pair diagram $(T,S)$, we refer to the pair of carets with infix number $p$,
one in each tree, as the {\em caret pair $p$}.

The left (resp. right) side of a binary rooted tree $T$ consists of the left (resp. right) edge of the root caret, together with the left (resp. right) side of the subtree consisting of all left (resp. right) descendants of the root caret.
A caret in a tree $T$ is said to be a {\em right} (resp. {\em left}) caret if
one of its edges lies on the right (resp. left) side of $T$.  The
root caret can be considered either left or right.  All other carets
are called {\em interior} carets.
We also number the leaves of the tree $T$ from left to right, from $0$
through $N(T)$.

An element $w \in F$ is represented by an equivalence class of
tree pair diagrams, among which there is a unique reduced tree
pair diagram.  We say that a pair of trees is {\em unreduced} if,
when the leaves are numbered from $0$ through $N(T)$, there is a
caret in both trees with two exposed leaves bearing the same leaf
numbers.  If so, we remove that pair of carets, and renumber the
carets in both trees. Repeating this process until there are no
such pairs produces the unique {\em reduced} tree pair diagram
representing $w$.

The equivalence of these two interpretations of Thompson's group is
given using the infinite normal form for elements with respect to the
standard infinite presentation, and the concept of leaf exponent. In
a single tree $T$ whose leaves are numbered from left to right
beginning with $0$, the {\em leaf exponent} $E_T(k)$ of leaf number
$k$ is defined to be the integral length of the longest path of left
edges from leaf $k$ which does not reach the right side of the tree.

Given the reduced tree pair diagram $(T,S)$ representing $w \in
F$, compute the leaf exponents $E_T(k)$ for all leaves $k$ in $T$,
numbered $0$ through $n=N(T)=N(S)$.   The negative part of the infinite
normal form for $w$ is then $x_n^{-E_T(n)} x_{n-1}^{-E_T(n-1)} \cdots
x_1^{-E_T(1)} x_0^{-E_T(0)}$. We compute the exponents $E_S(k)$ for the
leaves of the tree $S$ and thus obtain the positive part of the
infinite normal form as $x_0^{E_S(0)} x_{1}^{E_S(1)} \cdots
x_n^{E_S(n)}$.
 Many of these exponents will be $0$, and after deleting these,
we can index the remaining terms to correspond to the infinite
normal form given above, following \cite{\cfp}.  As a result of
this process, we often denote the unique reduced tree pair diagram
for $w$ by $w=(T_-(w),T_+(w))$, since the first tree in the pair
determines the terms in the infinite normal form with negative
exponents, and the second tree determines those terms with
positive exponents. We refer to $T_-(w)$ as the negative tree in
the pair, and $T_+(w)$ as the positive tree.

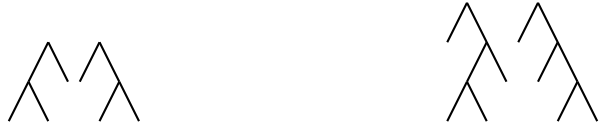
\begin{figure}[ht!]
\begin{center}
\begin{tikzpicture}[scale=.35,thick]
\coordinate
child {child child}
child;
\end{tikzpicture}
\begin{tikzpicture}[scale=.35,thick]
\coordinate
child
child{child child};
\end{tikzpicture}
\hspace{1.5in}
\begin{tikzpicture}[scale=.35,thick]
\coordinate
child
child{child {child child}child};
\end{tikzpicture}
\begin{tikzpicture}[scale=.35,thick]
\coordinate
child
child{child child{child child}};
\end{tikzpicture}
\caption{The reduced tree pair diagrams representing (respectively)
$x_0^{-1}$ and $x_1^{-1}$.}\label{fig:generators}
\end{center}
\end{figure}

Group multiplication is defined as follows when multiplying two
elements represented by tree pair diagrams.  Let $w = (T_-,T_+)$ and
$z = (S_-,S_+)$.  To form the product $wz$, we take unreduced
representatives of both elements, $(T'_-,T'_+)$ and $(S'_-,S'_+)$,
respectively, in which $S'_+ = T'_-$.  The product is then
represented by the (possibly unreduced) pair of trees $(S'_-,T'_+)$.
If the fewest possible carets are added to the tree pairs for
$g$ and $h$ in order to make $S'_+ = T'_-$, and yet the
pair $(S'_-,T'_+)$ is unreduced, we say that a caret must
be {\em removed} to reduce the tree pair diagram for $wz$.

Given any $w=(T_-(w),T_+(w))$ in $F$, let
$N(w):=N(T_-(w))=N(T_+(w))$ denote the number of carets in either
tree of a reduced tree pair diagram representing $w$.  For any
natural number $k$, let $R_k$ (respectively $L_k$) denote the tree
with $k$ right (respectively left) carets, and no other carets; if
$k=0$, $R_0$ (or $L_0$) denotes the empty tree. For $w=w_pw_n$,
where as above $w_p$ and $w_n$ are the positive and negative
subwords of the infinite normal form, the tree pair diagram
$(R_{N(w)},T_+(w))$ represents $w_p(w)$, and $(T_-(w),R_{N(w)})$
represents $w_n$. However, one of these tree pair diagrams may not
be reduced. If the last $k$ carets of $T_-(w)$ (respectively
$T_+(w)$) are all right carets, then at least $k-1$ of them must
be removed in order to produce the reduced tree pair diagram for
$w_n$ (respectively $w_p$). The inverse of $w$ is represented by
the reduced tree pair diagram $w^{-1}=(T_+(w),T_-(w))$.

For a word $y \in A^*=\{ x_0^{\pm 1}, x_1^{\pm
1}\}^*$, let $l(y)$ denote the number of letters in the
word $y$, and
for an element $w \in F$, let $l_A(w)$ be the length of the shortest
word over the generating set $A=\{ x_0^{\pm 1}, x_1^{\pm
1}\}$ that represents $w$. Following the notation of Horak, Stein and Taback
\cite{\horaksteintaback}, the length
$l_A(w)$ can be described in terms of the reduced tree pair
diagram $(T_-(w), T_+(w))$ for $w$, with carets numbered in infix
order. First, we say that caret number $p$ in a tree $T$ has {\em
type N} if caret $p+1$ is an interior caret which lies in the
right subtree of $p$.

\begin{definition} Caret pair $p$ in the reduced tree pair diagram
$(T_-(w),T_+(w))$ is a {\em penalty caret pair} if either
\begin{enumerate}
\item Caret $p$ has type $N$ in either $T_-(w)$ or $T_+(w)$, and is not a
left caret in either tree, or
\item Caret $p$ is a right caret in both $T_-(w)$ and $T_+(w)$ and
caret $p$ is neither the final caret in the tree pair diagram,
nor a left caret in either tree.
\end{enumerate}
\end{definition}

Using this notation, the following lemma is proved in \cite{\horaksteintaback}.

\begin{lemma}\label{lem:penalty}
For $w = (T_-(w),T_+(w))$, the length  $l_A(w)=
l_\infty(w)+2p(w)$, where  $l_{\infty}(w)$ is total number of
carets in both trees of the reduced tree pair diagram which are
not right carets, and $p(w)$ is the number of penalty caret pairs.
\end{lemma}

It then follows that $N(w)$ is a good estimate for the $l_A(w)$.
Lemma \ref{lengthisN} makes this relationship precise and is used
in the proof that the tameness function of the combing we
construct below is linear.

\begin{lemma}\label{lengthisN}
For $w \in F$, $N(w)-2 \leq l_A(w) \leq 4N(w)$.
\end{lemma}

\begin{proof}
Lemma \ref{lem:penalty} shows that each caret pair in the reduced
tree pair diagram for $w$ contributes 0, 1, 2, 3 or 4 to $l_A(w)$,
and the upper bound on $l_A(w)$ follows immediately. The caret pair
can contribute either 0, 1, or 2 to $l_{\infty}(w)$, and can
contribute another 2 if it is a penalty pair. In order for a
caret pair to contribute zero to the word length of the element,
both carets must be on the right side of the tree in order to not
contribute to $l_{\infty}(w)$, and either one is the root (in
which case the pair is not a penalty pair because the root is also
a left caret), or the pair is the last caret pair. So at most two
caret pairs do not contribute anything to $l_A(w)$, which yields the
lower bound on $l_A(w)$.
\end{proof}

Finally, we include here a lemma which will be used in Section \ref{sec:combing} and describes a family of words in $F$
which are always nontrivial.

\begin{lemma}\label{consecutiveindices}
Let $w\in F$, and suppose $w=a_1a_2 \cdots a_k$ where for each
$i$, either $a_i$ or $a_i^{-1}\in X_{\infty}=\{x_0, x_1, x_2,
\ldots \}$. In addition, suppose that for each $i$, if
$a_i=x_r^{\pm1}$, then $a_{i+1}=x_{r+1}^{\pm1}$ or
$x_{r-1}^{\pm1}$. Then $w\neq 1$ in $F$. \end{lemma}

\begin{proof}
We prove the lemma by induction on $k$. The base case $k=1$ is
trivial. Suppose $w=a_1 a_2 \cdots a_k$, and the indices of the
generators satisfy the hypothesis of the lemma, and let $i$ be the
smallest index appearing in $w$. We will show that if $w=1$, then
we can obtain a shorter word satisfying the conditions on indices
which is also $1$ in $F$, contradicting the inductive hypothesis.

Utilizing the representation of elements of $F$ as
piecewise linear homeomorphisms of the unit interval
(see \cite{\cfp} for details),
$x_i^{\pm1}$ has a breakpoint at $1-(1/2)^i$, and the right
derivative is $2^{\mp 1}$ at that breakpoint, but $x_j^{\pm1}$ has
support in $[1-(1/2)^j, 1]\subset [1-(1/2)^i, 1]$ for $j \geq i$.
It follows that
the net exponent of all generators $x_i^{\pm1}$ occurring in
$w$ must be zero.
We can write $w=x_i^{\epsilon_1} w_1 x_i^{\epsilon_2} \cdots w_m
x_i^{\epsilon_{m+1}}$, where for each $j$, $w_j$ is a nontrivial word in
generators of the form $x_l$ for $l>i$ and $\epsilon_j \in
\{ 1, -1 \}$, except possibly $\epsilon_1$ and $\epsilon_{m+1}$
which may be zero. Note that if either $\epsilon_1$ or
$\epsilon_{m+1}$ are zero, then necessarily $m \geq 2$, and if
both are zero, then $m \geq 3$. In any case, $m \geq 1$, and for
some pair of indices $r$ and $s$, $\epsilon_r=1$ and
$\epsilon_s=-1$.

Case 1: If for some $j$, $\epsilon_j=-1$ and $\epsilon_{j+1}=1$,
then let $w_j'$ be the word obtained from $w_j$ by increasing the
index of each generator by 1. Then
$w_j'=x_i^{-1}w_jx_i=x_i^{\epsilon_j}w_jx_i^{\epsilon_{j+1}}$ in
$F$. Furthermore, as $i$ is the minimal index in the word $w$, we know that
$w_j'$ begins and ends with
$x_{i+2}^{\pm1}$, $w_{j-1}$ ends in $x_{i+1}^{\pm1}$
(or does not exist if $j=1$), and
$w_{j+1}$ begins in $x_{i+1}^{\pm1}$ (or does not exist
if $j=m$), and so replacing
$x_i^{\epsilon_j}w_jx_i^{\epsilon_{j+1}}$ by $w_j'$ produces a
word of length $k-2$ satisfying the hypotheses of the lemma.

Case 2: If no such index $j$ exists, and
neither $\epsilon_1$ nor $\epsilon_{m+1}$ is zero, then $w$ begins with
$x_i$ and ends with $x_i^{-1}$. Therefore, since $w=1$ in $F$,
$x_i^{-1}wx_i=1$ as well, and if $w'$ is the word of length $k-2$
obtained from $w$ by deleting the first and last letters, then
$w'=x_i^{-1}wx_i=1$ in $F$, and $w'$ satisfies the hypotheses of
the lemma.

Case 3: If no such index $j$ exists, and $\epsilon_1=0$, then $m
\geq 2$ and $\epsilon_2=1$. Let $w_1'$ be the word $w_1$ with the
index of each generator increased by one. Then since $w_1'$ ends
in $x_{i+2}^{\pm1}$ and $w_2$ begins in $x_{i+1}^{\pm1}$, then
replacing $w_1x_i$ by $x_i w_1'$ results in a word of the same
length which still satisfies the hypotheses. Either this new word
satisfies the conditions of Case 2, or else it does not end in
$x_i^{-1}$. But if not, then $\epsilon_{m+1}=0$, and one can do a
similar substitution at that end to obtain a new word ending in
$x_i^{-1}$ and beginning in $x_i$ which satisfies the conditions
of Case 2. Applying the argument in Case 2 to this new word yields
a word of length $k-2$ satisfying the hypotheses of the lemma.
\end{proof}


\subsection{Nested traversal normal forms}\label{NTNF's}

In general, there are many minimal length representatives of
elements of F with respect to the standard finite generating set, and
Fordham \cite{\fordham} described effective methods for finding all such minimal
length representatives.
Cleary and Taback \cite{\seanjen} described a straightforward procedure which
canonically produces a minimal length element
(with respect to the generating set
$A=\{x_0^{\pm 1},x_1^{\pm 1}\}$)
for a purely positive or
purely negative element in
$F$; that is, an element $w$ whose infinite normal form $w_pw_n$
satisfies $w=w_p$
(hence contains only terms
with positive exponents) or $w=w_n$
(hence contains only terms with negative exponents).
They call these paths {\em
nested traversal paths} due to their construction.  The combing
paths used below will be built from concatenating and then freely reducing these nested
traversal paths.

Let $w \in F$ be a strictly negative element; that is,
$w=w_n$, and $w$ is represented
by a reduced tree pair diagram of the form $(T_-(w), R_{N(w)})$, where
$R_{N(w)}$
is a tree consisting only of $N(T_-(w))$ right carets.
To construct the nested
traversal path corresponding to $w$, we proceed as follows. We number
the carets of the tree
$T_-(w)$ in infix order, beginning with $1$. We proceed through the
carets in infix order, adding generators $x_0$, $x_0^{-1}$,  and
$x_1^{- 1}$ to the right end of the nested traversal path
at each step according to the following rules.
\begin{enumerate}
\item If the infix number of the caret is $1$, add nothing to
the nested traversal path.
\item If the caret is a left caret with infix number greater
than $1$, add $x_0^{-1}$ to the nested traversal path.
\item If the caret is an interior caret, let $T$ be the right
subtree of the caret.  If $T$ is nonempty, add
$x_0^{-1} \gamma_T x_0 x_1^{-1}$ to the nested traversal path,
where $\gamma_T$ is the nested traversal path obtained by
following these rules for the carets of $T$.
\item If the caret is an interior caret and the right subtree
of $T$ is empty, then add $x_1^{-1}$ to the nested traversal path.
\item If the caret is a right non-root caret, and its right subtree $T$
contains an interior caret, add $x_0^{-1} \gamma_T x_0$ to the
nested traversal path, where $\gamma_T$ is as above.
\item If the caret is a right non-root caret, and its right subtree $T$
contains no interior carets, then add nothing to the nested
traversal path.
\end{enumerate}
It is proved in \cite{\seanjen} that this method produces a
minimal length word representing a negative element $w_n$ of $F$, with
respect to the generating set $\{x_0,x_1\}$.  We denote this
nested traversal path for $w_n$ by $\eta(w_n)$. 
For a reduced tree pair diagram, the situation in rule (6) above never arises. 
Including it allows one to extend the algorithm to tree pair diagrams obtained 
by appending only right carets to both the last leaf in the tree $T_-(w)$ and 
to the last leaf of $R_{N(w)}$ without changing the word produced by the 
algorithm.

We define the {\em nested traversal normal form} $\eta(w)$ of an element $w
\in F$ as follows.  Let $w = w_pw_n$ be the infinite normal form for $w$.
Then the element $w_p^{-1}$, represented by the (not necessarily
reduced) tree pair diagram $(T_+(w),R_{N(w)})$, is strictly
negative, and so has a nested traversal path formed according to the above rules,
which is not affected by the possible reduction of the diagram,
according to rule (6) of the procedure above.
Hence we can define
the nested traversal normal form for $w_p$ to be
$\eta(w_p):=(\eta(w_p^{-1}))^{-1}$. It follows from the nested traversal construction that the words $\eta(w_p)$ and $\eta(w_n)$ are freely reduced, considered separately. However, their concatenation $\eta(w_p)\eta(w_n)$ may not be, so we define $\eta(w)$, the nested traversal
normal form for $w$, to be the result of freely reducing the word $ \eta(w_p)\eta(w_n)$.
Note that $\eta(w)$ is not necessarily a minimal length word
representing the element $w$.

Cleary and Taback show in the proof of Theorem 6.1 of
\cite{\seanjen} that along a strictly negative nested traversal normal form
$\eta(w_n)=a_1 a_2 \dots a_n$, the number of carets in the tree pair diagrams
corresponding to the prefixes $a_1 a_2 \dots a_i$ for $i \in \{1,2, \cdots ,n \}$
never decreases, that is, $N(a_1 a_2 \cdots a_i) \leq N(a_1 a_2 \cdots a_{i+1})$.
This follows from the construction of the paths: the multiplication $(a_1 a_2 \dots a_i) \cdot a_{i+1}$
never causes a reduction of carets. We prove below that the same holds for the general nested traversal normal form $\eta(w)$.
Our proof of
the tameness of the 0-combing
given by the nested traversal normal forms uses this
property combined with the relationship between $l_A(w)$ and $N(w)$
described in Lemma \ref{lengthisN}.

\begin{theorem}\label{Nincreases}
For $w \in F$, if $\eta(w)=a_1 a_2 \dots a_p$, then $N(a_1 a_2 \cdots a_{i-1}) \leq N(a_1 a_2 \cdots a_{i})$ for all $1 \leq i \leq p$.
\end{theorem}

\begin{proof}
We first claim that for any $u \in F$ and generator $a \in \{ x_0^{\pm 1}, x_1^{\pm 1}\}$,
carets cannot be both added and removed in the process of multiplying $ua$.
We check one case of this for the reader, in which
$a=x_1$ and the right child of the root caret in $T_-(u)$ exists and has an exposed left leaf labeled $n$.
In this case it is necessary to add a single caret to leaf $n$ of both $T_-(u)$ and $T_+(u)$, which has exposed
leaves numbered $n$ and $n+1$.  Before reduction of carets, we obtain a possibly unreduced tree pair diagram $(T_-',T_+')$ in which
 leaves $n$ and $n+1$ of $T_-'$
no longer form a caret.  Any exposed caret with leaves numbered greater than $n+1$ in $(T_-(u),T_+(u))$ has
its leaf labels increased by $1$ in $(T_-',T_+')$.  Thus if a caret pair is exposed in $(T_-',T_+')$, it would have been exposed in
$(T_-(u),T_+(u))$.  However, $(T_-(u),T_+(u))$ was reduced, and hence $(T_-',T_+')$ is reduced as well, and so equals $(T_-(ux_1),T_+(ux_1))$.  Other cases are checked similarly.

Note that $N(ua)$ and $N(u)$ can be related in one of the following three ways. Either,
\begin{itemize}
\item $N(ua) > N(u)$ if carets must be added in order to perform the multiplication, or
\item $N(ua)=N(u)$ if no carets must be added in order to perform the multiplication and no carets must be removed in order to reduce the resulting tree pair diagram, or
\item $N(ua)< N(u)$ if no carets must be added in order to perform the multiplication, but carets must be removed to reduce the resulting tree pair diagram.
\end{itemize}

We remark that in any case, $N(ua)$ and $N(u)$ differ by at most 1 when 
$a=x_0^{\pm 1}$ and by at most 2 when $a=x_1^{\pm 1}$.
In addition to the possible change in the number of carets, the resulting 
trees are rearranged slightly, in very constrained ways.

We will first prove the theorem in the case where $a_i=x_1^{\pm1}$.
Observe that for any $u \in F$, either
\begin{itemize}
\item $T_+(ux_1)$ has one more interior caret than $T_+(u)$, both $T_-(ux_1)$ and $T_-(u)$ have the same number of interior carets, and $N(ux_1) > N(u)$, or
\item $T_-(ux_1)$ has one fewer interior caret than $T_-(u)$, both $T_+(ux_1)$ and $T_+(u)$ have the same number of interior carets, and $N(ux_1) \leq N(u)$.
    \end{itemize}
To see this, note that
if carets must be added in order to perform the multiplication $ux_1$; that is,
if either the root of $T_-(u)$ does not have a right child, or
this root does have a right child but this  child does not have a left child, then
$T_+(ux_1)$ has one more interior caret than does $T_+(u)$.  In that case, $T_-(ux_1)$
is the tree $T_-(u)$ with one or two right carets added to the last leaf, and the
number of interior carets in the negative tree is preserved.  Also as noted above,
no carets can be removed, and hence $N(ux_1)>N(u)$.  On the other hand, if no carets are
added in performing 
the multiplication $ux_1$, then $N(ux_1) \leq N(u)$.
Moreover, during the multiplication process,
the left child of the right child of the root of $T_-(u)$ is an interior
caret, but in the tree $T_-(ux_1)$, either this caret has been
removed in the multiplication process, or the caret with this same number
is a right caret, and hence either way, the number of interior carets in
$T_-(ux_1)$ is strictly less than that of $T_-(u)$.
If carets are removed, they must be the final one or two carets of the tree
pair, and these must be right carets in the tree $T_+(u)$, so the number
of interior carets in the positive tree is left unchanged.
Similarly, one can verify that either
\begin{itemize}
\item $T_-(ux_1^{-1})$ has one more interior caret than $T_-(u)$, both $T_+(ux_1^{-1})$ and $T_+(u)$ have the same number of interior carets, and $N(ux_1^{-1}) \geq N(u)$, or
\item $T_+(ux_1^{-1})$ has one fewer interior caret than $T_+(u)$, both $T_-(ux_1^{-1})$ and $T_-(u)$ have the same number of interior carets, and $N(ux_1^{-1}) < N(u)$.
\end{itemize}

Furthermore, if $a=x_0^{\pm1}$, then $T_-(ua)$ and $T_-(u)$ have the same number of interior carets, as do $T_+(ua)$ and $T_+(u)$.

Now in the (possibly not freely reduced) word $\eta(w_p)\eta(w_n)$, all occurrences 
of $x_1$ precede all occurrences of $x_1^{-1}$. Furthermore, if $T_-(w)$ has $r$ 
interior carets and $T_+(w)$ has $s$ interior carets, then by construction, 
$\eta(w_p)\eta(w_n)$ has
precisely $s$ occurrences of $x_1$ followed by $r$ occurrences of $x_1^{-1}$. 
Thus in any prefix of $\eta(w)$, if $a_i = x_1^{\pm 1}$ it is always true that 
the number of interior carets in the tree pair diagram corresponding to 
$a_1a_2 \cdots a_i$ is one more than the number of interior carets in 
the tree pair diagram corresponding to $a_1a_2 \cdots a_{i-1}$.

Hence, it follows from the previous observations that if 
$ux_1$ is a prefix of $\eta(w_p)$, then $T_+(ux_1)$ has one 
more interior caret than $T_+(u)$, and $N(ux_1)> N(u)$. 
Similarly, if $ux_1^{-1}$ is a prefix of $\eta(w_n)$, then 
$T_-(\eta(w_p)ux_1^{-1})$ has one more interior caret than 
$T_-(\eta(w_p)u)$, and $N(\eta(w_p)ux_1^{-1})\geq N(\eta(w_p)u)$, 
which proves the theorem in the case $a_i=x_1^{\pm1}$. 

In addition, 
we can conclude from this analysis a few more facts about the 
relationship between the reduced word $\eta(w)$ and the 
potentially longer word $\eta(w_p)\eta(w_n)$, which we note here
for use again later. In particular, 
$\eta(w)=w_1w_2$, where $\eta(w_p)=w_1x_0^n$ and 
$\eta(w_n)=x_0^{-n}w_2$ for some $n \geq 0$, and if 
$w_1$ and $w_2$ are both nonempty words, then either 
$w_1$ ends in $x_1$ or $w_2$ begins with $x_1^{-1}$, but not both. 
Moreover, for any prefix
$a_1 \cdots a_i$ of $w_1$, $T_-(a_1 \cdots a_i)$ contains no interior carets.

To prove the theorem for the cases where $a_i=x_0^{\pm1}$, we repeatedly refer to the following six facts, each of which can be deduced by carefully following the process of multiplying by a generator.

\begin{enumerate}
\item $N(ux_0^{-1}) < N(u)$ if and only if the first caret of $T_+(u)$ is exposed, and the first two right carets of $T_-(u)$ have no left children.
\item $N(ux_0) < N(u)$ if and only if the last caret of $T_+(u)$ is exposed, and the last two left carets of $T_-(u)$ have no right children.
\item If $ N(ux_0^{-1}) \geq N(u)$, then the root caret of $T_-(ux_0^{-1})$ has a left child.

\item If $ N(ux_0) \geq N(u)$, then the root caret of $T_-(ux_0)$ has a right child.

\item If $ N(ux_1^{-1}) \geq N(u)$, then the second right caret of $T_-(ux_1^{-1})$ has a left child, so in particular the root caret has a right child.

\item If $ N(ux_1) \geq N(u)$, then the root caret of $T_-(ux_1)$ has a right child.
\end{enumerate}

We proceed by induction to prove the theorem. Clearly $N(a_1) > N(\epsilon)$, so now assume that $N(a_1 \cdots a_k) \geq N(a_1 \cdots a_{k-1})$ for all $k<i$.

Case 1: $a_i=x_0$.

Either $a_{i-1}=x_0$,  $a_{i-1}=x_1^{-1}$, or $a_{i-1}=x_1$. But then, since by the inductive hypothesis $N(a_1 \cdots a_{i-1}) \geq N(a_1 \cdots a_{i-2})$, facts 4, 5, and 6 show that the root caret of $T_-(a_1 \cdots a_{i-1})$ always has a right child. Hence, fact 2 above implies that $N(a_1 \cdots a_i) \geq N(a_1 \cdots a_{i-1})$.

Case 2: $a_i=x_0^{-1}$.

If $a_{i-1}=x_0^{-1}$ or  $a_{i-1}=x_1^{-1}$, then since by the inductive hypothesis $N(a_1 \cdots a_{i-1}) \geq N(a_1 \cdots a_{i-2})$, facts 3 and 5 above show that either the root caret or the second right caret of $T_-(a_1 \cdots a_{i-1})$  has a left child. Hence, in these cases fact 1 above implies that $N(a_1 \cdots a_i) \geq N(a_1 \cdots a_{i-1})$. So we must check the one remaining possibility, that $a_{i-1}=x_1$. We claim that for $a_{i-1}=x_1$, either the root caret of $T_-(a_1 \cdots a_{i-1})$ has a left child, or the first caret of $T_+(a_1 \cdots a_{i-1})$ is not exposed, which will again imply by fact 1 above that $N(a_1 \cdots a_i) \geq N(a_1 \cdots a_{i-1})$. To verify the claim, first note that since
the letter $x_1$ cannot occur in $\eta(w_n)$, the word $a_1 \cdots a_i$ is a prefix of
$w_1$, and so neither $T_-(a_1 \cdots a_{i-2})$ nor $T_-(a_1 \cdots a_{i-1})$ contain any interior carets. Now if the root caret of $T_-(a_1 \cdots a_{i-2})$ has a left child, then so does the root caret of $T_-(a_1 \cdots a_{i-1})$. On the other hand, if $T_-(a_1 \cdots a_{i-2})$ consists only of right carets, then $T_+(a_1 \cdots a_{i-1})$ is obtained by hanging a caret from the second leaf of $T_+(a_1 \cdots a_{i-2})$, so that in this case, the first caret of $T_+(a_1 \cdots a_{i-1})$ is not exposed.
\end{proof}
In addition, we remark that the nested traversal forms, which are certainly 
not in general geodesics, are at least quasigeodesics. To see this, it is 
helpful to first make some preliminary observations about nested traversal 
paths for strictly negative words, each of which can be deduced from the 
algorithm for the construction 
of nested traversal paths.

\begin{observation}\label{nestedtraversalproperties}

\begin{enumerate}

\item A word $a_1 \cdots a_n$, where $a_i \in \{x_0, x_0^{-1}, x_1^{-1}\}$
for each $i$, 
is a nested traversal path if and only if the word satisfies the following three conditions:
    \begin{enumerate}
    \item The word is freely reduced.
    \item The exponent sum of $x_0$ in any prefix $a_1 \cdots a_k$ for $1 \leq k \leq n$ is not positive.
    \item If $a_k=a_{k+1}=x_0$ for some $k$, then $a_j=x_0$ for $k \leq j \leq n$.
    \end{enumerate} Hence, any prefix of a nested traversal path is again a nested traversal path.

\item Let $a_1 \cdots a_n$ be a nested traversal path with reduced tree pair diagram $(T_-, R_l)$ for some $l$. Then the first caret of $T_-$ is exposed if and only if $a_1=x_0^{-1}$ and the exponent sum of $x_0$ in every nonempty prefix $a_1 \cdots a_k$ is strictly negative.  Note that the final caret of $T_-$, always a right caret, is never exposed.

\item If $a_1 \cdots a_n$ is a nested traversal path with reduced tree 
pair diagram $(T_-, R_l)$ for some $l$, the numbers labeling the exposed
carets of $T_-$ can be algorithmically determined as follows. 
Caret 2 is exposed if $a_1=x_1^{-1}$, 
and not if $a_1=x_0^{-1}$.
    So by reading through the prefix $a_1$, it can be determined whether 
or not caret 2 is exposed.
     Inductively, suppose that by reading through a prefix $a_1 \cdots a_k$, 
you have decided whether or not carets $2$ through $j$ are exposed. Then for caret $j+1$, if

   \[
a_{k+1}=\left\{
\begin{array}{ll}
x_0&
\mbox{then move on to $a_{k+2}$ to make a decision about caret $j+1$.} \\
x_1^{-1} &
\mbox{and $a_k=x_0$, then move on to $a_{k+2}$ to make a decision about caret $j+1$.}\\
x_0^{-1} &
\mbox{then caret $j+1$ is not exposed.} \\
x_1^{-1} &
\mbox{and $a_k=x_0^{-1}$, then caret $j+1$ is exposed.}\\
x_1^{-1} &
\mbox{and $a_k=x_1^{-1}$, then caret $j+1$ is not exposed.}\\

\end{array}\right. \]
In the latter three cases, then move on to $a_{k+2}$ to make a decision about whether caret $j+2$
is exposed.
\end{enumerate}

\end{observation}

With this observation in hand, we are ready to prove that nested traversal normal forms are quasigeodesics.
Recall that
for a group $G$ with finite generating set $A$,
a word $y$ is a $(\lambda,\epsilon)$-quasigeodesic for constants $\lambda \ge 1$ and
$\epsilon \ge 0$ if the unit speed path $p:[0,l(y)] \rightarrow \Gamma(G,A)$ labeled by $y$ satisfies
$\frac{1}{\lambda}|s-t| -\epsilon \le d_A(p(s),p(t)) \le \lambda |s-t|+\epsilon$
for all $s,t \in [0,l(y)]$.

\begin{theorem} For every $w \in F$
the nested traversal normal form $\eta(w)$ is a $(\lambda,\epsilon)$-quasigeodesic
with $\lambda=6$ and $\epsilon=0$.
\end{theorem}

\begin{proof}
Let $w \in F$.  From facts noted in the proof of Theorem \ref{Nincreases}, 
we can write $\eta(w_p)\eta(w_n)=w_1 x_0^n x_0^{-n}w_2$ and 
$\eta(w)=w_1w_2=a_1 \cdots a_p$ for some $n \ge 0$ and each 
$a_i \in \{x_0^{\pm 1},x_1^{\pm 1}\}$.
It follows from the formula for length in Lemma \ref{lem:penalty} 
that each pair of carets in $T_-(w)$ and $T_+(w)$, other than the 
first pair and the last pair, contribute some nonzero number between 
one and 4 to $l_A(w)$.  However, each such caret pair contributes at 
most 6 to $l(\eta(w))$, so the length contribution to the normal form 
is at most 6 times the contribution to length in the Cayley graph. The 
first and last caret pairs are a slightly special case, since together 
they may contribute 0, 1, or 2 to both $l_A(w)$ and $l(\eta(w))$. However, 
one checks that if there is no contribution to $l_A(w)$ from these carets, 
then the first caret in both trees is the root caret, and so the 
contribution to $l(\eta(w))$ from these carets will be zero as well. 
Thus we obtain $$l_A(w)  \leq l(\eta(w)) \leq 6~ l_A(w)$$ for any nested 
traversal normal form $\eta(w)$.

 We now show that these inequalities hold for an arbitrary subword of 
$\eta(w)$. So suppose $u=a_i \cdots a_{i+k}$, with $1 \leq i < i+k \leq r$. 
If $a_{i+k}$ is a letter of $w_1$, then since $w_1 x_0^n=\eta(w_p)$ is a 
geodesic, then $l_A(u)  \leq l(\eta(u)) \leq 6 l_A(u)$. The 
same argument, using the fact that $\eta(w_n)$ is a geodesic, 
holds in the case where $a_i$ is a letter in $w_2$.

 Now assume $a_1 \cdots a_i$ is a prefix of $w_1$ and $a_{i+k} \cdots a_p$ is a suffix of
 $w_2$. Let $u_1$ be the suffix of $w_1$ starting with $a_i$, and let $u_2$ be the prefix of $w_2$ ending with $a_{i+k}$; then $u=u_1u_2$. We prove below that for the subword $u$, $\eta(u_p)=u_1x_0^s$ and $\eta(u_n)=x_0^{-s}u_2$ for some $0 \leq s \leq n$. This then implies that $\eta(u)=u_1u_2$, and therefore $u$ is a nested traversal normal
 form, and $l_A(u)\leq l(\eta(u)) \leq 6 l_A(u)$.

Since $x_0^{-n}u_2$ and $x_0^{-n}u_1^{-1}$ are prefixes of $\eta(w_n)$ and $\eta(w_p^{-1})$, 
they are also nested traversal paths according to the first part of 
Observation~\ref{nestedtraversalproperties} .
Let $(T_-',R_k)$ denote the reduced tree pair diagram for $x_0^{-n}u_2$ and let $(T_-'',R_l)$ be the reduced tree pair for $x_0^{-n}u_1^{-1}$.
There is a two step process to transform the pair of trees $(T_-',T_-'')$ into the reduced tree pair diagram for $u=u_1u_2$.
In the first step, if
the numbers $k$ and $l$ of
carets in each of these trees are not equal, we add a string of $|k-l|$ right carets to
the final leaf of the smaller tree.
The second step consists of reducing the resulting tree pair.
Note that by construction, the final caret of both $T_-'$ and $T_-''$
cannot be exposed, and so any carets added in the first step will not be removed
in the second step.
Each time a caret is removed, there is a corresponding change
to the pair of nested traversal paths, which we track below;
at the end of this process, 
we obtain the nested traversal paths for $u_n$ and $u_p$.

If caret 1 is exposed in both trees, it must be that $n>0$, so
 removing this caret pair corresponds algebraically to canceling
 the central $x_0x_0^{-1}$ pair to obtain the word $u_1 x_0^{n-1} x_0^{-(n-1)} u_2$.
Note that both $x_0^{-(n-1)}u_2$ and $x_0^{-(n-1)}u_1^{-1}$ are both again 
nested traversal paths.
 Now repeat, and eventually reach a point where caret 1 is not exposed 
in one of $T_-( x_0^{-s} u_2)$ and $T_-(x_0^{-s}u_1^{-1})$, for some 
$0 \leq s \leq n$. Note that if $s=0$, caret 1 cannot be exposed in 
both trees.  Hence this process must successfully terminate.

Next we check for possible reduction of caret pairs numbered greater than 1.
Part 3 of Observation~\ref{nestedtraversalproperties} shows that caret 2 can only be exposed
in both trees $T_-( x_0^{-s} u_2)$ and $T_-(x_0^{-s}u_1^{-1})$
if $s=0$ and both $u_1^{-1}$ and $u_2$ start with the letter $x_1^{-1}$; however, $u_2$ is
a prefix of the word $w_2$ and $u_1$ is a suffix of $w_1$, and the word $w_1w_2$ is freely reduced.
For caret pairs numbered between 3 and $\min\{k,l\}-(n-s)$, the algebraic criteria
from part (3) of Observation~\ref{nestedtraversalproperties} by which we check
for exposure of these carets only depends upon letters in the words $u_2$ and $u_1^{-1}$ which,
as noted above, are prefixes of the words $w_2$ and $w_1^{-1}$, respectively.  Since the
tree pair diagram for $w_1w_2$ is reduced, then none of these carets can  be removed.
Carets numbered above $\min\{k,l\}-(n-s)$ were added in the first step, and hence also cannot
be removed.

 Hence, $\eta(u_p)\eta(u_n)= u_1 x_0^s x_0^{-s} u_2$, so $\eta(u)=u_1u_2$, and the inequalities follow. Thus nested traversal normal forms are (6,0)-quasigeodesics.
\end{proof}


\section{Constructing the combing of $F$}
\label{sec:combing}

In this section, we construct a $1$-combing of
the group $F$ with respect to the presentation
$$\langle x_0,x_1 |
[x_0x_1^{-1},x_0^{-1}x_1x_0],[x_0x_1^{-1},x_0^{-2}x_1x_0^2]
\rangle~;$$
in Section \ref{sec:tameness} we will show that this
combing satisfies a linear radial tameness
function.  Let $X$ be the Cayley complex for this
presentation.

We first construct a $0$-combing of $F$ with respect to
$A=\{x_0, x_1\}$ by defining a continuous function $\Psi:X^0 \times [0,1]
\rightarrow X^1$ where, for any $w \in F$, the restriction
$\Psi:\{w\}
\times [0,1] \rightarrow X^1$ is labeled by the nested traversal
normal form $\eta(w)$ for $w$.
We call this 0-combing $\Psi$ the {\em nested
traversal 0-combing}.

Now we must extend this 0-combing to a 1-combing.  All edges in
the Cayley graph fall into one of two categories, ``good" and
``bad". The good edges consist of those edges where the combing
path to one endpoint contains the other endpoint, and thus points
along that edge are combed through the $1$-skeleton.  The bad
edges include all of those edges where this is not the case, and thus
the points along the edge in question must be combed through the
$2$-skeleton. To make this more formal, we introduce some notation.
For each $w \in F$, let
$\Psi_w$ be the 0-combing path in $X^1$ from the identity to $w$ (labeled
by $\eta(w)$),
and let $\Psi_w^{-1}$ be the inverse path from $w$ to the identity.
Recall that each directed edge in the Cayley graph
$X^1=\Gamma(F,\{x_0,x_1\})$ is labeled either by the generator $x_0$
or the generator $x_1$.  We formalize the notion of good and bad edges in the following
definition.

\begin{definition}
If the set of endpoints of an edge $e$ is of the form $\{ w, wx_0^{-1} \}$,
we denote the edge as $e_0(w)$, and if the set of endpoints of an edge $e$
is of the form $\{ w, wx_1^{-1} \}$, we denote the edge as $e_1(w)$.

Moreover, for the edge $e_a(w)$, where $a \in \{0,1\}$, if the loop
$\gamma_e:= \Psi_w e_{a}(w) \Psi_{wx_{a}^{-1}}^{-1}$
is homotopic to the trivial loop in the Cayley graph, we call the
edge $e_a(w)$ a {\em good edge},
and if not, we call the edge $e_a(w)$ a {\em bad edge}.
\end{definition}

Theorem \ref{combingextends} describes conditions on the Cayley
complex which allow us to extend the nested traversal
0-combing to a 1-combing.

\begin{theorem}\label{combingextends}
Let $\mathcal B$ be the set of bad edges in the Cayley complex $X$
 with respect to the nested traversal 0-combing.
Suppose that
\begin{enumerate}
\item there
is a partial ordering of $\mathcal B$ with the property that for
any edge $e \in \mathcal B$, the set of edges $\{ f \in \mathcal B | f < e
\}$ is finite, and
\item there is a function $c$ from
$\mathcal B$ to the set of 2-cells of $X$, so that for every $e \in
\mathcal B$ the edge $e$ is on the boundary of $c(e)$, and
whenever $f \in \mathcal B$ is another edge on the boundary of $c(e)$,
then $f<e$.
\end{enumerate}
Then the nested traversal 0-combing $\Psi:X^0 \times [0,1]
\rightarrow X^1$ can be extended to a 1-combing $\Psi:X^1 \times
[0,1] \rightarrow X^2$.
\end{theorem}

\begin{proof}
We remark that the hypotheses of the theorem imply that the
mapping from bad edges to 2-cells is injective.
Let $\mathcal G$ be the set of good edges.
We extend the
0-combing in two stages. First, extend $\Psi:X^0 \times [0,1]
\rightarrow X^1$ to $\Psi:(X^0 \cup \mathcal G) \times [0,1]
\rightarrow X^1$ using the homotopies for the good edges. Next,
note that the partial ordering on $\mathcal B$ is
well-founded, and so we may apply Noetherian induction to
define $\Psi$ on bad edges as follows.
Suppose we have already extended the combing
to $\Psi:(X^0 \cup \mathcal G \cup S) \times [0,1] \rightarrow X^2$, where
$S:= \{ e' \in \mathcal B | e' < e \}$
for a particular edge $e$. Since $c(e)$ is a 2-cell, there is a homotopy
from $e$, through $c(e)$, to the remainder of the boundary $\partial c(e)$
excluding the interior $Int(e)$ of $e$, which fixes the endpoints of $e$.
More specifically, let $\Theta:e \times [0,1] \ra c(e)$ satisfy: for each point $p$
in the edge $e$,
$\Theta(p,0) \in \partial c(e) \setminus Int(e)$ and $\Theta(p,1)=p$;
the image $\Theta(e \times \{0\})=\partial c(e) \setminus Int(e)$;
and for the endpoints $g$ and
$h$ of $e$ and for all $t \in [0,1]$, $\Theta(g,t)=g$ and $\Theta(h,t)=h$.
Since all edges in $\partial c(e) \setminus Int(e)$ are in
$\mathcal G \cup S$, the combing $\Psi:(X^0 \cup \mathcal G
\cup S) \times [0,1] \rightarrow X^2$ provides
combing paths from the identity to each of the points of
$\partial c(e) \setminus Int(e)$.  Reparametrize these
paths and concatenate them with the paths from the homotopy
$\Theta$ to define the homotopy $\Psi:e \times [0,1] \ra X$.
This yields a homotopy $\Psi:(X^0 \cup \mathcal G \cup S \cup
\{e \}) \times [0,1] \rightarrow X^2$. Then, by
induction, the 0-combing extends to a 1-combing $\Psi:X^1 \times
[0,1] \rightarrow X^2$.
\end{proof}

The remainder of this section is devoted to establishing the
hypotheses of Theorem$~\ref{combingextends}$.


\subsection{Identifying the good edges}\label{subsec:good}

The goal of this section is to identify the good edges in
$\Gamma(F,A=\{x_0,x_1\})$. This is accomplished in the following
theorem.

\begin{theorem}\label{goodedges}
Let $w\in F$ have reduced tree pair diagram $(T_-(w), T_+(w))$.
If any one of the following four conditions holds, then
the edge $e_a(w)$, with $a \in \{0,1\}$, is a good edge:
\begin{enumerate}
\item The index $a=0$.
\item The tree $T_-(w)$ has at most two right carets.
\item The tree $T_-(w)$ has at least three right carets, no carets
need be removed to reduce the tree pair diagram for $wx_1^{-1}$, and
all carets following the third right caret in $T_-(w)$, if any, are right carets.
\item The tree $T_-(w)$ has at least three right carets but
no interior carets,
caret $n$ must be removed to reduce the tree pair diagram for
$wx_1^{-1}$, and caret $n$ is the first exposed
caret in $T_+(w)$.
\end{enumerate}
\end{theorem}

We prove this theorem in two lemmas, considering separately the
cases $e_0(w)$ and $e_1(w)$. To prove each lemma, we simply
compare the nested traversal forms for the two endpoints of the
edge in each situation in the hypotheses of
Theorem~\ref{goodedges}.

\begin{lemma}\label{x0generator}
Let $w \in F$.  Then $e_0(w)$ is a good edge.
\end{lemma}

\begin{proof}
Let $z=wx_0^{-1}$. We compare $\eta(w)$ and $\eta(z)$, and
show that either $\eta(z)=\eta(w)x_0^{-1}$ or $\eta(w)=\eta(z)x_0$, and so it follows immediately that $e_0(w)$
is a good edge.

As usual, let $w=(T_-(w),T_+(w))$ and $z=(T_-(z),T_+(z))$ denote
reduced tree pair diagrams. The tree pair diagram $(S_-,S_+)$ for $x_0^{-1}$
is given in Figure \ref{fig:generators}.

Suppose first that $T_-(w)$ has only one right caret, the root
caret.  The left subtree of the root caret must then be nonempty;
let $A(w)$ be this subtree, and
let $\gamma_A$ be the substring of $\eta(w_n)$
consisting of all generators corresponding to carets in $A(w)$.
Then $\eta(w_n)=\gamma_A x_0^{-1}$.
In multiplying $z=wx_0^{-1}$, a caret is appended to the
rightmost leaf of each of the trees
for $w$, and the tree $A(w)$ is appended to the leftmost leaf
of the trees $S_-$ and $S_+$ for $x_0^{-1}$.
Then $z=(T_-(z),T_+(z))$ where $T_-(z)$ consists of a root caret
with a left child whose left subtree is $A(w)$, and $T_+(z)$
is $T_+(w)$ with a single caret appended to its rightmost leaf.
This pair is reduced, so no caret is removed in performing
this product.
The tree $T_-(z)$ has a left caret
between the subtree $A(w)$ and the root, so $\eta(z_n)=\gamma_A
x_0^{-2}$. Since $T_+(z)$ is just $T_+(w)$ with a single caret
appended to its rightmost leaf, $\eta(w_p^{-1})=\eta(z_p^{-1})$.
Hence in this case, $\eta(z_p)\eta(z_n)=\eta(w_p) \gamma_A
x_0^{-2} = \eta(w_p)\eta(w_n)x_0^{-1}$.  Therefore, if $\eta(w)$ does not end in $x_0$, then $\eta(z)=\eta(w)x_0^{-1}$. However, if $\eta(w)$ does end in $x_0$, then $\eta(z)x_0=\eta(w)$.

For the remainder of this proof suppose that $T_-(w)$ has at least
two right carets. Let $A(w)$ be the left subtree of the root
caret, let $B(w)$ denote the left subtree of the right child of
the root caret, and let $E(w)$ denote the right subtree of the
right child of the root.  Let $\gamma_A$, $\gamma_B$, and
$\gamma_E$ denote the subwords of $\eta(w_n)$ corresponding to the
carets of these subtrees; note that any of these trees can be the
empty tree, and if so, the corresponding subword will be empty.
Define $\gamma_r:=1$ if the tree $A(w)$ is the empty tree
$\emptyset$ with no carets, and $\gamma_r:=x_0^{-1}$ if $A(w) \neq
\emptyset$, so that $\gamma_r$ is the contribution of the root
caret of $T_-(w)$ to the nested traversal normal form $\eta(w_n)$.
The nested traversal normal form for $w$ is then
\[
\eta(w_p)\eta(w_n)=\left\{
\begin{array}{ll}
\eta(w_p)\gamma_A \gamma_r \gamma_B &
\mbox{if $E(w) = R_k$ for some $k \ge 0$} \\
\eta(w_p) \gamma_A \gamma_r \gamma_B x_0^{-1} \gamma_E x_0 &
\mbox{if $E(w) \neq R_k$ for all $k$.}
\end{array} \right. \]
In this case no carets need to be
added to the tree pair for $w$ in order to perform the
multiplication $wx_0^{-1}$; the trees $A(w)$, $B(w)$, and $E(w)$ must
be appended to leaves 0, 1, and 2, respectively of the
trees $S_-$ and $S_+$.

A caret must be removed from the product $wx_0^{-1}$ to obtain the
reduced tree pair diagram if and only if the
trees $A(w)$ and $B(w)$ are both empty, and caret 1 in the tree
$T_+(w)$ is exposed. In this case, $T_+(z)$ is the tree $T_+(w)$
with the first caret removed. Note that caret 1 of $T_+(w)$
contributed nothing to the nested traversal normal form
$\eta(w_p^{-1})$, and that caret 2 of $T_+(w)$ must also be a left
caret, and so contributed $x_0^{-1}$.  The latter caret is caret 1
of $T_+(z)$. Hence $\eta(w_p^{-1})=x_0^{-1}\eta(z_p^{-1})$.
Analyzing the negative trees, we note that the tree $T_-(z)$ is
the tree with a single left caret, namely the root caret, having a
right subtree given by $E(w)$, and so $\eta(z_n)=\gamma_E$.
If $E(w)=R_k$ for some $k \ge 0$, then $\eta(w)=\eta(z)x_0$.
On the other hand, if $E(w) \neq R_k$ for any $k$, then
since $E(w)$ is the right subtree of the root of $T_-(z)$,
this subtree gives a
nonempty contribution to the nested traversal path, and hence $\eta(z_n)$
cannot end with the letter $x_0^{-1}$.  Thus when freely reducing
the word
$\eta(w_p)\eta(w_n)=\eta(z_p)x_0x_0^{-1}\eta(z_n)x_0$, only one
$x_0x_0^{-1}$ is removed, and $\eta(w)=\eta(z)x_0$.

Finally, suppose that no carets need to be removed in the
multiplication $wx_0^{-1}$. Then $T_+(z)=T_+(w)$ and $z_p=w_p$.  
Therefore  $\eta(z_p)\eta(z_n)=\eta(w_p) \gamma_A
\gamma_r \gamma_B x_0^{-1} \gamma_E$. 
If $E(w)=R_k$, then
$\eta(z_p)\eta(z_n)=\eta(w_p)\eta(w_n)x_0^{-1}$, and so either 
$\eta(z)=\eta(w)x_0^{-1}$, when $\eta(w)$ does not end in $x_0$, 
or else $\eta(z)x_0=\eta(w)$. On the other hand if $E(w) \neq R_k$, then
$\eta(w_p)\eta(w_n)=\eta(z_p)\eta(z_n)x_0$. Since $\eta(z_n)$ 
cannot end in $x_0^{-1}$, $\eta(w)=\eta(z)x_0$.
\end{proof}

Next, we turn to edges of the form $e_1(w)$. For such an edge, the
case where $T_-(w)$ has at least three right carets is by far the
most complicated, so before stating the desired lemma, we
establish some useful notation for that case.

\begin{notation} \label{not:4trees}
For any $w=(T_-(w),T_+(w)) \in F$ such that $T_-(w)$ has at least 3
right carets:
\begin{itemize}
\item Let $A(w)$ denote the left subtree of the root
caret, $B(w)$ denote the left subtree of the right child of
the root, and  $C(w)$ and $D(w)$ denote the left and right
subtrees, respectively, of the third right caret of $T_-(w)$.

\medskip

\item Define $N(w):=N(T_-(w))=N(T_+(w))$ (as above), and $N_D(w):=N(D(w))$ and
$N_A(w):=N(A(w))$.

\medskip

\item Define $j(w)$ to be the
number of the first exposed caret of $T_+(w)$.

\end{itemize}
\end{notation}

To understand the good edges, and later the definition of the partial order
on the edges, one must first understand explicitly how the tree pair diagram
for $w$ may change when $w$ is multiplied by $x_1^{-1}$.  To form the product
$wx_1^{-1}$, if $T_-(w)$ contains at least 3 right carets, then
no carets must be added to
the trees of the reduced pair diagram for $w$, but the subtrees
$A(w)$, $B(w)$, $C(w)$, and $D(w)$ are appended to the leaves
numbered 0, 1, 2, and 3, respectively, of the trees in the reduced
tree pair diagram for $x_1^{-1}$, which is given in Figure \ref{fig:generators}.

Continuing the case that $T_-(w)$ contains at least 3 right carets,
let $T_-'$ be the negative tree of the intermediate step in the
multiplication, $wx_1^{-1}$ before any carets are removed to reduce the tree pair diagram. The
carets of $T_-(w)$ and $T_-'$ with the same number have the same type (left, right or interior) in both
trees, with the exception of the caret
numbered $N_A(w)+1+N_B(w)+1$, which is a right caret in $T_-(w)$
and an interior caret in $T_-'$. This is the only caret that
can be exposed in $T_-'$ but not in $T_-(w)$, and hence the only
caret that might be removed if $T_-'$ is not reduced. As a consequence,
a caret must be
removed in the multiplication $wx_1^{-1}$ if and only if the
following property holds:

\begin{description}
\item[$(\ddagger)$]  $\quad B(w) = \emptyset$, $C(w) = \emptyset$,
and $T_+(w)$ has an exposed caret at caret number $N_A(w)+2$.
\end{description}

We are now ready to describe edges of the form $e_1(w)$ that are good edges.

\begin{lemma}\label{x1generator}
Let $w\in F$. Then $e_1(w)$ is a good edge if any of the following are satisfied:
\begin{enumerate}
\item $T_-(w)$ has at most two right carets.
\item $T_-(w)$ has 3
or more right carets, property $(\ddagger)$ is not satisfied and $D(w)=R_k$ for
some $k \geq 0$.
\item $T_-(w)$ has 3 or more right carets but no
interior carets, property $(\ddagger)$ holds and the number $n$ of
the caret that cancels satisfies
$n=N_A(w)+2 =j(w)$, so caret $n$ is the first exposed caret in $T_+(w)$.
\end{enumerate}
\end{lemma}

\begin{proof}
Let $u=wx_1^{-1}$. Again we compare $\eta(w)$ with $\eta(u)$, and
we claim that when $w$ satisfies either of the first two conditions of the hypothesis, then $\eta(u)=\eta(w)x_1^{-1}$. If $w$ satisfies the third condition, then  $\eta(w)=\eta(u)x_1$.
So in all cases it follows immediately that the edge
$e_1(w)$ is good. We proceed by cases according to which hypothesis is satisfied. The first two are straightforward, but for the third we separate into subcases.

\noindent{\it Case 1. $T_-(w)$ has fewer than three right carets.}
Suppose first that $T_-(w)$ has only one right
caret, and $A(w)$ is the left subtree of this root caret.
In multiplying $u=wx_1^{-1}$, two right carets are appended to the
rightmost leaf of each of the trees
for $w$, and $A(w)$ is appended to the leftmost leaf
of the trees for $x_1^{-1}$, but no carets are removed.
The two appended right carets in $T_+(u)$ contribute nothing
to $\eta(u_p^{-1})$, so $\eta(u_p)=\eta(w_p)$.
The root caret of $T_-(u)$ has left subtree $A(w)$,
and the right child of the root has left subtree consisting
of a single interior caret which contributes $x_1^{-1}$
to the nested traversal normal form $\eta(u_n)$.
Then $\eta(u_p)\eta(u_n)=\eta(w_p) (\eta(w_n) x_1^{-1})$, so it follows that $\eta(u)=\eta(w)x_1^{-1}$.  The proof in the case that
$T_-(w)$ has two right carets is similar.

For the remainder of the proof, assume that $T_-(w)$ has at least
three right carets.
Let $\gamma_A$, $\gamma_r$, $\gamma_B$, $\gamma_C$, and $\gamma_D$
be the subwords of the nested traversal normal form $\eta(w_n)$
corresponding to the carets of $A(w)$, the root, $B(w)$, $C(w)$,
and $D(w)$, respectively.  In this case
the nested traversal normal form for $w$
is then
\[
\eta(w_p)\eta(w_n)=\left\{
\begin{array}{ll}
\eta(w_p)\gamma_A \gamma_r \gamma_B &
\mbox{if $C(w) = \emptyset$ and $D(w) = R_k$ for some $k \ge 0$} \\
\eta(w_p)\gamma_A \gamma_r \gamma_B x_0^{-1} \gamma_C x_0&
\mbox{if $C(w) \neq \emptyset$ and $D(w) = R_k$ for some $k \ge 0$} \\
\eta(w_p)\gamma_A \gamma_r \gamma_B x_0^{-2} \gamma_D x_0^2&
\mbox{if $C(w) = \emptyset$ and $D(w) \neq R_k$ for all $k \ge 0$} \\
\eta(w_p)\gamma_A \gamma_r \gamma_B x_0^{-1}
  \gamma_C x_0^{-1} \gamma_D x_0^2&
\mbox{if $C(w) \neq \emptyset$ and $D(w) \neq R_k$ for all $k \ge
0$.}
\end{array}\right. \]

\noindent{\it Case 2. No carets must be removed to create the
reduced tree pair diagram for $u=wx_1^{-1}$, and $D(w)=R_k$ for some $k \geq
0$.} It follows immediately that $T_+(w)=T_+(u)$ and hence
$\eta(w_p)=\eta(u_p)$. From the discussion of $T_-(w)$ and
$T_-'=T_-(u)$ above, only caret number $N_A(w)+2+N_B(w)$ makes a
different contribution to the respective nested traversal normal
forms, yielding:

\[
\eta(u_p)\eta(u_n)=\left\{
\begin{array}{ll}
\eta(w_p)\gamma_A \gamma_r \gamma_B x_1^{-1}&
\mbox{if $C(w) = \emptyset$ and $D(w) = R_k$ for some $k \ge 0$} \\
\eta(w_p)\gamma_A \gamma_r \gamma_B x_0^{-1} \gamma_C x_0
x_1^{-1}&
\mbox{if $C(w) \neq \emptyset$ and $D(w) = R_k$ for some $k \ge 0$.} \\

\end{array}\right. \]

Comparing these words with the corresponding words $\eta(w_p)\eta(w_n)$ given above
yields $\eta(u)=\eta(w)x_1^{-1}$.

\noindent{\it Case 3.  Caret $n$ is removed when we form the product
$wx_1^{-1}$ (equivalently, property $(\ddagger)$ holds), $j(w)=n$, and
$T_-(w)$ has no interior carets.} From $(\ddagger)$, this caret necessarily
has caret number $n=N_A(w)+2$. As $T_-(w)$ has no interior carets in Case 3, we must have
$A(w)=L_{n-2}$, the tree with $n-2$ left carets, where $n-2 \ge 0$ and $D(w)=R_k$ for some
$k \ge 0$. Then $\eta(w_n)=x_0^{-(n-2)}$.  When caret $n$ is removed to form the
 tree pair diagram for $u$,
we see that $T_-(u)$ then has $n-1$ left carets including the root and $k+1$
right non-root carets,
and so $\eta(u_n)=x_0^{-(n-2)}$ as well.

Note that $N(w) \ge N_A(w)+3=n+1$, and so caret $n$ of $T_+(w)$ is
neither the first nor the last caret of this tree. Then this
is an interior caret of $T_+(w)$ which is an exposed caret, in particular it has an empty
right subtree. This caret will contribute $x_1^{-1}$ to the nested
traversal normal form $\eta(w_p^{-1})$. The tree $T_+(u)$ is the
tree $T_+(w)$ with caret $n$ removed.

For $1 \le j \le N(w)$, let ${\mathcal C}_j$ denote caret $j$ of
the tree $T_+(w)$. Caret ${\mathcal C}_1$ contributes nothing to
the nested traversal normal form $\eta(w_p^{-1})$. Whenever $2 \le
j < n-1$, the unexposed caret ${\mathcal C}_j$ is either an
interior caret or a right caret, and in both cases ${\mathcal C}_j$
has a nonempty right subtree containing the interior caret
${\mathcal C}_{n}$. Hence each of these carets ${\mathcal C}_j$
adds $x_0^{-1}$ to $\eta(w_p^{-1})$ before the subword $x_1^{-1}$
corresponding to caret ${\mathcal C}_{n}$, and also adds either
$x_0x_1^{-1}$ or $x_0$ to $\eta(w_p^{-1})$ after this subword.
Then $\eta(w_p^{-1}) = x_0^{-(n-2)}x_1^{-1} \beta$ for some word
$\beta$, and hence
$\eta(w_p)\eta(w_n)=(\beta^{-1}x_1x_0^{n-2})(x_0^{-(n-2)})$, so $\eta(w)=\beta^{-1}x_1$.

To analyze the nested traversal normal forms $\eta(u_p^{-1})$ and
$\eta(u)$ further, we now divide into four subcases, as follows.

\noindent {\it Case 3a. Suppose $A(w)=\emptyset$.} Then it follows
that $n=2$, and
the tree $T_+(u)$ is
$T_+(w)$ with caret $n=2$ removed.  Hence $\eta(u)=\beta^{-1}$, and so $\eta(w)=\eta(u)x_1$.

\noindent {\it Case 3b. Suppose $A(w) \neq \emptyset$, caret
$n$ is the left child of its parent caret in $T_+(w)$, and $N(w)=
n+1$.} Then it follows that all other carets of $T_+(w)$ are right
carets, or else a caret with infix number less than $n$ would be
the first exposed caret in $T_+(w)$.  Thus
$\eta(w_p)\eta(w_n)=x_0^{-(n-2)}x_1x_0^{n-2}x_0^{-(n-2)}$. The tree $T_+(u)$
contains only right carets, and so $\eta(u_p)\eta(u_n)=(1)x_0^{-(n-2)}$. Therefore 
the nested traversal normal form for $w$ is
$\eta(w)=x_0^{-(n-2)}x_1=\eta(u)x_1$ and
so $\eta(w)=\eta(u)x_1$.

\noindent {\it Case 3c. Suppose that $A(w) \neq \emptyset$, caret
$n$ is the left child of its parent caret in $T_+(w)$, and
$N(w)>n+1$.} Caret ${\mathcal C}_{n+1}$ is the parent of caret
${\mathcal C}_n$ in this case. If ${\mathcal C}_{n+1}$ is an
interior caret of $T_+(w)$, then ${\mathcal C}_{n+1}$ is an
interior caret contained in the right subtree of carets ${\mathcal
C}_j$ for all $2 \le j \le n-1$, and so in the tree $T_+(u)$,
these carets ${\mathcal C}_j$ also contain an interior caret in
their right subtrees. If instead ${\mathcal C}_{n+1}$ is a right
caret, then ${\mathcal C}_j$ is a right caret for all $1 \le j \le
n-1$. Note that the final caret $N(T_-(w))$ of $T_-(w)$ is exposed, and
the tree pair $(T_-(w),T_+(w))$ is reduced, so caret number
$N(T_+(w))>n+1$ of $T_+(w)$ is not exposed.  Then the left subtree
of the latter caret contains an interior caret ${\mathcal C}_i$ of
$T_+(w)$ with $i>n$, and hence this interior caret is contained in
the right subtrees of all of the carets ${\mathcal C}_j$ with $2
\le j \le n-1$. Then for both types of parent caret the nested
traversal path for $u_p^{-1}$ is the same as that for $w_p^{-1}$
except that the $x_1^{-1}$ subword corresponding to caret $n$ is
removed. Then $\eta(u_p)\eta(u_n)=\beta^{-1} x_0^{n-2} x_0^{-(n-2)}$, and so once again $\eta(w)=\eta(u)x_1$.

\noindent {\it Case 3d. Suppose that $A(w) \neq \emptyset$, and
caret $n$ is the right child of its parent in $T_+(w)$.}  Since $n
\geq 3$, $N(T_+(w)) \ge n+1$, and caret ${\mathcal C}_n$ is the
first exposed caret in $T_+(w)$, then caret ${\mathcal C}_{n-1}$
must be an interior caret in $T_+(w)$, which is contained in the
right subtree of each ${\mathcal C}_j$ with $2 \le j \le n-2$.
Then the nested traversal path for $u_p^{-1}$ is the same as that
for $w_p^{-1}$ except that the $x_0^{-1}x_1^{-1}x_0x_1^{-1}$
subword of $\eta(w_p^{-1})$ corresponding to carets ${\mathcal
C}_{n-1}$ and ${\mathcal C}_{n}$ is replaced with the word
$x_1^{-1}$ corresponding to the caret ${\mathcal C}_{n-1}$ of
$T_+(u)$. In this case
$\eta(w_p)\eta(w_n)=(\alpha^{-1}x_1x_0^{-1}x_1x_0^{n-2})(x_0^{-(n-2)})$,
(where $\beta=x_0x_1^{-1}\alpha$), and
$\eta(u_p)\eta(u_n)=(\alpha^{-1}x_1x_0^{n-3})(x_0^{-(n-2)})$. Therefore, $\eta(w)=\alpha^{-1}x_1x_0^{-1}x_1=\eta(u)x_1$.
\end{proof}

Lemmas \ref{x0generator} and \ref{x1generator} complete the proof
of Theorem \ref{goodedges}. In the next sections, we will most
frequently apply the contrapositive of Theorem \ref{goodedges},
rewritten below following Notation \ref{not:4trees}.

\begin{corollary}\label{newbad}
Let $w \in F$. If $e_a(w)$ is a bad edge, then $a=1$, the tree
$T_-(w)$ has at least 3 right carets, and either
\begin{enumerate}
\item $D(w)\neq R_{N_D(w)}$.
\item $D(w)= R_{N_D(w)}$, $A(w) = L_{N_A(w)}$, property $(\ddagger)$ holds,
  and $2 \leq j(w)\leq N_A(w)$.
\item $D(w)= R_{N_D(w)}$, $A(w) \neq L_{N_A(w)}$, and
property $(\ddagger)$ holds.
\end{enumerate}
\end{corollary}

\begin{proof}
From Theorem \ref{goodedges} parts (1) and (2) we know that $a=1$
and $T_-(w)$ contains at least 3 right carets.

If no caret is removed in the multiplication $wx_1^{-1}$ (that is, if
property $(\ddagger)$ fails), then
Part (3) of
Theorem \ref{goodedges} shows that we must have
$D(w)\neq R_{N_D(w)}$.

If a caret is removed in the
multiplication $wx_1^{-1}$, then property $(\ddagger)$ holds,
Additionally, if we are not in either of the
cases (1) or (3) of this corollary, then we have $D(w)=R_{N_D(w)}$
and  $A(w) = L_{N_A(w)}$.
In this case, $T_-(w)$ has no interior carets,
so $B(w) = \emptyset$ and
the caret that is canceled in the multiplication is
caret number $N_A(w)+2$, which must be exposed in $T_+(w)$.
It follows from
part (4) of Theorem \ref{goodedges} that this caret is
 not the first exposed caret in
$T_+(w)$, and so $j(w)< N_A(w)+2$. However, two consecutive carets cannot
be exposed, and we conclude that $j(w) \leq N_A(w)$. Furthermore, if
$j(w)=1$, then caret 1 would be exposed in both $T_-(w)$ and
$T_+(w)$ and the tree pair diagram would not be reduced. Hence, $2
\leq j(w) \leq N_A(w)$, and case (2) of the corollary holds.
\end{proof}


\subsection{Defining a partial order on the bad edges}

We now define a partial order on the set of all bad edges $e_1(w)$
as required for Theorem$~\ref{combingextends}$. This partial order
is based on numerical measures related to the tree pair diagram
for $w$. These include $N(w)$, as well as $N_A(w)$ and $N_D(w)$,
the number of carets in the subtrees $A(w)$ and $D(w)$ defined in
Notation \ref{not:4trees} above. To order the edges $e_1(w)$ and $e_1(w')$ where the values
$N_A$ and $N_D$ are the same for both elements, we first need to construct, for
each fixed number $k$, several different partial orderings of the
set of all rooted binary trees with $k$ carets. Before explaining these posets, we
first need some additional combinatorial information associated to
a rooted binary tree.

\begin{definition}\label{generalTinfo}
Let $T$ be a rooted, binary tree.
\begin{itemize}
\item We order the right carets of $T$ in infix order, and call
them $r_1, r_2, r_3, \ldots, r_{k}$, where $r_1$ is the root caret
of $T$. Let $T_i$ be the (possibly empty) left subtree of caret
$r_i$. Let $s_r(T):=i$, where $i$ is the smallest index with $0
\leq i \leq k$, with the property that for every $i < t \leq k$,
$T_t$ is empty.

\medskip

\item Similarly, we call
the left carets of $T$, in infix order, $l_m, l_{m-1}, \dots ,
l_1$, where $l_1$ is the root caret of $T$, and let $S_i$ be the
(possibly empty) right subtree of caret $l_i$. Then let
$s_l(T):=i$, where $i$ is the smallest index , $0 \leq i \leq m$,
such that $S_t$ is empty for every $i<t \leq m$.

\medskip

\item Let
$C_r(T):=N(T)-(k-s_r(T))$ where $k$ is the number of right carets
in $T$; that is, $C_r(T)$ is the number of carets in $T$ up to and
including caret $r_{s_r(T)}$.

\medskip

\item Let $C_l(T):=N(T)-( m-s_l(T))$ where $m$ is the number of
left carets in $T$; that is, $C_l(T)$ is the number of carets in
$T$ after, and including, caret $l_{s_l(T)}$.
\end{itemize}
\end{definition}

We remark that the simple condition of whether a tree
consists either only of right carets or only of left carets, which
was critical in recognizing bad edges in Corollary \ref{newbad},
simply translates into whether $s_r$ or $s_l$ equals
zero. More precisely, the condition $s_r(T)=0$ (respectively
$s_r(T)>0$) is equivalent to $T=R_{N(T)}$ (respectively $T\neq
R_{N(T)}$). Similarly, the condition $s_l(T)=0$ (respectively
$s_l(T)>0$) is equivalent to $T=L_{N(T)}$ (respectively $T\neq
L_{N(T)}$). In order to sort, rather than simply recognize, the
bad edges, however, we need to keep track of the numerical values $s_r$ and
$s_l$.

Consider the set of rooted binary trees with $k$ carets.  We
define the \textbf{right poset of rooted binary trees with $k$
carets} which will be used to order edges $e_1(w)$ where $N_D(w) =
k$. For each tree $D$ with $k$ carets with $s_r(D)
> 0$, we define the tree $f(D)$ as follows:

\begin{itemize}
\item If $s_r(D)$ is odd, and $T_1$, the left subtree of the root
caret of $D$, is empty, $f(D)$ is the tree formed by rotating $D$
to the left at caret $r_1$. That is, if $g$ is the element of $F$ with tree
pair diagram $(D,R_k)$ where $R_k$ is the tree consisting
of $k$ right carets, then $gx_0^{-1}$ has (possibly
unreduced) tree pair diagram
$(f(D),R_k)$.

\medskip

\item If $s_r(D)$ is
odd, and $T_1$ is not empty, $f(D)$ is the tree formed by rotating
$D$ to the right at caret $r_1$. That is, if $g$ is the element of $F$ with
tree pair diagram $(D,R_k)$, then $gx_0$ has tree pair diagram
$(f(D),R_k)$.

\medskip

\item If $s_r(D)$ is even, and $T_2$, the left subtree
of the right child of the root caret of $D$, is empty, $f(D)$ is
the tree formed by rotating $D$ to the left at caret $r_2$. If $g$
is the element of $F$ with tree pair diagram $(D,R_k)$, then
$gx_1^{-1}$ has tree pair diagram $(f(D),R_k)$.

\medskip

\item If $s_r(D)$ is
even, and $T_2$ is not empty, $f(D)$ is the tree formed by
rotating $D$ to the right at caret $r_2$. If $g$ is the element of
$F$ with tree pair diagram $(D,R_k)$, then $gx_1$ has tree pair
diagram $(f(D),R_k)$.
\end{itemize}

Now declare $f(D) <_r D$ for every $D$. We claim that the
transitive closure of this order is a well-founded partial order,
with unique minimal element $R_k$, the tree with $k$ right
carets. To see this, notice that $C_r(D)=0$ if and only if
$D=R_k$. Now $C_r(f(D)) \leq C_r(D)$, and if $C_r(f(D)) = C_r(D)$,
 then $s_r(D)$ and $s_r(f(D))$ have different parities. So
if $f^n(D)=D$ for some positive integer $n$, this implies that
there is a word $x_0^{\pm 1} x_1 ^{\pm 1} \cdots x_0^{\pm1}
x_1^{\pm1}$ (where possibly the first and/or last generators are
absent) which is trivial in $F$, contradicting
Lemma~\ref{consecutiveindices}.

Since there are only a finite number of trees with $k$
carets, $C_r(f^m(D)) < C_r(D)$ for some $m$, and hence
$C_r(f^n(D))=0$ for some $n$. Hence, we see that this is a partial
order with a unique minimal tree $R_k$, which is less than all
other trees in the poset. We denote the order in this poset by
$<_r$.

We now define the \textbf{left posets of rooted binary trees with
$k$ carets}, which will be used to sort bad edges $e_1(w)$ for
which $N_A(w)=k$. Using the method given above, we could have
constructed a poset using $s_l$, $S_i$, and $C_l$ instead of
$s_r$, $T_i$, and $C_r$, replacing the words \lq\lq rotate left "
by \lq\lq rotate right" and vice-versa. This yields a dual poset,
where the minimal element is the tree $L_k$ consisting of only
left carets. We denote relationships in this order by $A_1 <_l
A_2$.

However, in some cases we will need a modification of this left
poset in order to sort our edges, depending on an index $1 \leq j
\leq k$. For any natural numbers $k \ge 3$ and $2 \leq j \leq
k-1$, let $B_j(k)$ be a tree consisting of $k$ carets, none of
which are interior, so that the root caret has infix number $j+1$.
Note that $B_{k-1}(k)=L_k$, the tree consisting of $k$ left
carets. In the left poset with order relation $<_l$, there is a
unique path from each tree to the minimal element
$B_{k-1}(k)=L_k$, and hence there also is a unique (undirected)
path from each tree to $B_j(k)$. For each $2 \leq j\leq k-2$, we
form a new poset, reordering the trees by declaring $A_1 <_l^j
A_2$ if $A_1$ is on the unique path from $A_2$ to $B_j(k)$. For
each such $j$, the new poset now has least element $B_j(k)$, and
whereas $L_k=B_{k-1}(k) <_l B_{k-2}(k)<_l \cdots <_l
B_{j+1}(k)<_lB_j(k)$, exactly the reverse holds in $<_l^j$, namely
$B_{j}(k) <_l^j B_{j+1}(k)<_l^j \cdots <_l^j
B_{k-2}(k)<_l^jB_{k-1}(k)=L_k$. If $j=1,k-1$ or $k$, we use the
original poset, and declare $<_l^j=<_l$. Thus we have constructed
only $k-2$ distinct posets in all, for each $k \geq 3$. In the
trivial cases $k=1$ and $k=2$, simply declare $<_l^j=<_l$ for any
$1\leq j \leq k$.

To summarize: for each natural number $k \geq 3$, we have defined
$k-1$ distinct partial orderings of the set of rooted binary trees
with $k$ carets. There is a unique right poset which has as
minimal element $R_k$, which will be used to sort bad edges $e_1(w)$ with
$N_D(w)=k$; there is a family of $k-2$ distinct left posets which
have, respectively, the trees $B_j(k)$ for $2\leq j \leq k-1$ as
unique minimal elements, which will be used to sort edges with
$N_A(w)=k$ and $j(w)=j$.

The following notation, based on the quantities introduced in
Notation \ref{not:4trees} and Definition \ref{generalTinfo}, will
simplify the description of the ordering.

\begin{notation}\label{not:ss}
Let $e_1(w)$ be a bad edge, for an element $w=(T_-(w),T_+(w)) \in
F$.
\begin{itemize}
\item  Let $s_r(w):=s_r(D(w))$.

\item Let $s_l(w):=s_l(A(w))$. \item Let $C_r(w):=C_r(D(w))$. \item
Let $C_l(w):=C_l(A(w))$.\item Let $n(w)$ be the infix number of the
right caret of $T_-(w)$ whose left subtree is not empty, but whose
right subtree is either empty or consists only of right carets. If
no such caret exists, $T_-(w)$ consists only of right carets, and
we set $n(w)=0$.
\end{itemize}
\end{notation}

\begin{figure}
\begin{center}
\includegraphics[width=2in]{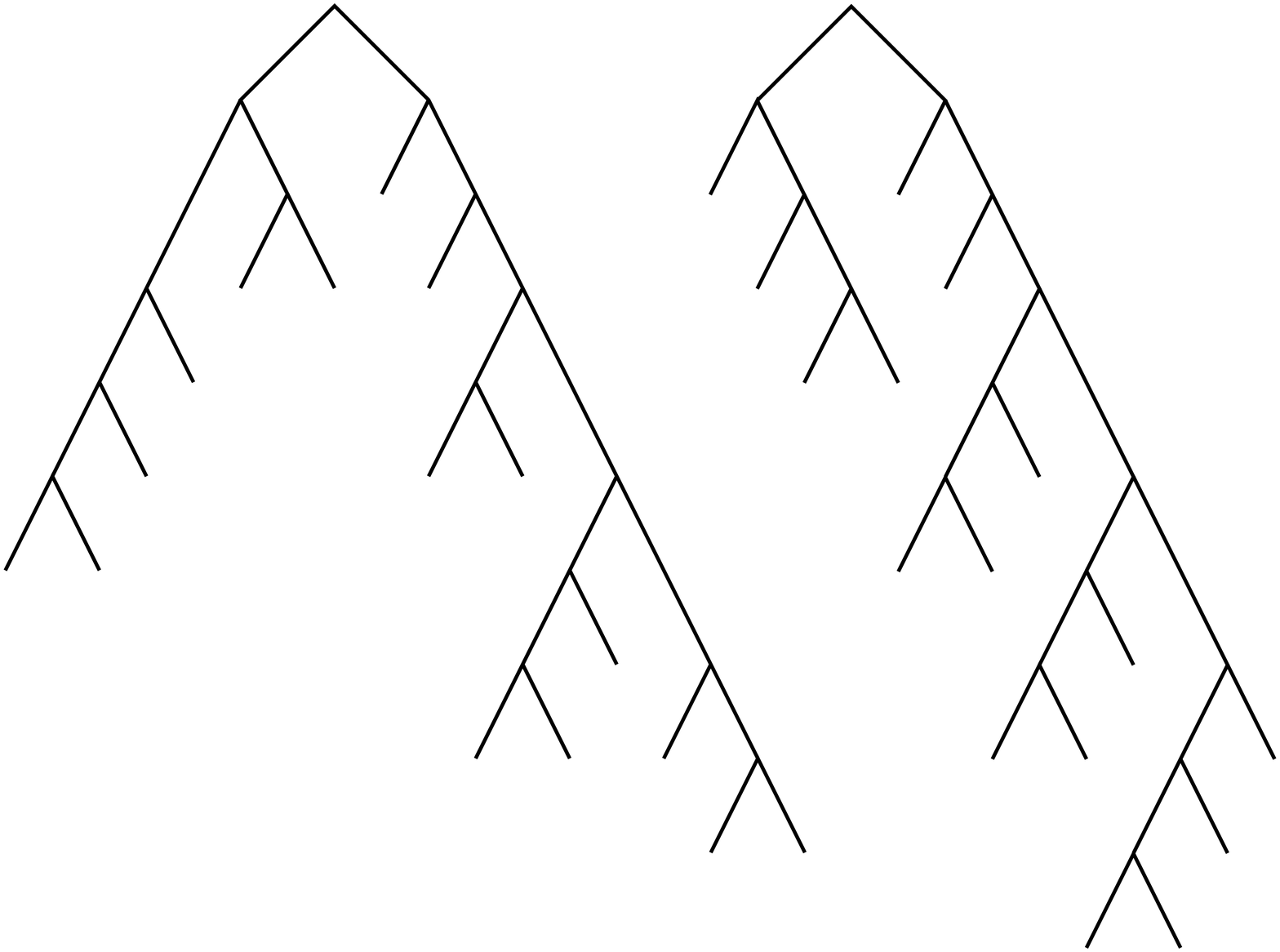}
\caption{In the example of the pair of trees $(T_-(w),T_+(w))$
given above, the subtree $D(w)$ (resp. $A(w)$) has four right
(resp. left) carets.  We compute the following quantities:
$N(w)=15, \ N_D(w)=7, \ N_A(w)=5, \ s_r(w)=2, \ s_l(w) = 1, \
C_r(w) = 5, \ C_l(w)=2, \ n(w) = 13$ and $j(w) = 3$.}
\label{fig:example}
\end{center}
\end{figure}

In the following definition, we define a set of comparisons
between certain pairs of bad edges. We then prove that the
transitive closure of this set of order relationships is a partial
order. Some details of this partial order (particularly the fourth
set of comparisons) may seem mysterious at this point, but they
are exactly the relationships needed for the cell map from the set of bad
edges into the 2-cells which is defined in the next section to satisfy the hypotheses of
Theorem~\ref{combingextends}.

\begin{definition}\label{def:order}
Let $e_1(w)$ and $e_1(z)$ be bad edges. We say $e_1(z)<e_1(w)$ in
the following situations:
\begin{enumerate}
\item If $N(z) < N(w)$.

 \item If $N(z) = N(w)$, $T_+(z)=T_+(w)$, both $s_r(w)>0$
and $s_r(z)>0$, and either:

\begin{enumerate}

\item $N_D(z) < N_D(w)$ and $n(z)=n(w)$, or

\item $N_D(z) = N_D(w)$, $n(z) \leq n(w)$, and $D(z) <_r D(w)$.

\end{enumerate}

\item If $N(z) = N(w)$, $T_+(z)=T_+(w)$, $s_r(z)=s_r(w)=0$, and
either:

\begin{enumerate}

\item  $N_A(z) < N_A(w)$ and $n(z) \leq n(w)$ or

\item  $N_A(w) = N_A(z)$, $n(z)=n(w)$, and $A(z) <_l^j A(w)$ for
$j=j(w)=j(z) \leq N_A(w)$.

\end{enumerate}

\item If $N(z)=N(w)$, $T_+(z)=T_+(w)$, exactly one of the pair
$\{s_r(w), s_r(z) \}$ is zero, and either:
\begin{enumerate} \item $s_r(z)=0$, $s_r(w)=1$ or $2$, and $n(z) < n(w)$, or
\item $s_r(z)=1$, $s_r(w)=0$, $n(z)=n(w)$, and $N_A(z)<
N_A(w)$.
\end{enumerate}

\end{enumerate}
\end{definition}

\begin{lemma}
The transitive closure of the set of order relationships defined
above is a partial order satisfying the property that
for all bad edges
$e$, the set of bad edges less than $e$ with respect to this partial
order is finite.
\end{lemma}

\begin{proof}
In order to show this is a partial order, we must show that for
every set of bad edges satisfying $e_1(w_1) > e_1(w_2) > \cdots
> e_1(w_n)$,  $w_1 \neq w_n$. Suppose $e_1(w_1) > e_1(w_2) > \cdots
> e_1(w_n)$. If $N(w_i)$ is not constant for all $i$, then
$N(w_n)<N(w_1)$, and so $w_1\neq w_n$. So we may assume $N=N(w_i)$
for $1 \leq i \leq n$. Next we observe that if $s_r(w_i)>0$ for
every $1 \leq i \leq n$, then for each $i$ either $N_D(w_{i+1}) <
N_D(w_i)$, or else $N_D(w_{i+1}) = N_D(w_i)$ and $D(w_{i+1})<_l
D(w_i)$, so $w_1 \neq w_n$. Similarly, if $s_r(w_i)=0$ for every
$i$, either $N_A(w_{i+1}) < N_A(w_i)$, or else $N_A(w_{i+1}) =
N_A(w_i)$ and $A(w_{i+1})<_l^j A(w_i)$ for $j=j(w_i)=j(w_{i+1})$,
so $w_1 \neq w_n$. Therefore, if $w_1=w_n$ the value of the
variable $s_r$ must change twice in the sequence of edges between
a strictly positive value and 0. Thus there must be indices for
which the value of $s_r$ increases from 0 to a (strictly) positive
number and for which the value decreases from positive to 0. In
particular, there must be some index $i$ for which $s_r(D(w_i))=1$
or 2, $s_r(D(w_{i+1}))=0$, and $n(w_{i+1}) < n(w_i)$. But since
for every index $j$ we have $n(w_{j+1}) \leq n(w_j)$, then
$n(w_n)<n(w_1)$, and hence $w_1 \neq w_n$. Finally, since the
subset of all edges $e_1(w)$ with a fixed value of $N(w)$ is
finite, the finiteness condition is satisfied and this partial
order is well-founded.
\end{proof}


\subsection{The mapping from the set of bad edges to the set of
$2$-cells in the Cayley complex}

In this section we define a mapping $c$ from the set of bad edges
to the set of 2-cells in the Cayley complex.
We will set up the map $c$ so that the bad edge $e_1(w)$ is on
the boundary of
the cell $c(e_1(w))$.

In order to specify
this mapping, we will first define notation
for 2-cells in the Cayley complex
with a specified basepoint and orientation.
For each vertex $w$ and edge $e_1(w)$ in the
Cayley complex, there are eight 2-cells containing
this edge in their boundaries.  For four of
these 2-cells, there are 10 edges on the boundary;
these are the 2-cells labeled $Rr_1^{\pm 1}(w)$ and
$Rl_1^{\pm 1}(w)$ in Figure \ref{fig:relators1}.
For the other four 2-cells whose boundaries contain
$e_1(w)$, there are 14 boundary edges;
these are the 2-cells labeled $Rr_2^{\pm 1}(w)$ and
$Rl_2^{\pm 1}(w)$ in Figure \ref{fig:relators2}.

In each of these 2-cells, in addition to $e_1(w)$ the boundary
contains three other edges
of the form $e_1(v)$ for some $v \in F$, and none of the
$e_1$ edges in the boundary of a particular $2$-cell are adjacent.
The edge $e_1(w)$ will be referred to as the top
$e_1$ edge in these eight 2-cells.
The $e_1$ edges closest to $w$ and $wx_1^{-1}$ are the left and
right side edges $e_1(z_l)$ and $e_1(z_r)$, respectively,
and the last $e_1$ edge is the bottom edge $e_1(z_b)$.

For a bad edge $e_1(w)$, the 2-cell $c(e_1(w))$ must be chosen from
among these eight cells.
The map will be defined so that $z_b$ can be represented by a (not
necessarily reduced) tree pair diagram $(T'_-(z_b), T'_+(z_b))$, where the
negative trees $T_-(w)$ and $T'_-(z_b)$ differ by a single rotation at a
particular caret, and the positive trees satisfy $T_+(w)=T'_+(z_b)$.
The notation $Ra_n^{\pm1}(w)$ (where $R$ stands for relator,)
has been motivated by this.  The letter
$a=l$ or $a=r$ depends on whether the rotation needed to
transform $T_-(w)$ to $T'_-(z_b)$ takes place at a left or right
caret of $T_-(w)$. The superscript $\pm 1$ takes into account the
direction of this rotation, and the subscript $n$ specifies at
which caret the rotation takes place. More specifically, in the
case of a rotation at a left caret, $n=1$ means this caret is the
left child of the root of $T_-(w)$, while $n=2$ means rotation is
at the left child of the left child of the root. In the case of a
rotation at a right caret, if caret $m$ is the right child of the
right child of the root of $T_-(w)$, then $n=1$ means rotating at
the right child of caret $m$, and $n=2$ means rotating at the
right child of the right child of caret $m$.

\begin{figure}[ht]
\begin{center}
\begin{tikzpicture}[>=triangle 45,scale=.8]
\draw[->] (0,0)--(1,0); \draw (1,0)--(2,0); \draw[->]
(2,0)--(3,0); \draw (3,0)--(4,0); \draw (4,0)--(5,0); \draw[<-]
(5,0)--(6,0); \draw (6,0)--(7,0); \draw[<-] (7,0)--(8,0); \fill
(0,0) circle (2pt); \fill (2,0) circle (2pt); \fill (4,0) circle
(2pt); \fill (6,0) circle (2pt); \fill (8,0) circle (2pt);
\draw[->] (0,2)--(1,2); \draw (1,2)--(2,2); \draw[->]
(2,2)--(3,2); \draw (3,2)--(4,2); \draw[very thick] (4,2)--(5,2);
\draw[<-, very thick] (5,2)--(6,2); \draw (6,2)--(7,2); \draw[<-]
(7,2)--(8,2); \fill (0,2) circle (2pt); \fill (2,2) circle (2pt);
\fill (4,2) circle (2pt); \fill (6,2) circle (2pt); \fill (8,2)
circle (2pt); \draw[->] (0,0)--(0,1); \draw(0,1)--(0,2); \draw[->]
(8,0)--(8,1); \draw (8,1)--(8,2);

\draw (1,0) node [anchor=north]{$x_0$}; \draw (3,0) node
[anchor=north]{$x_0$}; \draw (5,0) node [anchor=north]{$x_1$};
\draw (7,0) node [anchor=north]{$x_0$}; \draw (1,2) node
[anchor=south]{$x_0$}; \draw (3,2) node [anchor=south]{$x_0$};
\draw (5,2) node [anchor=south]{$x_1$}; \draw (7,2) node
[anchor=south]{$x_0$}; \draw (0,1) node [anchor=east]{$x_1$};
\draw (8,1) node [anchor=west]{$x_1$}; \draw (4,1.5) node
[anchor=north] {\textbf{$Rr_1(w)$}}; \draw (4,0) node
[anchor=north] {$z_b$}; \draw[<-] (4,2.2)--(4,3); \draw (4,3)
node[anchor=south]{$w$}; \draw (0,2) node [anchor=east]{$z_l$};
\draw (8,2) node [anchor=west]{$z_r$};


\draw[->] (10,0)--(11,0); \draw (11,0)--(12,0); \draw[->]
(12,0)--(13,0); \draw (13,0)--(14,0); \draw (14,0)--(15,0);
\draw[<-] (15,0)--(16,0); \draw[-<] (16,0)--(17,0); \draw
(17,0)--(18,0);

\fill (10,0) circle (2pt); \fill (12,0) circle (2pt); \fill (14,0)
circle (2pt); \fill (16,0) circle (2pt); \fill (18,0) circle
(2pt); \draw[->] (10,2)--(11,2); \draw (11,2)--(12,2); \draw[->]
(12,2)--(13,2); \draw (13,2)--(14,2); \draw[very thick]
(14,2)--(15,2); \draw[<-, very thick] (15,2)--(16,2); \draw[-<]
(16,2)--(17,2); \draw (17,2)--(18,2); \fill (10,2) circle (2pt);
\fill (12,2) circle (2pt); \fill (14,2) circle (2pt); \fill (16,2)
circle (2pt); \fill (18,2) circle (2pt); \draw (10,0)--(10,1);
\draw[<-](10,1)--(10,2); \draw (18,0)--(18,1); \draw[<-]
(18,1)--(18,2);

\draw (11,0) node [anchor=north]{$x_0$}; \draw (13,0) node
[anchor=north]{$x_0$}; \draw (15,0) node [anchor=north]{$x_1$};
\draw (17,0) node [anchor=north]{$x_0$}; \draw (11,2) node
[anchor=south]{$x_0$}; \draw (13,2) node [anchor=south]{$x_0$};
\draw (15,2) node [anchor=south]{$x_1$}; \draw (17,2) node
[anchor=south]{$x_0$}; \draw (10,1) node [anchor=east]{$x_1$};
\draw (18,1) node [anchor=west]{$x_1$}; \draw (14,1.5) node
[anchor=north] {\textbf{$Rr_1^{-1}(w)$}}; \draw (14,0) node
[anchor=north] {$z_b$}; \draw[<-] (14,2.2)--(14,3); \draw (14,3)
node[anchor=south]{$w$}; \draw (10,0) node [anchor=east]{$z_l$};
\draw (18,0) node [anchor=west]{$z_r$};
\end{tikzpicture}

\begin{tikzpicture}[>=triangle 45,scale=.8]
\draw[-<] (0,0)--(1.33,0); \draw (1.33,0)--(2.66,0); \draw[-<]
(2.66,0)--(4,0); \draw (4,0)--(5.33,0); \draw[->]
(5.33,0)--(6.66,0); \draw (6.66,0)--(8,0);  \fill (0,0) circle
(2pt); \fill (2.66,0) circle (2pt); \fill (5.33,0) circle (2pt);
\fill (8,0) circle (2pt);  \draw[-<] (0,2)--(.75,2); \draw
(.75,2)--(1.5,2); \draw[-<] (1.5,2)--(2.25,2); \draw
(2.25,2)--(3,2); \draw[-<,very thick] (3,2)--(4,2); \draw[very
thick] (3.75,2)--(5,2); \draw[->] (5,2)--(5.75,2); \draw
(5.75,2)--(6.5,2); \draw[->] (6.5,2)--(7.25,2); \draw
(7.25,2)--(8,2); \fill (0,2) circle (2pt); \fill (1.5,2) circle
(2pt); \fill (3,2) circle (2pt); \fill (5,2) circle (2pt); \fill
(6.5,2) circle (2pt); \fill (8,2) circle (2pt);  \draw[->]
(0,0)--(0,1); \draw(0,1)--(0,2); \draw[->] (8,0)--(8,1); \draw
(8,1)--(8,2);

\draw (1.33,0) node [anchor=north]{$x_0$}; \draw (4,0) node
[anchor=north]{$x_1$}; \draw (6.66,0) node [anchor=north]{$x_0$};
 \draw (.75,2) node
[anchor=south]{$x_0$}; \draw (2.25,2) node [anchor=south]{$x_0$};
\draw (4,2) node [anchor=south]{$x_1$}; \draw (5.75,2) node
[anchor=south]{$x_0$}; \draw (7.25,2) node [anchor=south]{$x_0$};
\draw (0,1) node [anchor=east]{$x_1$}; \draw (8,1) node
[anchor=west]{$x_1$}; \draw (4,1.5) node [anchor=north]
{\textbf{$Rl_1(w)$}}; \draw (2.66,0) node [anchor=north] {$z_b$};
\draw[<-] (3,2.2)--(3,3); \draw (3,3) node[anchor=south]{$w$};
\draw (0,2) node [anchor=east]{$z_l$}; \draw (8,2) node
[anchor=west]{$z_r$};

\draw[-<] (10,2)--(11.33,2); \draw (11.33,2)--(12.66,2);
\draw[very thick] (12.66,2)--(14,2); \draw[<-,very thick]
(14,2)--(15.33,2); \draw[->] (15.33,2)--(16.66,2); \draw
(16.66,2)--(18,2);  \fill (10,2) circle (2pt); \fill (12.66,2)
circle (2pt); \fill (15.33,2) circle (2pt); \fill (18,2) circle
(2pt); \draw[-<] (10,0)--(10.75,0); \draw (10.75,0)--(11.5,0);
\draw[-<] (11.5,0)--(12.25,0); \draw (12.25,0)--(13,0); \draw[-<]
(13,0)--(14,0); \draw (14,0)--(15,0); \draw[->] (15,0)--(15.75,0);
\draw (15.75,0)--(16.5,0); \draw[->] (16.5,0)--(17.25,0); \draw
(17.25,0)--(18,0); \fill (10,0) circle (2pt); \fill (11.5,0)
circle (2pt); \fill (13,0) circle (2pt); \fill (15,0) circle
(2pt); \fill (16.5,0) circle (2pt); \fill (18,0) circle (2pt);
\draw[-<] (10,0)--(10,1); \draw(10,1)--(10,2); \draw[-<]
(18,0)--(18,1); \draw (18,1)--(18,2);

\draw (11.33,2) node [anchor=south]{$x_0$}; \draw (14,2) node
[anchor=south]{$x_1$}; \draw (16.66,2) node [anchor=south]{$x_0$};
 \draw (10.75,0) node
[anchor=north]{$x_0$}; \draw (12.25,0) node [anchor=north]{$x_0$};
\draw (14,0) node [anchor=north]{$x_1$}; \draw (15.75,0) node
[anchor=north]{$x_0$}; \draw (17.25,0) node [anchor=north]{$x_0$};
\draw (10,1) node [anchor=east]{$x_1$}; \draw (18,1) node
[anchor=west]{$x_1$}; \draw (14,1.5) node [anchor=north]
{\textbf{$Rl_1^{-1}(w)$}}; \draw (13,0) node [anchor=north]
{$z_b$}; \draw[<-] (12.66,2.2)--(12.66,3); \draw (12.66,3)
node[anchor=south]{$w$}; \draw (10,0) node [anchor=east]{$z_l$};
\draw (18,0) node [anchor=west]{$z_r$};
\end{tikzpicture}

\caption{The four 2-cells $Rr_1^{\pm 1}(w)$ and $Rl_1^{\pm 1}(w)$ with
boundary consisting of 10 edges including $e_1(w)$.  In each rectangle,
the vertices $w$, $z_l$, $z_r$, and $z_b$ are labeled.}\label{fig:relators1}
\end{center}
\end{figure}
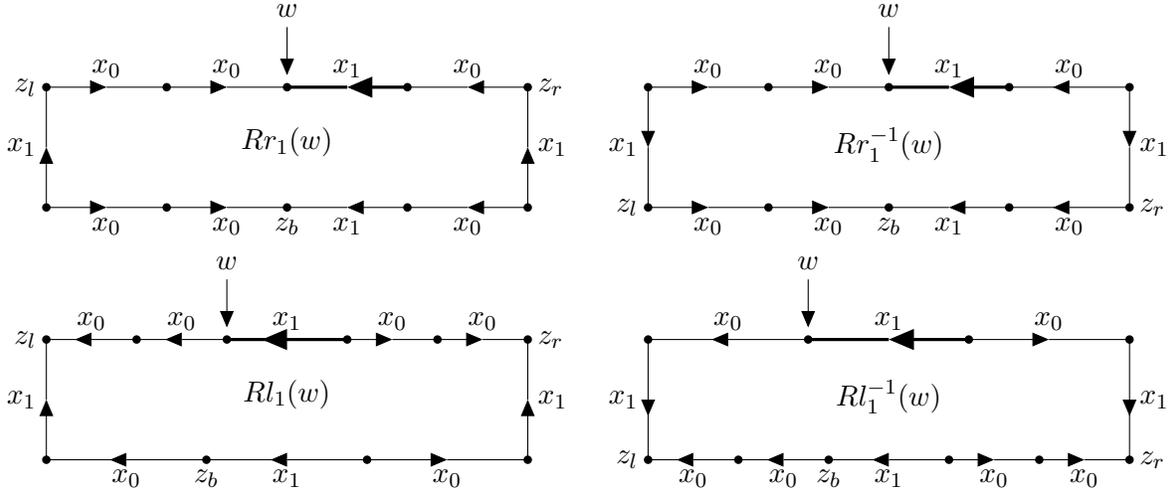

\begin{figure}[ht]
\begin{center}

%
%

\begin{tikzpicture}[>=triangle 45,scale=.95]
\fill (0,0) circle (2pt); \fill (1,0) circle (2pt); \fill (2,0)
circle (2pt); \fill (3,0) circle (2pt); \fill (4,0) circle (2pt);
\fill (5,0) circle (2pt); \fill (6,0) circle (2pt); \fill (0,2)
circle (2pt); \fill (1,2) circle (2pt); \fill (2,2) circle (2pt);
\fill (3,2) circle (2pt); \fill (4,2) circle (2pt); \fill (5,2)
circle (2pt); \fill (6,2) circle (2pt);

\draw (.5,0) node [anchor=north]{$x_0$}; \draw (1.5,0) node
[anchor=north]{$x_0$}; \draw (2.5,0) node [anchor=north]{$x_0$};
\draw (3.5,0) node [anchor=north]{$x_1$}; \draw (4.5,0) node
[anchor=north]{$x_0$}; \draw (5.5,0) node [anchor=north]{$x_0$};
\draw (.5,2) node [anchor=south]{$x_0$}; \draw (1.5,2) node
[anchor=south]{$x_0$}; \draw (2.5,2) node [anchor=south]{$x_0$};
\draw (3.5,2) node [anchor=south]{$x_1$}; \draw (4.5,2) node
[anchor=south]{$x_0$}; \draw (5.5,2) node [anchor=south]{$x_0$};
\draw (0,1) node [anchor=east]{$x_1$}; \draw (6,1) node
[anchor=west]{$x_1$}; \draw (3,1.5) node [anchor=north]
{\textbf{$Rr_2(w)$}};

\draw[->] (0,0)--(.5,0); \draw (.5,0)--(1,0); \draw[->]
(1,0)--(1.5,0); \draw (1.5,0)--(2,0); \draw[->] (2,0)--(2.5,0);
\draw (2.5,0)--(3,0); \draw[-<] (3,0)--(3.5,0); \draw
(3.5,0)--(4,0); \draw[-<] (4,0)--(4.5,0); \draw (4.5,0)--(5,0);
\draw[-<] (5,0)--(5.5,0); \draw (5.5,0)--(6,0); \draw[->]
(0,2)--(.5,2); \draw (.5,2)--(1,2); \draw[->] (1,2)--(1.5,2);
\draw (1.5,2)--(2,2); \draw[->] (2,2)--(2.5,2); \draw
(2.5,2)--(3,2); \draw[very thick] (3,2)--(3.5,2); \draw[<-,very
thick] (3.5,2)--(4,2); \draw[-<] (4,2)--(4.5,2); \draw
(4.5,2)--(5,2); \draw[-<] (5,2)--(5.5,2); \draw (5.5,2)--(6,2);
\draw[->] (0,0)--(0,1); \draw (0,1)--(0,2); \draw[->]
(6,0)--(6,1); \draw (6,1)--(6,2);

\draw (3,0) node [anchor=north] {$z_b$}; \draw[<-] (3,2.2)--(3,3);
\draw (3,3) node[anchor=south]{$w$}; \draw (0,2) node
[anchor=east]{$z_l$}; \draw (6,2) node [anchor=west]{$z_r$};

\fill (9,0) circle (2pt); \fill (10,0) circle (2pt); \fill (11,0)
circle (2pt); \fill (12,0) circle (2pt); \fill (13,0) circle
(2pt); \fill (14,0) circle (2pt); \fill (15,0) circle (2pt); \fill
(9,2) circle (2pt); \fill (10,2) circle (2pt); \fill (11,2) circle
(2pt); \fill (12,2) circle (2pt); \fill (13,2) circle (2pt); \fill
(14,2) circle (2pt); \fill (15,2) circle (2pt);

\draw (9.5,0) node [anchor=north]{$x_0$}; \draw (10.5,0) node
[anchor=north]{$x_0$}; \draw (11.5,0) node [anchor=north]{$x_0$};
\draw (12.5,0) node [anchor=north]{$x_1$}; \draw (13.5,0) node
[anchor=north]{$x_0$}; \draw (14.5,0) node [anchor=north]{$x_0$};
\draw (9.5,2) node [anchor=south]{$x_0$}; \draw (10.5,2) node
[anchor=south]{$x_0$}; \draw (11.5,2) node [anchor=south]{$x_0$};
\draw (12.5,2) node [anchor=south]{$x_1$}; \draw (13.5,2) node
[anchor=south]{$x_0$}; \draw (14.5,2) node [anchor=south]{$x_0$};
\draw (9,1) node [anchor=east]{$x_1$}; \draw (15,1) node
[anchor=west]{$x_1$}; \draw (12,1.5) node [anchor=north]
{\textbf{$Rr_2^{-1}(w)$}};

\draw[->] (9,0)--(9.5,0); \draw (9.5,0)--(10,0); \draw[->]
(10,0)--(10.5,0); \draw (10.5,0)--(11,0); \draw[->]
(11,0)--(11.5,0); \draw (11.5,0)--(12,0); \draw[-<]
(12,0)--(12.5,0); \draw (12.5,0)--(13,0); \draw[-<]
(13,0)--(13.5,0); \draw (13.5,0)--(14,0); \draw[-<]
(14,0)--(14.5,0); \draw (14.5,0)--(15,0); \draw[->]
(9,2)--(9.5,2); \draw (9.5,2)--(10,2); \draw[->] (10,2)--(10.5,2);
\draw (10.5,2)--(11,2); \draw[->] (11,2)--(11.5,2); \draw
(11.5,2)--(12,2); \draw[very thick] (12,2)--(12.5,2);
\draw[<-,very thick] (12.5,2)--(13,2); \draw[-<] (13,2)--(13.5,2);
\draw (13.5,2)--(14,2); \draw[-<] (14,2)--(14.5,2); \draw
(14.5,2)--(15,2); \draw (9,0)--(9,1); \draw[<-] (9,1)--(9,2);
\draw (15,0)--(15,1); \draw[<-] (15,1)--(15,2);

\draw (12,0) node [anchor=north] {$z_b$}; \draw[<-]
(12,2.2)--(12,3); \draw (12,3) node[anchor=south]{$w$}; \draw
(9,0) node [anchor=east]{$z_l$}; \draw (15,0) node
[anchor=west]{$z_r$};
\end{tikzpicture}

\begin{tikzpicture}[>=triangle 45]
\fill (0,4) circle (2pt); \fill (1.5,4) circle (2pt); \fill (3,4)
circle (2pt); \fill (5,4) circle (2pt); \fill (6.5,4) circle
(2pt); \fill (8,4) circle (2pt); \fill (0,6) circle (2pt); \fill
(1,6) circle (2pt); \fill (2,6) circle (2pt); \fill (3,6) circle
(2pt); \fill (5,6) circle (2pt); \fill (6,6) circle (2pt); \fill
(7,6) circle (2pt); \fill (8,6) circle (2pt);

\draw (.75,4) node [anchor=north]{$x_0$}; \draw (2.25,4) node
[anchor=north]{$x_0$}; \draw (4,4) node [anchor=north]{$x_1$};
\draw (5.75,4) node [anchor=north]{$x_0$}; \draw (7.25,4) node
[anchor=north]{$x_0$};

\draw (.5,6) node [anchor=south]{$x_0$}; \draw (1.5,6) node
[anchor=south]{$x_0$}; \draw (2.5,6) node [anchor=south]{$x_0$};
\draw (4,6) node [anchor=south]{$x_1$}; \draw (5.5,6) node
[anchor=south]{$x_0$}; \draw (6.5,6) node [anchor=south]{$x_0$};
\draw (7.5,6) node [anchor=south]{$x_0$};

\draw (0,5) node [anchor=east]{$x_1$}; \draw (8,5) node
[anchor=west]{$x_1$}; \draw (4,5.5) node [anchor=north]
{\textbf{$Rl_2(w)$}};

\draw (0,4)--(.75,4); \draw[<-] (.75,4)--(1.5,4); \draw
(1.5,4)--(2.25,4); \draw[<-] (2.25,4)--(3,4); \draw (3,4)--(4,4);
\draw[<-] (4,4)--(5,4); \draw[->] (5,4)--(5.75,4); \draw
(5.75,4)--(6.5,4); \draw[->] (6.5,4)--(7.25,4); \draw
(7.25,4)--(8,4);

\draw (0,6)--(.5,6); \draw[<-] (.5,6)--(1,6); \draw
(1,6)--(1.5,6); \draw[<-] (1.5,6)--(2,6); \draw (2,6)--(2.5,6);
\draw[<-] (2.5,6)--(3,6); \draw[very thick] (3,6)--(4,6);
\draw[<-,very thick] (4,6)--(5,6); \draw[->] (5,6)--(5.5,6); \draw
(5.5,6)--(6,6); \draw[->] (6,6)--(6.5,6); \draw (6.5,6)--(7,6);
\draw[->] (7,6)--(7.5,6); \draw (7.5,6)--(8,6);

\draw[->] (0,4)--(0,5); \draw (0,5)--(0,6); \draw[->]
(8,4)--(8,5); \draw (8,5)--(8,6);

\draw (3,4) node [anchor=north] {$z_b$}; \draw[<-] (3,6.2)--(3,7);
\draw (3,7) node[anchor=south]{$w$}; \draw (0,6) node
[anchor=east]{$z_l$}; \draw (8,6) node [anchor=west]{$z_r$};

\fill (0,2) circle (2pt); \fill (1.5,2) circle (2pt); \fill (3,2)
circle (2pt); \fill (5,2) circle (2pt); \fill (6.5,2) circle
(2pt); \fill (8,2) circle (2pt); \fill (0,0) circle (2pt); \fill
(1,0) circle (2pt); \fill (2,0) circle (2pt); \fill (3,0) circle
(2pt); \fill (5,0) circle (2pt); \fill (6,0) circle (2pt); \fill
(7,0) circle (2pt); \fill (8,0) circle (2pt);

\draw (.75,2) node [anchor=south]{$x_0$}; \draw (2.25,2) node
[anchor=south]{$x_0$}; \draw (4,2) node [anchor=south]{$x_1$};
\draw (5.75,2) node [anchor=south]{$x_0$}; \draw (7.25,2) node
[anchor=south]{$x_0$};

\draw (.5,0) node [anchor=north]{$x_0$}; \draw (1.5,0) node
[anchor=north]{$x_0$}; \draw (2.5,0) node [anchor=north]{$x_0$};
\draw (4,0) node [anchor=north]{$x_1$}; \draw (5.5,0) node
[anchor=north]{$x_0$}; \draw (6.5,0) node [anchor=north]{$x_0$};
\draw (7.5,0) node [anchor=north]{$x_0$};

\draw (0,1) node [anchor=east]{$x_1$}; \draw (8,1) node
[anchor=west]{$x_1$}; \draw (4,1.5) node [anchor=north]
{\textbf{$Rl_2^{-1}(w)$}};

\draw (0,2)--(.75,2); \draw[<-] (.75,2)--(1.5,2); \draw
(1.5,2)--(2.25,2); \draw[<-] (2.25,2)--(3,2); \draw[very thick]
(3,2)--(4,2); \draw[<-,very thick] (4,2)--(5,2); \draw[->]
(5,2)--(5.75,2); \draw (5.75,2)--(6.5,2); \draw[->]
(6.5,2)--(7.25,2); \draw (7.25,2)--(8,2);

\draw (0,0)--(.5,0); \draw[<-] (.5,0)--(1,0); \draw
(1,0)--(1.5,0); \draw[<-] (1.5,0)--(2,0); \draw (2,0)--(2.5,0);
\draw[<-] (2.5,0)--(3,0); \draw (3,0)--(4,0); \draw[<-]
(4,0)--(5,0); \draw[->] (5,0)--(5.5,0); \draw (5.5,0)--(6,0);
\draw[->] (6,0)--(6.5,0); \draw (6.5,0)--(7,0); \draw[->]
(7,0)--(7.5,0); \draw (7.5,0)--(8,0);

\draw (0,0)--(0,1); \draw[<-] (0,1)--(0,2); \draw (8,0)--(8,1);
\draw[<-] (8,1)--(8,2);

\draw (3,0) node [anchor=north] {$z_b$}; \draw[<-] (3,2.2)--(3,3);
\draw (3,3) node[anchor=south]{$w$}; \draw (0,0) node
[anchor=east]{$z_l$}; \draw (8,0) node [anchor=west]{$z_r$};

\end{tikzpicture}
\caption{The four 2-cells $Rr_2^{\pm 1}(w)$ and $Rl_2^{\pm 1}(w)$ with
boundary consisting of 14 edges including $e_1(w)$.  In each rectangle,
the vertices $w$, $z_l$, $z_r$, and $z_b$ are labeled.}\label{fig:relators2}
\end{center}
\end{figure}
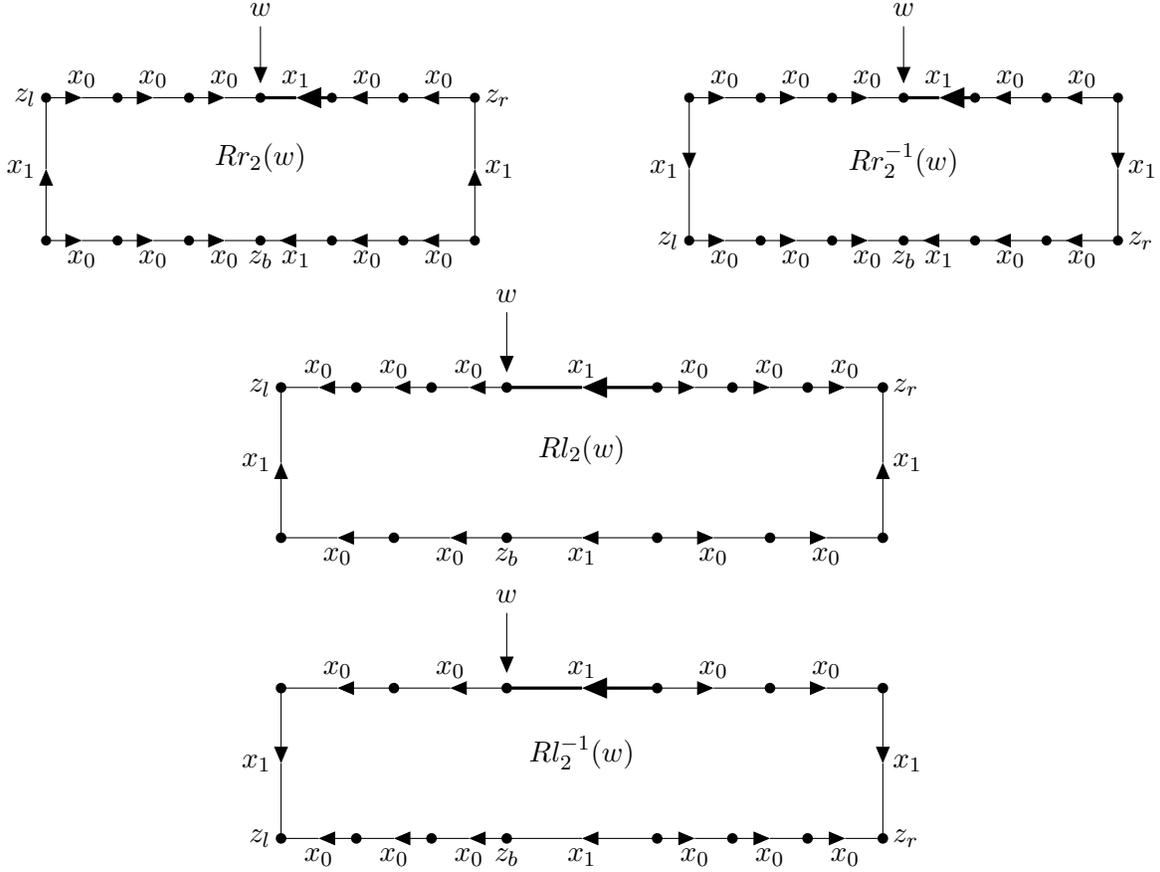

Rewriting the result of Corollary \ref{newbad} using
the quantities in Notation \ref{not:ss}, we have that
the bad edge $e_1(w)$ satisfies either
$s_r(w)>0$ or else property $(\ddagger)$ holds and either
$s_l(w)>0$ or $2 \leq j(w)\leq N_A(w)$.  It will be useful
to re-organize these cases for the definition of the map $c$,
as follows.

\begin{corollary}
Let $w \in F$.  If $e_a(w)$ is a bad edge,
then $a=1$, the tree $T_-(w)$ has at least 3 right carets, and either
\begin{enumerate}
\item $s_r(w)>0$,
\item $s_r(w)=0$, $s_l(w) \in \{0,1\}$, property $(\ddagger)$ holds,
$N_A(w) \ge 2$, and either
\begin{enumerate}
\item  $2 \leq j(w) \leq N_A(w)-1$ and $A(w)=B_{j(w)}(N_A(w))$,
\item $j(w)=N_A(w)$ and $A(w)=B_{N_A(w)-1}(N_A(w))$, or
\item  $2 \leq j(w) \leq N_A(w)-2$ and $A(w)=B_i(N_A(w))$ with
$j(w)+1 \leq i \leq N_A(w)-1$,
\end{enumerate}
or
\item $s_r(w)=0$, $s_l(w)>0$, property $(\ddagger)$ holds, and
the conditions of case (2) are not satisfied.
\end{enumerate}
\end{corollary}

The proof of this corollary follows directly from
Corollary \ref{newbad}, using the fact that when $s_l(w)=0$
then $A=L_{N_A(w)}=B_{N_A(w)-1}(N_A(w))$, and is left
to the reader.

Using these cases, we will choose $c(e_1(w))$ to accomplish the
following:

\begin{itemize}
\item If $s_r(w)>0$, then $D(w)$ is not the minimal
element $R_{N_D(w)}$ relative to $<_r$; in this case $c(e_1(w))$ is chosen so
that either $N(z_b)<N(w)$, or
$N(z_b)=N(w)$, $N_D(z_b)=N_D(w)$ and $D(z_b)<_r D(w)$ (see part (1) of the definition below).

\medskip

\item If $s_r(w)=0$, but $A(w)$ is not the minimal tree relative
to $<_l^{j(w)}$, $c(e_1(w))$ is chosen (in parts (2c) and (3)) so
that either $N(z_b)<N(w)$, or $N(z_b)=N(w)$, $N_A(z_b)=N_A(w)$ and $A(z_b)<_l^{j(w)}A(w)$.

\medskip

\item Finally, if both $A(w)$ and $D(w)$ are minimal, then 
$c(e_1(w))$ is chosen (in parts (2a) and (2b)) so that caret
$j(w)$ is removed in moving around the 2-cell from $w$ to $z_b$, so
$N(z_b) < N(w)$.

\end{itemize}

\begin{definition}\label{def:collapse}
We define a map $c$ from the set of bad edges to the set of
$2$-cells in several cases. Consider a bad edge $e_1(w)$, and let
$k=N_A(w)$. Let $T_1$ be the left subtree of the root of $D(w)$,
and let $T_2$ be the left subtree of the right child of the root
of $D(w)$. Similarly, let $S_1$ be the right subtree of the root
caret of $A(w)$, and let $S_2$ be the right subtree of the left
child of the root caret of $A(w)$.

\begin{enumerate}
\item If $s_r(w)>0$ and:

\begin{itemize}
\item If $s_r(w)$ is odd, and $T_1$ is empty, then define $c(e_1(w)):=Rr_1(w)$.
\item If $s_r(w)$ is odd, and $T_1$ is not empty, let $c(e_1(w)):=Rr_1^{-1}(w)$.
\item If $s_r(w)$ is even, and $T_2$ is empty, let $c(e_1(w)):=Rr_2(w)$.
\item If $s_r(w)$ is even, and $T_2$ is not empty, let $c(e_1(w)):=Rr_2^{-1}(w)$.
\end{itemize}

\medskip

\item If $s_r(w)=0$, $s_l(w) \in \{0,1\}$, property $(\ddagger)$ holds,
$k \ge 2$, and:
\begin{enumerate}
\item If $2 \leq j(w) \leq k-1$ and $A(w)=B_{j(w)}(k)$, then let $c(e_1(w)):=Rl_2(w)$.
\item If $j(w)=k$ and $A(w)=B_{k-1}(k)$, then let $c(e_1(w)):=Rl_1(w)$.
\item If $2 \leq j(w) \leq k-2$ and $A(w)=B_i(k)$ for $j(w)+1 \leq i \leq k-1$,
  let $c(e_1(w)):=Rl_1(w)$.
\end{enumerate}

\medskip

\item If $s_r(w)=0$, $s_l(A)>0$, property $(\ddagger)$ holds, and
the conditions of case (2) are not satisfied, and:
\begin{itemize}
\item If $s_l(w)$ is odd, and $S_1$ is empty, then let $c(e_1(w)):=Rl_1(w)$.
\item If $s_l(w)$ is odd, and $S_1$ is not empty, let $c(e_1(w)):=Rl_1^{-1}(w)$.
\item If $s_l(w)$ is even, and $S_2$ is empty, let $c(e_1(w)):=Rl_2(w)$.
\item If $s_l(w)$ is even, and $S_2$ is not empty, let $c(e_1(w)):=Rl_2^{-1}(w)$.
\end{itemize}

\end{enumerate}
\end{definition}

See Figures \ref{fig:cell1} and \ref{fig:cell2} for examples of
bad edges and their corresponding two cells. Figures
\ref{fig:trees1} and \ref{fig:trees2} show the tree pair diagrams
corresponding to the elements $w$ and $z_b$, where $e_1(z_b)$ is
the edge across the two-cell from the bad edge $e_1(w)$.

 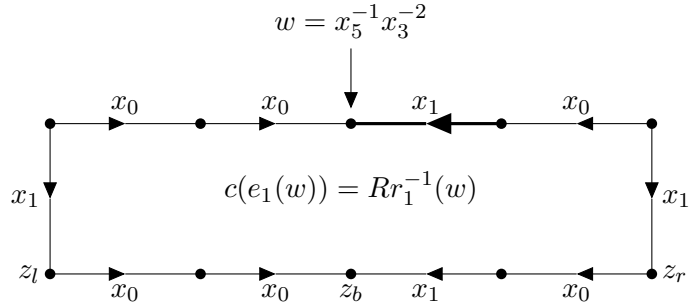
\begin{figure}[ht!]
\begin{center}
\begin{tikzpicture}[>=triangle 45]
\draw[->] (0,0)--(1,0); \draw (1,0)--(2,0); \draw[->]
(2,0)--(3,0); \draw (3,0)--(4,0); \draw (4,0)--(5,0); \draw[<-]
(5,0)--(6,0); \draw (6,0)--(7,0); \draw[<-] (7,0)--(8,0);

\fill (0,0) circle (2pt); \fill (2,0) circle (2pt); \fill (4,0)
circle (2pt); \fill (6,0) circle (2pt); \fill (8,0) circle (2pt);

\draw[->] (0,2)--(1,2); \draw (1,2)--(2,2); \draw[->]
(2,2)--(3,2); \draw (3,2)--(4,2); \draw[very thick] (4,2)--(5,2);
\draw[<-, very thick] (5,2)--(6,2); \draw (6,2)--(7,2); \draw[<-]
(7,2)--(8,2); \fill (0,2) circle (2pt); \fill (2,2) circle (2pt);
\fill (4,2) circle (2pt); \fill (6,2) circle (2pt); \fill (8,2)
circle (2pt); \draw (0,0)--(0,1); \draw[<-](0,1)--(0,2); \draw
(8,0)--(8,1); \draw[<-] (8,1)--(8,2);

\draw (1,0) node [anchor=north]{$x_0$}; \draw (3,0) node
[anchor=north]{$x_0$}; \draw (5,0) node [anchor=north]{$x_1$};
\draw (7,0) node [anchor=north]{$x_0$}; \draw (1,2) node
[anchor=south]{$x_0$}; \draw (3,2) node [anchor=south]{$x_0$};
\draw (5,2) node [anchor=south]{$x_1$}; \draw (7,2) node
[anchor=south]{$x_0$}; \draw (0,1) node [anchor=east]{$x_1$};
\draw (8,1) node [anchor=west]{$x_1$}; \draw (4,1.5) node
[anchor=north] {\textbf{$c(e_1(w))=Rr_1^{-1}(w)$}}; \draw (4,0)
node [anchor=north] {$z_b$}; \draw[<-] (4,2.2)--(4,3); \draw (4,3)
node[anchor=south]{$w=x_5^{-1}x_3^{-2}$}; \draw (0,0) node
[anchor=east] {$z_l$}; \draw (8,0) node [anchor=west] {$z_r$};
\end{tikzpicture}
\caption{The 2-cell corresponding to the bad edge
$e_1(x_5^{-1}x_3^{-2})$, where
$z_b=x_5^{-1}x_3^{-1}$.}\label{fig:cell1}
\end{center}
\end{figure}

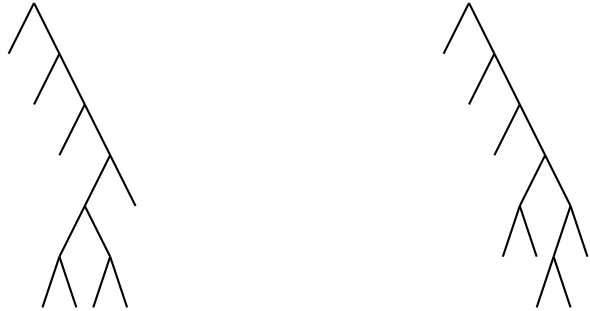
\begin{figure}[ht!]
\begin{center}
\begin{tikzpicture}[scale=.45,thick]
\tikzstyle{level 6}=[sibling distance=10mm] \coordinate child
child {child child {child child {child {child {child child}
child{child child}}child}}};
\end{tikzpicture}
\hspace{1.5in}
\begin{tikzpicture}[scale=.45,thick]
\tikzstyle{level 2}=[sibling distance=15mm] \tikzstyle{level
2}=[sibling distance=15mm] \tikzstyle{level 3}=[sibling
distance=15mm] \tikzstyle{level 4}=[sibling distance=15mm]
\tikzstyle{level 5}=[sibling distance=10mm] \tikzstyle{level
6}=[sibling distance=10mm] \coordinate child child {child child
{child child {child {child child} child {child {child child}
child}}} } ;
\end{tikzpicture}
\caption{The left (negative) trees from the pair diagrams corresponding to $w
= x_5^{-1}x_3^{-2}$ and $z_b=x_5^{-1}x_3^{-1}$. Notice that these
two trees differ by a rotation at the root caret of the subtree
$D(w)$.}\label{fig:trees1}
\end{center}
\end{figure}

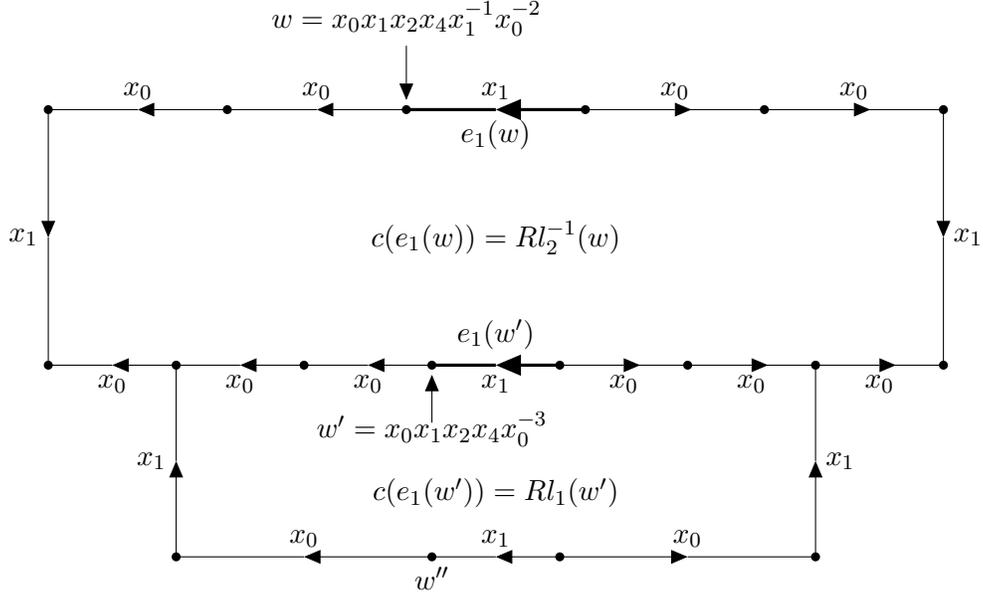
\begin{figure}[ht!]
\begin{center}
\begin{tikzpicture}[>=triangle 45,scale=.85]

\draw (0,0)--(1,0); \draw[<-] (1,0)--(2,0); \draw (2,0)--(3,0);
\draw[<-] (3,0)--(4,0); \draw (4,0)--(5,0); \draw[<-]
(5,0)--(6,0); \draw[very thick] (6,0)--(7,0); \draw[<-, very
thick] (7,0)--(8,0); \draw (8,0)--(9,0); \draw[>-] (9,0)--(10,0);
\draw (10,0)--(11,0); \draw[>-] (11,0)--(12,0); \draw
(12,0)--(13,0); \draw[>-] (13,0)--(14,0);

\fill (0,0) circle (2pt); \fill (2,0) circle (2pt); \fill (4,0)
circle (2pt); \fill (6,0) circle (2pt); \fill (8,0) circle (2pt);
\fill (10,0) circle (2pt); \fill (12,0) circle (2pt); \fill (14,0)
circle (2pt);

\draw (1,0) node [anchor=north]{$x_0$}; \draw (3,0) node
[anchor=north]{$x_0$}; \draw (5,0) node [anchor=north]{$x_0$};
\draw (7,0) node [anchor=north]{$x_1$}; \draw (9,0) node
[anchor=north]{$x_0$}; \draw (11,0) node [anchor=north]{$x_0$};
\draw (13,0) node [anchor=north]{$x_0$};

\draw (0,4)--(1.4,4); \draw[<-] (1.4,4)--(2.8,4); \draw
(2.8,4)--(4.2,4); \draw[<-] (4.2,4)--(5.6,4); \draw[very thick]
(5.6,4)--(7,4); \draw[<-, very thick] (7,4)--(8.4,4); \draw
(8.4,4)--(9.8,4); \draw[>-] (9.8,4)--(11.2,4); \draw
(11.2,4)--(12.6,4); \draw[>-] (12.6,4)--(14,4);

\fill (0,4) circle (2pt); \fill (2.8,4) circle (2pt); \fill
(5.6,4) circle (2pt); \fill (8.4,4) circle (2pt); \fill (11.2,4)
circle (2pt); \fill (14,4) circle (2pt); \draw (1.4,4) node
[anchor=south]{$x_0$}; \draw (4.2,4) node [anchor=south]{$x_0$};
\draw (7,4) node [anchor=south]{$x_1$}; \draw (9.8,4) node
[anchor=south]{$x_0$}; \draw (12.6,4) node [anchor=south]{$x_0$};

\draw (0,0)--(0,2); \draw[<-](0,2)--(0,4); \draw (14,0)--(14,2);
\draw[<-] (14,2)--(14,4); \draw (0,2) node [anchor=east]{$x_1$};
\draw (14,2) node [anchor=west]{$x_1$};

\draw (7,2) node  {\textbf{$c(e_1(w))=Rl_2^{-1}(w)$}}; \draw[<-]
(5.6,4.15)--(5.6,5); \draw (5.6,5)
node[anchor=south]{$w=x_0x_1x_2x_4x_1^{-1}x_0^{-2}$}; \draw
(7,3.6) node {$e_1(w)$};

\draw (2,0)--(2,-1.5); \draw[<-] (2,-1.5)--(2,-3); \draw
(12,0)--(12,-1.5); \draw[<-] (12,-1.5)--(12,-3); \draw (2,-1.5)
node [anchor=east] {$x_1$}; \draw (12,-1.5) node [anchor=west]
{$x_1$};

\draw (2,-3)--(4,-3); \draw[<-] (4,-3)--(6,-3); \draw
(6,-3)--(7,-3); \draw[<-] (7,-3)--(8,-3); \draw[->]
(8,-3)--(10,-3); \draw (10,-3)--(12,-3);

\fill (2,-3) circle (2pt); \fill (6,-3) circle (2pt); \fill (8,-3)
circle (2pt); \fill (12,-3) circle (2pt); \draw (6,-3) node
[anchor=north]{$w''$}; \draw (4,-3) node [anchor=south]{$x_0$};
\draw (10,-3) node [anchor=south]{$x_0$}; \draw (7,-3) node
[anchor=south]{$x_1$};

\draw (7,-2) node  {\textbf{$c(e_1(w'))=Rl_1(w')$}}; \draw[<-]
(6,-.15)--(6,-.9); \draw (6,-1) node {$w'=x_0x_1x_2x_4x_0^{-3}$};
\draw (7,.1) node [anchor=south]{$e_1(w')$};
\end{tikzpicture}
\caption{The 2-cells corresponding to the bad edges $e_1(w)$ and
$e_1(w')$, for $w=x_0x_1x_2x_4x_1^{-1}x_0^{-2}$ and
$w'=x_0x_1x_2x_4x_0^{-3}$. The edge across the bottom 2-cell from
$e_1(w')$ is $e_1(w'')$ where $w''=x_0x_1x_3x_0^{-2}$.}
\label{fig:cell2}
\end{center}
\end{figure}

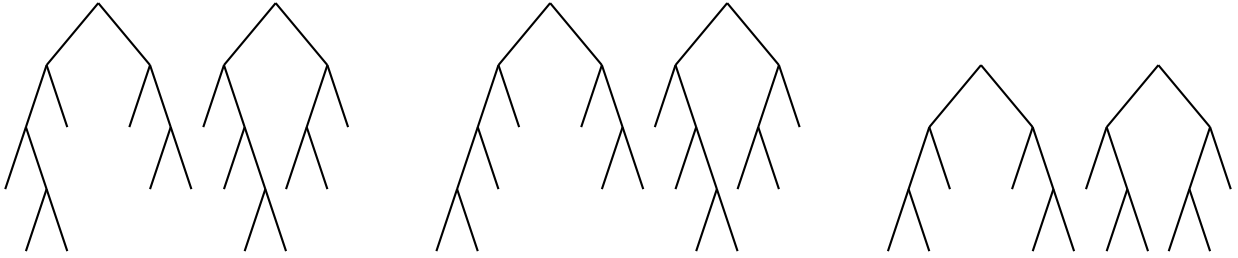
\begin{figure}[ht!]
\begin{center}
\begin{tikzpicture}[scale=.55,thick]
\tikzstyle{level 1}=[sibling distance=25mm] \tikzstyle{level
2}=[sibling distance=10mm] \coordinate child {child {child child
{child child}} child } child {child child {child child}};
\end{tikzpicture}
\begin{tikzpicture}[scale=.55,thick]
\tikzstyle{level 1}=[sibling distance=25mm] \tikzstyle{level
2}=[sibling distance=10mm] \tikzstyle{level 3}=[sibling
distance=10mm] \coordinate child {child child {child child {child
child}}} child {child {child child} child};
\end{tikzpicture}
\hspace{.35in}
\begin{tikzpicture}[scale=.55,thick]
\tikzstyle{level 1}=[sibling distance=25mm] \tikzstyle{level
2}=[sibling distance=10mm] \coordinate child {child {child {child
child} child} child } child {child child {child child}};
\end{tikzpicture}
\begin{tikzpicture}[scale=.55,thick]
\tikzstyle{level 1}=[sibling distance=25mm] \tikzstyle{level
2}=[sibling distance=10mm] \tikzstyle{level 3}=[sibling
distance=10mm] \coordinate child {child child {child child {child
child}}} child {child {child child} child};
\end{tikzpicture}
\hspace{.35in}
\begin{tikzpicture}[scale=.55,thick]
\tikzstyle{level 1}=[sibling distance=25mm] \tikzstyle{level
2}=[sibling distance=10mm] \coordinate child {child {child child}
child} child {child child {child child}};
\end{tikzpicture}
\begin{tikzpicture}[scale=.55,thick]
\tikzstyle{level 1}=[sibling distance=25mm] \tikzstyle{level
2}=[sibling distance=10mm] \tikzstyle{level 3}=[sibling
distance=10mm] \coordinate child {child child {child child}} child
{child {child child} child};
\end{tikzpicture}
\caption{The tree pair diagrams corresponding to $w =
x_0x_1x_2x_4x_1^{-1}x_0^{-2}$, $w'=x_0x_1x_2x_4x_0^{-3}$, and
$w''=x_0x_1x_3x_0^{-2}$ which are labeled in Figure
\ref{fig:cell2} above.} \label{fig:trees2}
\end{center}
\end{figure}

In the following theorem, we verify that the map defined above and
the partial order on the set of bad edges satisfy the hypothesis
of Theorem~\ref{combingextends}. In addition, we prove another
fact which will be used later in showing that the combing
satisfies a linear radial tameness function.

\begin{theorem}\label{thm:boundaryorder}
If $e_1(w)$ is a bad edge, then all other vertices $z$ on the
boundary of $c(e_1(w))$ have $N(z) \leq N(w)$. Furthermore,
 every edge of the form $e_1(z)$ along the boundary is either
a good edge, or precedes $e_1(w)$ in the ordering of the bad edges.
\end{theorem}

\begin{proof}
Let $e_1(z_b)$ be the bottom $e_1$ edge in the $2$-cell
$c(e_1(w))$, and $e_1(z_l)$
(respectively $e_1(z_r)$) be the
left (respectively right) side $e_1$ edges. The first statement in the theorem is a
consequence of the following observation. The tree $T_-(w)$ has
enough carets in the left subtree of the root caret, and in both
subtrees of the right child of the root caret to ensure that as we
read around $c(e_1(w))$ to the left, starting from $w$,
terminating at $z_b$, and form the successive products, no carets
ever need to be added to the tree pair diagrams in order to
perform these multiplications. The same holds for the path from
$wx_1^{-1}$, around to the right ending at $z_bx_1^{-1}$. Since
$N(wx_1^{-1}) \leq N(w)$, it follows that for each vertex $z$ of
$c(e_1(w))$, $N(z) \leq N(w)$. In addition, if $N(z)=N(w)$, then
$T_+(z)=T_+(w)$.

To prove the second statement of the theorem, we proceed by cases
according to the size of $s_r(w)$. In each case we show that
$e_1(z) < e_1(w)$, or else $e_1(z)$ is a good edge. We consider
separately the three subcases of $e_1(z)$ for $z \in \{z_b, z_l,
z_r \}$.
\begin{enumerate}
\item Case 1: $s_r(w) > 0$. In this case $c(e_1(w))=Rr_n^{\pm1}$
for $n\in \{1,2\}$. Also, note that $n(w)=N(w)-N_D(w)+C_r(w)$.

\medskip

\begin{enumerate}

\item $z=z_l$. In this case either:

\begin{itemize}
\item $N(z_l)<N(w)$ (and $e_1(z_l) < e_1(w)$ by (1) of
Definition~\ref{def:order} if $e_1(z_l)$ is a bad edge), or

\medskip

\item $N(z_l)=N(w)$, $s_r(z_l)>0$ , $N_D(z_l) < N_D(w)$ and $n(z_l)=n(w)$ (and $e_1(z_l) < e_1(w)$ by
(2a) of Definition~\ref{def:order} if $e_1(z_l)$ is a bad edge), or

\medskip

\item $N(z_l)=N(w)$ and $s_r(z)=0$. But one checks that if
$s_r(z)=0$, then $c(e_1(w))=Rr_n^{-1}$, $n \in \{1,2\}$, and
$s_r(w)=n$. But since no carets are ever added in moving from
$z_lx_1^{-1}$ to $z_l$, $e_1(z_l)$ is a good edge.
\end{itemize}

\medskip

\item $z=z_r$. If it is not the case that $N(z_r) < N(w)$, then
it is easily checked through the
definition of $Rr_n^{\pm 1}(w)$ that $T_-(z_l)$ and $T_-(z_r)$ differ only in the configuration of the
carets in the left subtree of the root. Therefore, the argument
for $e_1(z_l)$ goes through exactly, replacing $z_l$ by $z_r$.

\medskip

\item $z=z_b$. In this case either:

\begin{itemize}

\item $N(z_b)<N(w)$ (and $e_1(z_b) < e_1(w)$ by (1) of Definition~\ref{def:order} if
$e_1(z_b)$ is a bad edge), or

\medskip

\item $N(z_b)=N(w)$ and $s_r(z_b)>0$, in which case  $N_D(z_b)=N_D(w)$,
 and $D(z_b) <_r D(w)$. Then $C_r(z_b) \leq C_r(w)$, which implies
that $n(z_b)\leq n(w)$, (and
$e_1(z_b) < e_1(w)$ by (2b) of Definition~\ref{def:order} if
$e_1(z_b)$ is a bad edge), or

\medskip

\item $N(z_b)=N(w)$ and $s_r(z_b)=0$. However, this can happen only when $s_r(w)=n$
for $n \in \{1,2\}$,
$c(e_1(w))=Rr_n^{-1}(w)$, and $n(z_b) < (n_w)$ (and
$e_1(z_b) < e_1(w)$ by (4a) of Definition~\ref{def:order} if
$e_1(z_b)$ is a bad edge).

\end{itemize}

\end{enumerate}

\medskip

\item Case 2: $s_r(w)=0$. In this case, $c(e_1(w))=Rl_n^{\pm1}$
for $n\in \{0,1\}$. Also, note that $n(w)=N_A(w)+1$.
\begin{enumerate}

\medskip

\item $z=z_l$. In this case, $N_A(z_l) < N_A(w)$. Now either:

\begin{itemize}

\item $N(z_l)<N(w)$ (and $e_1(z_l) < e_1(w)$ by (1) of
Definition~\ref{def:order} if $e_1(z_l)$ is a bad edge), or

\medskip

\item $N(z_l)=N(w)$ and $s_r(z_l) =0$, and hence $n(z) \leq n(w)$ (and
$e_1(z_l) < e_1(w)$ by (3a) of Definition~\ref{def:order} if
$e_1(z_l)$ is a bad edge), or

\medskip

\item $N(z_l)=N(w)$ and $s_r(z_l)>0$. However, this only occurs if
$c(e_1(w))=Rl_2^{\pm 1}(w)$, and then $s_r(z_l)=1$ and $n(z_l)=n(w)$ (and $e_1(z_l) < e_1(w)$ by (4b) of
Definition~\ref{def:order} if $e_1(z_l)$ is a bad edge).
\end{itemize}

\medskip

\item $z=z_r$. If $e_1(z_r)$ is a bad edge, then $s_r(w) = 0$ implies that property $(\ddagger)$ holds.  In this case, $N(z_r)<N(w)$ because a caret is
removed when moving from $w$ to $wx_1^{-1}$.

\medskip

\item $z=z_b$. Then either:

\begin{itemize}

\item In cases (2a) and (2b) of Definition~\ref{def:collapse},
$N(z_b)<N(w)$, since caret $j(w)$ is removed in moving from $z_l$
to $z_lx_1^{-1}$ (and $e_1(z_b) < e_1(w)$ by (1) of
Definition~\ref{def:order} if $e_1(z_b)$ is a bad edge).

\medskip

\item In cases (2c) and (3) of Definition~\ref{def:collapse}, either $N(z_b)<N(w)$ (and $e_1(z_b) < e_1(w)$ by (1) of
Definition~\ref{def:order} if $e_1(z_b)$ is a bad edge), or
 $N(z_b)=N(w)$, $s_r(z_b)=0$, $N_A(z_b)=N_A(w)$, $n(z_b)=n(w)$, and
$A(z_b) <_l^{j(w)} A(w)$ (and $e_1(z_b) < e_1(w)$ by
(3b) of Definition~\ref{def:order} if $e_1(z_b)$ is a bad edge).

\end{itemize}

\end{enumerate}
\end{enumerate}
\end{proof}

Since all hypotheses of Theorem~\ref{combingextends} have now been
verified, Theorem~\ref{combingextends} shows that the nested traversal
0-combing $\Psi$
extends to a 1-combing $\Psi: X^1 \times [0,1] \ra X$.


\section{The combing of F satisfies a linear radial tameness function}\label{sec:tameness}

The fact that our combing $\Psi$ satisfies a linear radial tameness function
will follow from the fact that the number of carets in the tree pair diagrams representing the vertices along a
nested traversal normal form path never decreases, and from the close relationship
between word length over the alphabet $A=\{x_0^{\pm 1},x_1^{\pm 1}\}$ and
the number of carets.
First, we extend the concept of the number of carets in a tree pair diagram
from $F=X^0$ to all of $X$.

\begin{definition}
For any $x \in X$, we define $N_{\mx}(x)$ and $N_{\mn}(x)$ by cases.
\begin{enumerate}
\item If $x\in X^0$, then $x=g\in F$, and we let $N_{\mx}(x)=N_{\mn}(x)=N(g)$,
the number of carets in either tree of a reduced tree pair diagram for $g$.

\medskip

\item If $x\in X^1-X^0$, then $x$ is on the interior of some edge, with
vertices $g,h \in X^0$. Then define
$N_{\mx}(x)= \mx(N(g), N(h))$, and $N_{\mn}(x)= \mn(N(g), N(h))$.

\medskip

\item If $x\in X-X^1$, then $x$ is in the interior of some 2-cell, with
vertices $g_1, g_2,  \ldots, g_n$ along the boundary. Then we
define $N_{\mx}(x)=\mx(N(g_1),N(g_2), \dots, N(g_n))$, and
$N_{\mn}(x)=\mn(N(g_1),N(g_2), \dots, N(g_n))$.

\end{enumerate}
\end{definition}

The following lemma proves that using this expanded notion of the
number of carets of $x \in X$, the number of carets does not
decrease along the combing paths defined by $\Psi$.

\begin{lemma}\label{lem:maxmax}
For any $x \in X^1$ and $0 \leq s < t \leq 1$, we have
$N_{\mx}(\Psi(x,s)) \leq N_{\mx}(\Psi(x,t))$, where $\Psi$ is the
$1$-combing defined in Section \ref{sec:combing}.
\end{lemma}

\begin{proof}
In the case where $x \in X^0$, from Theorem~\ref{Nincreases} we know that 
along the nested traversal normal form $\eta(x)=a_1
a_2 \dots a_n$, we have $N(a_1 a_2 \cdots a_i) \leq N(a_1 a_2 \cdots
a_{i+1})$. For $x \in X^1 - X^0$, if $x$ is in the interior of a
good edge the conclusion of this lemma follows from the previous
sentence. If $x$ is in the interior of a bad edge $e$, then
the inequality follows from Noetherian induction and the fact that
for $y$ on any bad edge $e$ and $z$
on the complement of the edge $e$ in the closure of the 2-cell
$c(e)$, we have
$N_{\mx}( z) \leq N_{\mx}(y)$ as shown in Theorem \ref{thm:boundaryorder}.
\end{proof}

The next lemma relates the level of $x \in X$ to the quantities
$N_{\mn}(x)$ and $N_{\mx}(x)$.  Recall that when $x \in X^0$, the
level of $x$ and $l_A(x)$, the word length of $x$ with respect to
$A$, are identical. The lengths of the two relators in this
presentation are 10 and 14, so the constant $c$ used in defining
the level of a point in the interior of a 2-cell of the Cayley
complex for this presentation of $F$  will be $c=4(10)(14)+1$.

\begin{lemma}\label{lem:Nlev}
 For any $x \in X$ we have
$$N_{\mn}(x)-2 \leq \mathrm{lev}(x) < 4N_{\mx}(x)+1$$
and additionally $N_{\mx}(x) - N_{\mn}(x) \leq 9$.
\end{lemma}

\begin{proof}
When $x \in X^0$, Lemma \ref{lengthisN} gives $N(x)-2 \leq l_A(x)
\leq 4N(x)$. Therefore
$$ N_{\mn}(x)-2= N(x)-2 \leq \mathrm{lev}(x)=l_A(x) \leq
4N(x)=4N_{\mx}(x) < 4N_{\mx}(x)+1$$
for $x \in X^0$.  If, on the other hand, $x \in X^1 - X^0$, then
$x$ is in the interior of an edge, whose endpoints are $g, h \in
F$. Now $$\mathrm{lev}(x)= \frac{\mathrm{lev}(g)+\mathrm{lev}(h)}{2}
+ \frac{1}{4} \leq
\frac{4N(g)+4N(h)}{2}+\frac{1}{4} \leq 4 \mx(N(g), N(h)) +
\frac{1}{4}$$ $$ = 4N_{\mx}(x)+\frac{1}{4}< 4N_{\mx}(x)+1 .$$ But,
on the other hand,
$$\mathrm{lev}(x) \geq \frac{N(g)+N(h)}{2} -2 +\frac{1}{4} \geq
\mn(N(g),N(h))-2 + \frac{1}{4}= N_{\mn}(x) -2 + \frac{1}{4}
\geq N_{\mn}(x)-2.$$
In summary, in this case, we have
$$ N_{\mn}(x)-2  \leq \mathrm{lev}(x) \leq
4N_{\mx}(x)+\frac{1}{4}< 4N_{\mx}(x)+1. $$
And finally, if $x\in X-X^1$, $x$ is in the interior of some 2-cell,
with $g_1, g_2, \dots, g_k$ on the boundary, then
$$
\mathrm{lev}(x)= \frac{\mathrm{lev}(g_1)+ \cdots +
\mathrm{lev}(h)}{k} + \frac{1}{4}+\frac{1}{c}
\leq \frac{4(N(g_1)+ \cdots +N(g_k))}{k}+\frac{1}{4}+\frac{1}{c}
$$
$$
\leq 4 \mx(N(g_1),\dots, N(g_k)) + \frac{1}{4}+\frac{1}{c}
= 4N_{\mx}(x)+\frac{1}{4}+\frac{1}{c}< 4N_{\mx}(x)+1.
$$
On the other hand,
$$
\mathrm{lev}(x) \geq \frac{N(g_1)+ \cdots N(g_k)}{k} -2 +\frac{1}{4}
\geq  \mn(N(g_1), \dots ,N(g_k))-2 + \frac{1}{4}+ \frac{1}{c}
$$
$$
= N_{\mn}(x) -2 + \frac{1}{4}+\frac{1}{c} \geq N_{\mn}(x)-2.
$$
 And so, in this case, we have
$$
N_{\mn}(x)-2  \leq \mathrm{lev}(x) \leq
4N_{\mx}(x)+\frac{1}{4}+\frac{1}{c}<
4N_{\mx}(x)+1.
$$
This establishes the first statement of the lemma. Now if $x \in
X^0$, $N_{\mx}(x)=N_{\mn}(x)$. For $x \in X^1 - X^0$, $N_{\mx}(x)-
N_{\mn}(x) \leq 2$, since one either needs to add at most 1 caret
(or can cancel at most one caret) when multiplying by
$x_0^{\pm1}$, and one needs to add at most two carets (or can
cancel at most two carets) when multiplying by $x_1^{\pm1}$. Now
the relators in our presentation of $F$ have length either 10 or
14, and two vertices $v$ and $w$ on the boundary of a relator can
be at most seven edges apart. Furthermore, examining the relators,
we see that at most two of these seven edges correspond to
multiplication by $x_1^{\pm1}$. Therefore, for $x \in X-X^1$,
$N_{\mx}(x)- N_{\mn}(x) \leq 2(2)+5=9$.
\end{proof}

We are now able to prove that the combing $\Psi$ defined in Section \ref{sec:combing} satisfies
a linear radial tameness function.

\begin{theorem}\label{thm:flintame}
Thompson's group $F$ is in $TC_\rho$ with $\rho$ linear.  More
specifically, the Cayley complex of the presentation
$F=\langle x_0,x_1 |
[x_0x_1^{-1},x_0^{-1}x_1x_0],[x_0x_1^{-1},x_0^{-2}x_1x_0^2]
\rangle$ has a 1-combing admitting a radial tameness function of $\rho
(q)=4q+45$.
\end{theorem}

\begin{proof}
Let $\Psi:X^1 \times [0,1] \rightarrow X$ be the 1-combing of $F$
constructed in the previous section. Suppose that $x\in X^1$, $0
\leq s < t \leq 1$, and $\mathrm{lev}(\Psi(x,s)) > 4q+45$.
In Lemma \ref{lem:Nlev} we have  shown that
$\mathrm{lev}(\Psi(x,s)) < 4N_{\mx}( \Psi(x,s))+1$, which implies that
$4N_{\mx}( \Psi(x,s))
> 4q+44$, or $N_{\mx}( \Psi(x,s)) > q+11$.
From Lemma \ref{lem:maxmax} we have
$N_{\mx}(\Psi(x,t))
\geq N_{\mx}( \Psi(x,s))$, and so $N_{\mx}(\Psi(x,t)) > q+11$.
The last statement in
Lemma \ref{lem:Nlev} also shows that
$N_{\mx}(\Psi(x,t)) - N_{\mn}(\Psi(x,t)) \leq 9$, and so
$N_{\mn}(\Psi(x,t)) > q+2$.  Using Lemma \ref{lem:Nlev} once more,
we obtain $\mathrm{lev}(\Psi(x,t)) > q$.
\end{proof}


\section{Linear tame combing for $BS(1,\p)$}
\label{sec:bslinear}

In this section we prove the following.

\begin{theorem}\label{thm:bsp}
For every natural number $p \ge 3$,
the group $BS(1,\p)$ is in the class $TC_\rho$ with $\rho$ linear.
Moreover, the Cayley complex for the presentation
$BS(1,\p)=\langle a,t \mid tat^{-1}=a^{\p} \rangle$
has a 1-combing admitting a radial tameness function of
$\rho(q)=4(\h + 2)q+[4(\h + 2)+1](\h + 4)$ where $\h=\lfloor \frac{p}{2} \rfloor$.

\end{theorem}

Throughout this section, let $p \ge 3$, $G=BS(1,\p)$,
$\h=\lfloor \frac{p}{2} \rfloor$,
and $A=\{a^{\pm 1},t^{\pm 1}\}$.  Let
$X$ denote the Cayley 2-complex associated with the presentation
$\langle a,t \mid tat^{-1}=a^{\p} \rangle$ of $G$. Give the Cayley
graph $X^1=\Gamma(G,A)$ the path metric.
For a word $v$ over the alphabet $A$, let
$l(v)$ denote the length of the word $v$, and
let $l_\Gamma(v)$
denote the length of a geodesic (with respect to the
path metric on $\Gamma$) representative of the element of
$G$ represented by $v$.
Each element of $G$ has a
particularly simple, not necessarily geodesic, normal form
$t^{-m}a^jt^s$ with $m \geq 0$ and $s \geq 0$.
The combing we
construct will be based on this set of normal forms. The following
lemma, which characterizes a set of geodesics for elements of $G$
and relates them to the normal forms above, is a direct
consequence of Elder and Hermiller \cite[Prop.~2.3]{\murraysusan}.

\begin{lemma}\label{lem:bsgeod}
Let $g$ be an arbitrary element of $G$. Then $g$ is represented by
a geodesic word $w$ satisfying one of the following.
\begin{enumerate}

\item  $w = t^k a^{i_k} t^{-1} a^{i_{k-1}} t^{-1} \cdots t^{-1}
a^{i_{-m}} t^s$ with $0 \le s \le  m$, $0 < m$, $-m \le k$,
$|i_l| \le \h$ for $-m \le l \le k-1$, $|i_k| \le \h + 1$, and
either $1 \le |i_k|$ or
($k=-m$ and $s=0$).

\item  $w = t^{-m} a^{i_{-m}} t a^{i_{-m+1}} t \cdots t a^{i_k}
t^{s-m-k}$ with $0 \le m \le s$, $-m \le k$, $|i_l| \le \h$
for $-m \le l \le k-1$, $1 \le |i_k| \le \h + 1$, and either $1 \le |i_k|$ or $k=-m=0$.

\end{enumerate}
Moreover, in each case, $g$ also has a (not necessarily geodesic)
representative of the form $t^{-m}a^jt^s$ with $m \ge 0$, $s \ge
0$, and $j=i_{-m} + i_{-m+1} \p + \cdots + i_k \p^{k+m} \in \Z$.
\end{lemma}

The next two lemmas show that lower and upper bounds on the
length of a geodesic representative of $t^{-m}a^jt^s$ in the
Cayley graph imply lower
and upper bounds, respectively, on the value of $|j|$.

\begin{lemma}\label{lem:levimpliesh}
If $0 \le m < n$, $0 \le s < n$, $\h + 2 < B$, and
$l_\Gamma(t^{-m} a^j t^s) > Bn$, then
$|j| > \p^{({\frac{1}{\h + 2}}B-2)n}$.
\end{lemma}

\begin{proof}  We will prove the contrapositive;
suppose that $0 \le m < n$, $0 \le s < n$, $\h + 2 < B$, and
$|j| \le \p^{({\frac{1}{\h + 2}}B-2)n}$. Let $w$ be a word in one of
the forms (1)-(2) from Lemma \ref{lem:bsgeod} that is a geodesic
representative of the element of $G$ that is also represented by
$t^{-m} a^j t^s$.
As $w$ is a geodesic, it follows that $l_\Gamma(t^{-m} a^j t^s)$ is
simply the length $l(w)$ of the word $w$.

First note that if $i_k = 0$, then $j=0$ and either $s=0$ or $m=0$.  In
both of these instances, we have $l_\Gamma(t^{-m} a^j t^s) < n < B n$.

For the rest of the proof we suppose that $|i_k| \ge 1$. In both cases
(1)-(2), we have $|j| = |i_{-m} + i_{-m+1} \p + \cdots + i_k
\p^{k+m}| \le \p^{({\frac{1}{\h + 2}}B-2)n}$, and hence $|i_k|
\p^{k+m} - |\sum_{l=-m}^{k-1} i_l \p^{l+m}| \le
\p^{({\frac{1}{\h + 2}}B-2)n}$. Since each $|i_l| \le \h$ for $-m
\le l \le k-1$, then $|\sum_{l=-m}^{k-1} i_l \p^{l+m}| \le
\sum_{l=-m}^{k-1} \h \p^{l+m} = \h { \frac{\p^{k+m} - 1}{\p -
1}} < {\frac{2}{3}} \p^{k+m}$, where the
last inequality uses the hypothesis that $p \ge 3$. Plugging this into the previous
inequality, and using the fact that $|i_k| \ge 1$, gives
${\frac{1}{3}} \p^{k+m} \le |i_k| \p^{k+m} - {\frac{2}{3}}
\p^{k+m} < \p^{({\frac{1}{\h + 2}}B-2)n}$. Then
$\p^{k+m-({\frac{1}{\h + 2}}B-2)n} < 3$, and so $k+m -
({\frac{1}{\h + 2}}B-2)n \le 0$.  Since $0 \le m$, this gives $k \le
({\frac{1}{\h + 2}}B-2)n$.

If $w$ is of the form in (1) with $k > 0$,
then
\begin{eqnarray*}
l(w) & = & 2k + m + s + |i_{-m}| + \cdots + |i_k|
< 2\Big({\frac{1}{\h + 2}}B-2\Big)n + n + n + \h (k -1 + m) + \h + 1 \\
& < & 2\Big({\frac{1}{\h + 2}}B\Big)n  + \h \Big(\Big({\frac{1}{\h + 2}}B-2\Big)n + n\Big)
< Bn .
\end{eqnarray*}
If $w$ is of the form in (1) with $k \le 0$, then
\begin{eqnarray*}
l(w) & = & m + s + |i_{-m}| + \cdots + |i_k| < n + n + \h (m-1) + \h + 1 < (\h + 2)n + 1.
\end{eqnarray*}
For $w$ of the form (2) with $k>s-m$, we have
\begin{eqnarray*}
l(w) & = & 2m + k + (k-s+m) + |i_{-m}| + \cdots + |i_k| \\
& < & 2n + 2\Big({\frac{1}{\h + 2}}B-2\Big)n + \h(m+k-1) + \h + 1 \\
& \le & 2\Big({\frac{1}{\h + 2}}B\Big)n + \h\Big(n+\Big({\frac{1}{\h + 2}}B-2\Big)n\Big)
 < Bn .
\end{eqnarray*}
And finally, for $w$ in form (2) with $k \le s-m$, we have
\begin{eqnarray*}
l(w) & = & m + s + |i_{-m}| + \cdots + |i_k|
< n + n + \h(k-1+m) + \h + 1 < 2n+\h s + 1 \\
& < & 2n+\h n+1 = (\h + 2)n+1.
\end{eqnarray*}
Hence, in all possible cases, $l_\Gamma(t^{-m} a^j t^s) = l(w) < Bn+1$, and so
this nonnegative integer satisfies $l_\Gamma(t^{-m} a^j t^s) \le Bn$.
\end{proof}

\begin{lemma}\label{lem:himplieslev}
If $0 \le m < n$, $0 \le s < n$, $1 < E$, and $|j| > \p^{E n}$,
then $l_\Gamma(t^{-m} a^j t^s) > (E - 1) n$.
\end{lemma}

\begin{proof}
Suppose that $0 \le m < n$, $0 \le s < n$, $1 < E$, and $|j| >
\p^{E n}$. Let $w$ be a word in one of the forms (1)-(2) from
Lemma \ref{lem:bsgeod} that is a geodesic representative of the
element of $G$ also represented by $t^{-m} a^j t^s$.

In both cases, we have $\p^{E n} < |j| = |i_{-m} + i_{-m+1}
\p + \cdots + i_k \p^{k+m}|$, and in particular we must have $|i_k| \ge 1$.
Using the fact that
$|\sum_{l=-m}^{k-1} i_l \p^{l+m}| <  {\frac{2}{3}} \p^{k+m}$
(see the proof of Lemma \ref{lem:levimpliesh}) and the inequality
$|i_k| \le \h + 1$ yields $\p^{E n} < {\frac{2}{3}} \p^{k+m} +
|i_k| \p^{k+m} < (\h + 2) \p^{k+m}$. Since $p \ge 3$,
this gives $\p^{En-k-m}<\h + 2 \le p$, and
so $En-k-m \le 0$.  Then $(E-1)n < En-m \le k$.

Note that the inequality $(E-1)n < k$ implies that $0<k$.
We again consider the length of $w$ in each case.

If $w$ is of the
form in (1), then
\begin{eqnarray*}
l(w) & = & 2k + m + s + |i_{-m}| + \cdots + |i_k|
> 2(E-1)n + 0 + 0 + 1.
\end{eqnarray*}
For $w$ of the form (2) with $k>s-m$, we have
\begin{eqnarray*}
l(w) & = & 2m + k + (k-s+m) + |i_{-m}| + \cdots + |i_k|
> 0 + (E-1)n + 0 + 1.
\end{eqnarray*}
And finally, for $w$ in form (2) with $k \le s-m$, we have $k \le s$ and hence
\begin{eqnarray*}
l(w) & = & m + s + |i_{-m}| + \cdots + |i_k|
> 0 + k + (s-k) + 1 > (E-1)n+0+1.
\end{eqnarray*}
Thus in all possible cases we have $l_\Gamma(t^{-m}a^jt^s)=l(w)>(E-1)n$.
\end{proof}

The Cayley complex $X$ can be constructed using rectangles
homeomorphic to $[0,1] \times [0,1]$, with the top labeled $a$ and
oriented to the right, the bottom labeled $a^{\p}$ and also
oriented to the right, and the left and right sides labeled $t$
and oriented upward. Gluing these rectangles along commonly
labeled and oriented sides, the Cayley complex $X$ is homeomorphic
to the product $\R \times T$ of the real line with a tree $T$. The
projection maps $\pi_\R:X \ra \R$ and $\pi_T:X \ra T$ are
continuous, and we can write a point $x \in X$ uniquely as
$[\pi_\R(x),\pi_T(x)]$.

The vertices of $T$ are the projections via $\pi_T$ of the
vertices of $X$.  Two vertices of $X$ project to the same vertex of
$T$ if and only if there is a path in $X^1$ labeled by a power of $a$ between
the two vertices.  Each edge of $T$ can be considered as oriented
upward with a label $t$, the projection under $\pi_T$ of edges
labeled by $t$ in the Cayley complex.  Each vertex of $T$ is the
initial vertex for $\p$ edges and the terminal vertex for one edge.

The projection $\pi_\R$ maps a vertex $t^{-m}a^jt^s$ to the real
number $j \p^{-m}$.  The points on a vertical edge between vertices
$t^{-m}a^jt^s$ and $t^{-m}a^jt^{s+1}$  also all map under $\pi_\R$
to $j \p^{-m}$, and the projection $\pi_\R$ maps the horizontal edge from
$t^{-m}a^jt^s$ to $t^{-m}a^jt^sa$ homeomorphically to the interval
from $j \p^{-m}$ to $j \p^{-m} + \p^{-m+s}$.

On a rectangular ($[0,1] \times [0,1]$) 2-cell, the top left and
top right vertices have the form $[j\p^{-m},z]$ and
$[j\p^{-m}+\p^{-m+s},z]$, respectively. Two points
$x=[j\p^{-m},\pi_T(x)]$ and
$y=[j\p^{-m}+\p^{-m+s},\pi_T(y)]$ on the left and right sides
of this 2-cell, respectively, determine a horizontal line segment
if $\pi_T(x)=\pi_T(y)$ is a point on the unique edge in the tree
$T$ oriented toward the vertex $z$. The projection $\pi_R$ maps
this horizontal line segment homeomorphically to the interval from
$j\p^{-m}$ to $j\p^{-m}+\p^{-m+s}$ in $\R$, and the
projection $\pi_T$ is constant on this segment. Let $z'$ be the
initial vertex of the edge in $T$ whose terminus is $z$, and let
$r$ be any real number in the interval from $j\p^{-m}$ to
$j\p^{-m}+\p^{-m+s}$. The two points $[r,z']$ and
$[r,z]$ are on the bottom and top sides of this 2-cell,
respectively, and they determine a vertical line segment in the
2-cell which maps via $\pi_\R$ constantly to $r$, and which
maps via $\pi_T$ homeomorphically to the edge from $z'$ to $z$.

It will frequently be useful to move from points in the interiors
of 1-cells or 2-cells to vertices in the Cayley complex.  If $y$
is a vertex in $X$, let $\tilde y:=y$. If $y$ is in the interior
of a 1-cell in $X$ labeled $t$ directed upward, then let $\tilde
y$ be the initial vertex of that edge. If $y$ is in the interior
of a 1-cell in $X$ labeled $a$ directed right, then let $\tilde y$
be the endpoint of that edge whose image under $\pi_\R$ has the
maximum absolute value. Finally if $y$ is in the interior of a
2-cell, let $\tilde y$ be the bottom (left or right) corner of
that rectangular 2-cell whose image under $\pi_\R$ has the maximum
absolute value. In a 2-cell there are $p+3$ vertices, so the
difference in levels of vertices in that cell is at most \h + 2,
resulting in a bound on the difference between the level of the
2-cell and the level of any vertex in that cell, as well. Then in
all cases, we have $\tilde y \in X^0$ and $|\mathrm{lev}(\tilde y) -
\mathrm{lev}(y)| < \h + 3$.  We will call $\tilde y$ the {\it vertex associated
to} $y$.

More information on Cayley complexes for Baumslag-Solitar groups
can be found in \cite{\echlpt} or \cite{\murraysusan}.

Next, we apply the lemmas above to prove Theorem \ref{thm:bsp}.

\begin{proof}
We first define a 1-combing $\Psi:X^1 \times [0,1] \ra X$ as
follows.

Let $x$ be an arbitrary point in $X^1$.  Since $T$ is a tree,
there is a unique geodesic in $T$ from $\pi_T(\epsilon)$ to $\pi_T(x)$.
This geodesic first follows a (possibly empty) edge path in the
direction opposite to each edge orientation from $\pi_T(\epsilon)$ down
to a point $z(x)$ (which we will call the {\it nadir} of $x$), and
then follows a (possibly empty) edge path in the same direction as
each edge orientation up to $\pi_T(x)$. If $z(x) \neq \pi_T(x)$ so
that the upward portion is nonconstant, then the nadir $z(x)$ must
be a vertex of $T$. Let the path $p_x:[0,\frac{1}{3}] \ra T$
follow the geodesic from $\pi_T(\epsilon)$ to $z(x)$ with constant speed,
and let the path $q_x:[\frac{2}{3},1] \ra T$ follow the geodesic
from $z(x)$ to $\pi_T(x)$ with constant speed (with respect to the
path metric on $T$).

Define the path $\Psi:\{x\} \times [0,1] \ra X$ by
$\Psi(x,u)=[0,p_x(u)]$ for $u \in [0,\frac{1}{3}]$,
$\Psi(x,u)=[3(u-\frac{1}{3})\pi_\R(x),z(x)]$ for $u \in
[\frac{1}{3},\frac{2}{3}]$, and $\Psi(x,u)=[\pi_\R(x),q_x(u)]$ for
$u \in [\frac{2}{3},1]$. In the first third of the interval this
path
goes directly downward, in the second third it travels
horizontally, and in the last third it goes directly upward in the
Cayley complex $X$. Note that some of these three component paths
may be constant. We will refer to this path as the {\it DHU-path}
for $x$.

Continuity of the function $\Psi:X^1 \times [0,1] \ra X$ defined
by these DHU-paths follows from the continuity of the two
projection functions $\pi_\R$ and $\pi_T$.  For a vertex $x \in
X^0$ regarded as an element of $G$, the representative
$t^{-m}a^jt^s$ of $x$ from Lemma \ref{lem:bsgeod} satisfies
$z(x)=\pi_T(t^{-m})$, and so the path $\Psi:\{x\} \times [0,1] \ra
X^1$ follows the edge path labeled by the word $t^{-m}a^jt^s$ and
remains in the 1-skeleton of $X$. Hence $\Psi$ is a 1-combing,
which we will refer to as the {\it DHU-combing}.

In order to show that the DHU-combing satisfies a linear radial
tameness function, we will show that for the constants $B :=
4(\h + 2)$ and $C := (\h + 4)(B+1)$, whenever $x \in X^1$, $0 \le b < c \le 1$,
$0 \le q \in \Q$, $\mathrm{lev}(\Psi(x,b)) > Bq+C$, and $\mathrm{lev}(\Psi(x,c)) \le
q$, we have a contradiction.

To that end, fix a
point $x$ in $X^1$, $0 \le b < c \le 1$, and $0 \le q \in \Q$. Let
$v:=\Psi(x,b)$, $w:=\Psi(x,c)$, $\sigma:=\Psi(x,\frac{1}{3})$, and
$\tau:=\Psi(x,\frac{2}{3})$, and assume that $\mathrm{lev}(v) > Bq+C$ and
$\mathrm{lev}(w) \le q$.

Case I.  {\it Suppose that $w \in \Psi(\{x\} \times [0,\frac{1}{3}])$.}
Then both $v$ and $w$ are points on the downward portion of the
DHU-path for $x$, on
the infinite ray labeled $t^{-\infty}$ going down from $\epsilon$ in $X$.
Now $t^{-m}$ is a geodesic in the Cayley graph for all $m \ge 0$,
and so traveling along a downward path, the level is a
nondecreasing function.  Then we have $Bq+C < \mathrm{lev}(v) \le \mathrm{lev}(w)
\le q$.  Hence we obtain the required contradiction in this case.

Case II.  {\it Suppose that $w \in \Psi(\{x\} \times
(\frac{1}{3},\frac{2}{3}]) \setminus \{\sigma\}$.} In this case the
DHU-combing path for $x$ has a nontrivial horizontal component and
$w$ is in its image.

If $v \in \Psi(\{x\} \times [0,\frac{1}{3}))$, then let
$v':=\sigma$ and $b':=\frac{1}{3}$; otherwise $v \in \Psi(\{x\}
\times [\frac{1}{3},c))$ and we let $v'=v$ and $b'=b$. Again
applying the fact that the level is a nondecreasing function on a
downward path from the identity $\epsilon$, we also have that $v'=\Psi(x,b')$ with
$\frac{1}{3} \le b' < c$ and $Bq+C < \mathrm{lev}(v) \le \mathrm{lev}(v')$. Now the
points $v'$ and $w$ are both on the horizontal portion of the
DHU-path for $x$, satisfying $\pi_T(v')=\pi_T(w)=z(x)$ and
$|\pi_\R(v')| < |\pi_\R(w)|$.

The geodesic in $T$ from $\pi_T(\epsilon)$ to $\pi_T(x)$ may not have an
upward component in Case II, and so the nadir $z(x)$ may not be a
vertex of $T$.  This implies that the points $v'$ and $w$ may not
be in $X^1$. Let $\tilde {v'}$ and $\tilde w$ be the vertices
associated to $v'$ and $w$, respectively.  Since
$\pi_T(v')=\pi_T(w)$ is on the $t^{-\infty}$ ray in $T$,  the
associated vertices project to a vertex
$\pi_T(\tilde{v'})=\pi_T(\tilde w)=\pi_T(t^{-m})$ for some integer
$m \ge 0$ on this ray. The construction of the associated vertices
implies that $\tilde{v'}$ is represented by a word $t^{-m}a^i$ and
$\tilde w$ is represented by a word $t^{-m}a^j$ with $0 \le m$ and
$|i|\p^{-m} =  |\pi_\R(\tilde{v'})| \le  |\pi_\R(\tilde w)| =
|j|\p^{-m}$.

We also have $|\mathrm{lev}(\tilde w)-\mathrm{lev}(w)|<\h + 3$, so
$l_\Gamma(t^{-m}a^j) = l_\Gamma(\tilde w) =
\mathrm{lev}(\tilde w) < q + \h + 3 \le \lfloor q \rfloor + \h+4$.
Define $n:=\lfloor q \rfloor +\h+4$.  Then $l_\Gamma(t^{-m}a^j) < n$ and as
a consequence $0 \le m <  n$ as well. Similarly since
$|\mathrm{lev}(\tilde{v'})-\mathrm{lev}(v')|<\h + 3$, we have
\begin{eqnarray*}
l_\Gamma(t^{-m}a^i) & = & l_\Gamma(\tilde{v'}) = \mathrm{lev}(\tilde{v'})
> Bq+C-(\h + 3) \\
& = & B (\lfloor q \rfloor + \h+4) + B(q-\lfloor q \rfloor) -(\h+4)B + C-(\h + 3)
> Bn.
\end{eqnarray*}
Now $B>\h + 2$, and so we may apply
Lemma~\ref{lem:levimpliesh} to $t^{-m}a^i$, yielding the
inequality $|i|>\p^{(\frac{1}{\h + 2}B-2)n}$.

Combining the inequalities at the ends of the last two paragraphs
together with the value $B=4(\h + 2)$
gives $|j|>\p^{2n}$.  Lemma~\ref{lem:himplieslev}
applied to the word $t^{-m}a^j$ with $E=2$ says that
$l_\Gamma(t^{-m}a^j)>(E-1)n=n$. However, from the previous paragraph
we have $l_\Gamma(t^{-m}a^j)<n$, a contradiction.

Case III.  {\it Suppose that $w \in \Psi(\{x\} \times (\frac{2}{3},1])
\setminus \{\tau\}$.} In this case the DHU-combing path for $x$ has
a nontrivial upward component, and $w$ is in its image.

As in Case II, define $v':=\sigma$ and $b':=\frac{1}{3}$ if $v \in
\Psi(\{x\} \times [0,\frac{1}{3}))$, and define $v':=v$ and
$b':=b$ if $v \in \Psi(\{x\} \times [\frac{1}{3},c))$. Then
$v'=\Psi(x,b')$ with $\frac{1}{3} \le b' < c$,
 $Bq+C < \mathrm{lev}(v) \le \mathrm{lev}(v')$, and $v'$ is either on the
horizontal or the upward portion of the DHU-path for $x$.

The geodesic in $T$ from $\pi_T(1)$ to $\pi_T(x)$ must have an
upward component in case III, and hence the nadir $z(x)$ is a
vertex of $T$.  Then $z(x)=\pi_T(t^{-m})$ for some integer $0 \le
m$.

The DHU-path for $x$ travels from $v'=\Psi(x,b')$ to $w=\Psi(x,c)$
either via a nontrivial upward path, or else through a horizontal
and then nonconstant upward path.  The DHU-paths for $v'$ and $w$
are reparameterizations of the portion of the DHU-path for $x$
traveling from $\epsilon$ to each endpoint, and so they have the same
nadir $z(x)=z(v')=z(w)=\pi_T(t^{-m})$.  Moreover, we have
$|\pi_\R(v')| \le |\pi_\R(w)|$, and there is an upward path in $T$
from $\pi_T(v')$ to $\pi_T(w)$.

Although the horizontal portion of the DHU-path for $x$ must stay
in the 1-skeleton of $X$ (since it projects to $\pi_T(t^{-m})$),
the upward portion of the DHU-path for $x$ may leave $X^1$, and so
$v'$ and $w$ may not be in $X^1$.  Let $\tilde{v'}$ and $\tilde w$
be the vertices associated to $v'$ and $w$, respectively.
It follows from the definition of associated vertices that these
vertices satisfy $z(\tilde{v'})=z(\tilde w)=\pi_T(t^{-m})$,
$|\pi_\R(\tilde{v'})| \le |\pi_\R(\tilde w)|$, and there is there
is a (possibly empty) upward path in $T$ from $\pi_T(\tilde{v'})$
to $\pi_T(\tilde w)$.

Using Lemma~\ref{lem:bsgeod}, the vertex $\tilde w$ is represented
by a word $t^{-m}a^jt^s$ and the vertex
$\tilde{v'}$ is represented by a word $t^{-m} a^i t^r$.
The relations between these associated vertices above imply that
$0 \le |i| \le |j|$ and 	
$0 \le r \le s$.

The definition of associated vertices implies that
$|\mathrm{lev}(\tilde w)-\mathrm{lev}(w)|<\h + 3$, and hence
$l_\Gamma(t^{-m}a^jt^s) =
l_\Gamma(\tilde w) = \mathrm{lev}(\tilde w) <
q + \h + 3 \le \lfloor q \rfloor +\h+4 =: n$ as in
case II. As a consequence we have both $0 \le m <  n$ and $0
< s < n$ as well.

Also, as in case II, the inequality
$|\mathrm{lev}(\tilde{v'})-\mathrm{lev}(v')|<\h + 3$
implies that $l_\Gamma(t^{-m}a^it^r) =
l_\Gamma(\tilde{v'}) = \mathrm{lev}(\tilde{v'}) > Bn$.
Combining inequalities from above, we
also have $r<n$.

The rest of the proof in this case is similar to that in Case II.
In particular, Lemma~\ref{lem:levimpliesh} applied to
$t^{-m}a^it^r$ yields the inequality
$|i|>\p^{(\frac{1}{\h + 2}B-2)n}=\p^{2n}$. Combining this with
the inequality $|i| \le |j|$ from above yields
$|j|>\p^{2n}$.  In turn, using
Lemma~\ref{lem:himplieslev} with the word $t^{-m}a^jt^s$ and $E=2$ shows
that $l_\Gamma(t^{-m}a^jt^s)>n$, contradicting the
inequality $l_\Gamma(t^{-m}a^jt^s)<n$ found above.

Having achieved a contradiction in each case, this shows that the
DHU-combing for the group $BS(1,\p)$ and generating set $\{a,t\}^{\pm
1}$ satisfies a radial tameness function $\rho:\Q \ra {\mathbb
R}_+$ for the linear function $\rho(q)=Bq+C$ with the constants
$B = 4(\h + 2)$ and $C = (\h + 4)(B+1)$.
\end{proof}


\section{Coefficients in linear tame combings}
\label{sec:bscoeff1}

In this section we show that the linear coefficient for a linear
tame combing can be bounded away from 1 for a specific generating set.

\begin{theorem}\label{thm:bscoeff1}
For every natural number $p \ge 8$,
the group $G = BS(1,\p)=\langle a,t \mid tat^{-1}=a^{\p} \rangle$
with the generating set $A = \{a^{\pm 1},t^{\pm 1}\}$ does not
admit a 1-combing with radial tameness function of the form
$\rho(q)=q+C$ for any constant $C$.
\end{theorem}

\begin{proof}
Let $p \ge 8$ and let $X$ be the Cayley complex of
the presentation $\langle a,t \mid tat^{-1}=a^{\p} \rangle$,
described in Section~\ref{sec:bslinear}.
Suppose to the contrary that $\Psi:X^1 \times [0,1] \ra X$ is a 1-combing with
radial tameness function $\rho(q)=q+C$. Replacing $C$ by any
larger constant results in another radial tameness function
satisfied by the 1-combing $\Psi$, so we may assume that $C$ is a
natural number larger than four.

Consider the word $t^{C} a t^{-C} a t^{C} a^{-1} t^{-C} a^{-1}$.
Since $t^{C} a t^{-C} = a^{\p^{C}}$ in the group, this word
labels a loop in the Cayley graph $X^1$ based at the vertex
corresponding to the identity $\epsilon$ of $G$. Let $Y \subset X^1$ be
the subcomplex  of points on the vertices and edges along this
loop.  The restriction $\Psi:Y \times [0,1] \rightarrow X$ of the
1-combing then defines a homotopy from the identity vertex to the
loop $Y$, and so the image $\Psi(Y \times [0,1])$ is the image of
a disk filling in the loop $Y$.

Since the Cayley complex $X$ is the product of the real line $\R$
with the tree $T$ (described in Section~\ref{sec:bslinear}), this
complex is aspherical. Then the image set $\Psi(Y \times [0,1])$
must include all of the points in the rectangle of points $z \in
X$ with projections $0 \le \pi_R(z) \le \p^{C}$ and $\pi_T(z)$
on the geodesic in $T$ from $\pi_T(1)$ to $\pi_T(t^{C})$; that is,
the rectangle in $X$ bounded by the loop labeled $t^{C} a t^{-C}
a^{-\p^{C}}$ based at $1$, including this boundary loop. (The
image $\Psi(Y \times [0,1])$ must also contain all of the points
in the rectangle of $X$ bounded by the loop $t^{C} a t^{-C}
a^{-\p^{C}}$ based at $a$.) In particular, the vertex
corresponding to the element $g \in G$ represented by the word
$a^{(\h-2)\frac{\p^{C}-1}{\p-1}}$, where $\h=\lfloor \frac{p}{2} \rfloor$
as before, is in $\Psi(Y \times [0,1])$. We
obtain two estimates for $l_\Gamma(g)$ which, taken together, contradict
our assumption that $C>4$.

First, we observe that the points in the set $Y$ all lie on the
(geodesic) paths $t^{C} a t^{-C} a $ or $a t^{C} a t^{-C}$
starting at $1$, and so the levels of all of the points in $Y$ are
at most $2C+2$. Then the level of every point in the image set
$\Psi(Y \times [0,1])$ must be at most $\rho(2C+2)=3C+2$.  Hence
$l_\Gamma(g) =\mathrm{lev}(g) \le 3C+2$.

On the other hand, note that $(\h-2)\frac{\p^{C}-1}{\p-1} =
(\h-2)+(\h-2)\p+\cdots+(\h-2)\p^{C-1}$. Thus the element $g =_G
a^{(\h-2)\frac{\p^{C}-1}{\p-1}}$ of $G=BS(1,\p)$ is also represented
by the word $v = (a^{(\h-2)}t)^{C-1}a^{(\h-2)}t^{-(C-1)}$. We claim that
the word $v$ is a geodesic.  First note that since $g$ is a
nontrivial power of the generator $a$, we have $m=s=0$ in
Lemma~\ref{lem:bsgeod}, and the geodesic word $w$ representing $g$
provided by the lemma is in the form (2), $w = a^{i_0} t a^{i_1} t
\cdots t a^{i_k} t^{-k}$ with $0 \le k$, $|i_l| \le \h$ for $0 \le
l \le k-1$, and $1 \le |i_k| \le \h+1$. We will show that in fact
the words $v$ and $w$ are the same. So far we have $w =_G v$;
that is, $a^{i_0} t a^{i_1} t \cdots t a^{i_k} t^{-k} =_G
(a^{(\h-2)}t)^{C-1}a^{(\h-2)}t^{-(C-1)} $, and hence $i_{0} + i_{1} \p +
\cdots + i_k \p^{k} = (\h-2)+(\h-2)\p + \cdots (\h-2)\p^{C-1}$.
If $k \ge C$, then
\begin{eqnarray*}
|i_k|\p^k & \le &
 \Big|\sum_{l=0}^{C-1} ((\h-2)-i_l)\p^l \Big| + \Big|\sum_{l=C}^{k-1} -i_l\p^l\Big|
 \le  \sum_{l=0}^{C-1} (\p-2)\p^l + \sum_{l=C}^{k-1} \h\p^l \\
& < & (\p-2) \frac{\p^{k}-1}{\p-1} < \p^k.
\end{eqnarray*}
Since $1 \le |i_k|$, this shows that we must have $k \le C-1$. If
$k \le C-2$, then
\begin{eqnarray*}
(\h-2)\p^{C-1} & \le &
 \Big|\sum_{l=0}^{k} (i_l-(\h-2))\p^l\Big| + \Big|\sum_{l=k+1}^{C-2} -(\h-2)\p^l\Big|
 \le \sum_{l=0}^{k} (\p-2)\p^l + \sum_{l=k+1}^{C-2} (\h-2)\p^l \\
& < & (\p-2) \frac{\p^{C-1}-1}{\p-1} < \p^{C-1},
\end{eqnarray*}
again resulting in a contradiction. Hence $k = C-1$.  Subtracting
once more, we get
\begin{eqnarray*}
|i_{C-1}-(\h-2)|\p^{C-1} & \le &
 \Big|\sum_{l=0}^{C-2} ((\h-2)-i_l)\p^l\Big|
\le  \sum_{l=0}^{C-2} (\p-2)\p^l < \p^{C-1},
\end{eqnarray*}
and so $i_{C-1}=\h-2$.  Using induction, then $i_l=\h-2$ for all $0
\le l \le C-1$.  Hence $w$ and $v$ are the same word, and the word
$v$ is a geodesic.

This gives us another way to compute the word length over $A$ of $g$, since
$v$ is a geodesic representative of $g$, and so
$l_\Gamma(g)=l(v)=((\h-2)+1)(C-1)+(\h-2)+C-1=\h C-2$. Earlier in this proof we had
$l_\Gamma(g) \le 3C+2$, which gives $\h C-2 \le 3C+2$.  The hypothesis
that $p \ge 8$ gives $\h \ge 4$, and so we have $C \le
\frac{4}{\h-3} \le 4$, contradicting our choice of $C$.
\end{proof}

\bibliographystyle{plain}

\end{document}